\documentclass[a4paper,11pt]{amsart}

\usepackage[ansinew]{inputenc}
\usepackage[T1]{fontenc}
\usepackage{amsmath}
\usepackage{amssymb}
\usepackage{oldgerm}
\usepackage[curve]{xypic}
\usepackage{amsthm}
\usepackage{hyperref}
\usepackage{graphicx}
\usepackage{tikz}
\usetikzlibrary{shapes,arrows}
\usepackage{enumitem}

\DeclareRobustCommand{\SkipTocEntry}[5]{}

\newtheorem{satz}{Theorem}[section]
\newtheorem{Theorem}[satz]{Theorem}
\newtheorem{Lemma}[satz]{Lemma}
\newtheorem*{Hauptsatz}{Theorem}
\newtheorem{Prop}[satz]{Proposition}
\newtheorem{example}[satz]{Example}
\newtheorem{Cor}[satz]{Corollary}

\newtheorem*{Conj*}{Conjecture}

\theoremstyle{definition}
\newtheorem{Definition}[satz]{Definition}
\newtheorem{Notation}[satz]{Notation}
\newtheorem{remark}[satz]{Remark}

\newcommand{\Z}{\mathbb{Z}}
\newcommand{\D}{\mathcal{D}}

\newcommand{\E}{\mathcal{E}}
\newcommand{\bs}{\backslash}
\newcommand{\variable}{\underline{\;\;}}

\providecommand{\Mf}[1]{\langle#1\rangle}
\DeclareMathOperator{\im}{im}
\DeclareMathOperator{\id}{id}
\DeclareMathOperator{\NF}{NF}

\DeclareMathOperator{\llcm}{l-lcm}
\DeclareMathOperator{\rlcm}{r-lcm}
\DeclareMathOperator{\lgcd}{l-gcd}
\DeclareMathOperator{\rgcd}{r-gcd}
\DeclareMathOperator{\QF}{QF}
\DeclareMathOperator{\sh}{sh}
\DeclareMathOperator{\hgt}{ht}
\newcommand{\wt}{\widetilde}
\newcommand{\sm}{\smallsetminus}

\providecommand{\abs}[1]{\lvert#1\rvert}
\providecommand{\norm}[1]{\lVert#1\rVert}

\newenvironment{pf}{\begin{proof}[Proof]}{\end{proof}}

\numberwithin{equation}{section}

\newcommand{\sms}[1]{\raisebox{0pt}[\height][0pt]{\ensuremath{\scriptstyle#1}}}

\tikzstyle{cloud} = [draw, rectangle, node distance=2cm, minimum size=6.5mm]

\begin{document}

\author{Alexander He\ss{}}
\author{Viktoriya Ozornova}
\address{ Fachbereich Mathematik, Universit\"at Bremen, 28359 Bremen }
\thanks{\texttt{ozornova@math.uni-bremen.de}}
\title{Factorability, String Rewriting and Discrete Morse Theory}

\thanks{\texttt{hess@math.uni-bonn.de}}

\begin{abstract}
This article deals with the notion of factorability. Elements of a factorable group or monoid possess a normal form, which leads to a small complex homotopy equivalent to its bar complex, thus computing its homology. We investigate the relations to string rewriting and to discrete Morse theory on the bar complex. Furthermore, we describe a connection between factorability and Garside theory.
\end{abstract}

\maketitle

\tableofcontents

\section{Introduction}
\label{Introduction}

This article investigates combinatorial properties of certain groups and monoids with a view towards their homology. 

One of the methods to compute the homology of a group is to consider its bar complex. However, the bar complex is very large and hard to deal with. While studying moduli spaces of Riemann surfaces, C.-F.~B\"odigheimer and B.~Visy discovered that symmetric groups have normal forms with some remarkable properties. Among other consequences, these normal forms lead to a chain complex, much smaller than the bar resolution, computing homology of the symmetric groups. C.-F.~B\"odigheimer and B.~Visy abstracted the properties needed for these normal forms and defined the notion of a factorable group (cf.\  \cite{Visy}, \cite{CFBVisy}). Later on, R.~Wang and the first author extended the definition of factorability to categories and monoids in \cite{Wang} and \cite{AlexThesis}.

A factorability structure depends not only on the group or monoid itself but also on the chosen generating system. The factorability structure consists of a factorization map which assigns to each element of the monoid a preferred generator that is split off in a geodesic, i.e., word-length-preserving way. This map is subject to several axioms, which in particular ensure a certain, non-obvious compatibility with the multiplication in the monoid. Such a factorability structure yields a choice of geodesic normal forms, i.e., minimal representatives of each element of the monoid in terms of the chosen generating system.

A consequence of the aforementioned axioms is the existence of a quite small complex computing the homology of those objects. This complex was introduced and first studied by B.~Visy, and he showed that it computes the homology in the case of symmetric groups in \cite{Visy}, \cite{CFBVisy}. R.~Wang has extended this result to groups with finite chosen generating system in \cite{Wang}. We present a proof in the more general case of monoids with arbitrary chosen generating systems (see also \cite{AlexThesis}).

In more detail, the first aim of this article is to investigate the rewriting system and the corresponding matching on the bar complex arising from a factorability structure. While the former turns out not to be noetherian in some cases, the latter always is, and the resulting chain complex can be identified with the one introduced by B.~Visy (cf.\ \cite{Visy}, \cite{CFBVisy}). The basis of this complex is described by the following theorem, generalizing the results by Visy and Wang (for the notion of stability, see Definition \ref{DefinitionStable}):

\begin{Hauptsatz}
 Let $(M, \E, \eta)$ be a factorable monoid. Then there is a chain complex $(\mathbb{V}_*, \partial^{\mathbb{V}})$ of free abelian groups with bases
\begin{eqnarray*}
\left\lbrace [m_n|\ldots |m_1]\,  |\, m_i \in \E, m_i\neq 1, (m_{i+1}, m_i) \mbox{ unstable}\right\rbrace
\end{eqnarray*}
in degree $n$, which computes the homology of the monoid $M$.
\end{Hauptsatz}

Moreover, in Proposition \ref{VisyDifferentialCoherent} and Theorem \ref{THM:differential Visy resolution}, we give an explicit description of the differentials in this complex. 

Counterexamples for noetherianity of the rewriting system exist even if this rewriting system is finite and the monoid described by it is right-cancellative. An example of such a monoid is given in the appendix.  However, there are several cases where the corresponding rewriting system is noetherian, as for example in the following theorem (this is Corollary \ref{RewritingSystemLLGaussch}):

\begin{Hauptsatz}
 Let $M$ be a right-cancellative, left-noetherian monoid in which any two elements admitting a common left-multiple also admit a least common left-multiple. Then the rewriting system associated as above with the factorability structure on $M$ is complete.
\end{Hauptsatz}

This statement generalizes the complete rewriting systems for Garside groups described in \cite{HermillerMeier}. 
\\

 A further aim of this article is to show that a large class of groups and monoids can be equipped with a factorability structure. Before we formulate our result, we report shortly on previously known examples. Among the few families of examples besides symmetric groups, the dihedral groups were known to carry an interesting factorability structure. The first-named author investigated another family of monoids closely related to Thompson's group $F$ in his thesis \cite{AlexThesis}. The second-named author and C.-F.~Bödigheimer showed that $O(V)$ carries a factorability structure with respect to the generating system of all reflections in \cite{MyThesis}. Moreover, in \cite{MyThesis} it is shown that the Coxeter groups of $B$-series admit a factorability structure with respect to all reflections. Whether the $D$-series and the exceptional irreducible finite Coxeter groups are factorable with respect to all reflections, is open.
 
 One should keep in mind that the existence of a factorability structure depends on the choice of the generating system. For example, the symmetric groups are factorable with all transpositions as a generating system, but they do not admit a factorability structure if we consider the generating system of simple transpositions. One obstruction to being a factorable monoid is a theorem due to M.~Rodenhausen \cite{Moritz} stating that the monoid has to admit with the given generating system a presentation with relations of length at most $4$.

The main result of this part provides factorability structures on a large class of monoids (see Theorem \ref{DefinitionFaktorabilitaetGauss}):

\begin{Hauptsatz} 
 Let $M$ be a a right-cancellative, left-noetherian monoid so that any two elements admitting a common left-multiple also admit a least common left-multiple. Let $\mathcal{E}$ be a generating subset of $M$ that is closed under least common left-multiple and left-complement. Then $(M, \mathcal{E})$ is factorable. 
\end{Hauptsatz}

This class of factorable monoids includes in particular the Artin monoids introduced by E.~Brieskorn and K.~Saito (\cite{BrieskornSaito}) while studying the corresponding Artin groups. Monoids as in this theorem were studied by P.~Dehornoy and Y.~Lafont under the name ``left locally Gaussian'' in \cite{DehornoyLafont}. Similar concepts have already been defined by P.~Dehornoy in \cite{DehornoyF} and by P.~Dehornoy and L.~Paris in \cite{DehornoyParis}. These concepts were developed to abstract and generalize the work by F.~Garside (\cite{Garside}), where he solves the word problem and the conjugation problem for braid groups. There are several further treatments of these and similar structures in the literature, mostly united by the term ``Garside theory'', e.g. in \cite{DDGKM}, \cite{GodelleParis}, \cite{DigneMichel}. 

For a certain subclass of monoids mentioned above, namely for the Garside monoids, it is possible to extend the factorability structure to the group of fractions of the monoid. In particular, one 
obtains a factorability structure on the 
braid groups and, more generally, on all Artin groups of finite type. 
\\

The structure of the article is as follows. In Section \ref{Factorability: Basic Definitions}, we report on basic definitions and existing results about factorability. In Section \ref{Local factorability}, we report on the alternative description of factorability by M.~Rodenhausen \cite{Moritz}. 

In Section \ref{Rewriting system basics}, we give a brief overview of rewriting systems basics needed in this article. In Section \ref{Rewriting System of a Factorable Monoid}, we describe the rewriting system arising from a factorability structure, which is always convergent. To show that it is noetherian in some cases, we have to prove a combinatorial lemma in Section \ref{Finiteness of Qn'}. Then the proof of the completeness of the corresponding rewriting systems is given in Section \ref{Factorable Monoids with Complete Rewriting Systems}. 

Next, the basics of discrete Morse theory used in this article are collected in Section \ref{sec:discrete morse theory}. After this, we define a matching on the bar complex of a factorable monoid, similar to the matching arising from the complete rewriting system on a monoid (cf.\ \cite{Brown}, \cite{Cohen}) and show that it is noetherian in Section \ref{Matching for Factorable Monoids}. We describe the resulting chain complex in Section \ref{Chain Complex for Homology of Factorable Monoids}. 

In Section \ref{Garside theory: Basics}, we first report on basic Garside theory. Then we explore the connection between factorability and Garside theory. In particular, we show in Section \ref{Factorability structure on locally Gaussian monoids} that each right-cancellative, left-noetherian monoid in which any two elements admitting a common left-multiple also admit a least common left-multiple is factorable. We describe the factorability structure more concretely for Artin monoids in Section \ref{Factorability Structure on Artin Monoids}. In Section \ref{Factorability structure on Garside groups}, we show that the factorability structures on Garside monoids can be extended to the corresponding groups of fractions. Finally, in the appendix, we exhibit an example where the associated rewriting system of a factorable monoid is not noetherian.

\addtocontents{toc}{\SkipTocEntry}
\subsection*{Acknowledgements}

This article arises from the theses of the authors written under the supervision of C.-F. B\"odigheimer. The authors wish to express their gratitude to him for his permanent encouragement and support. We would like to thank Lennart Meier and Felix Boes for reading preliminary versions of this article. The first author wishes to thank GRK 1150 Homotopy and Cohomology (University of Bonn) for financial support. The second author thanks GRK 1150 Homotopy and Cohomology (University of Bonn), the International Max Planck Research School on Moduli Spaces (MPI Bonn) and the SFB 647 Space-Time-Matter (Free University Berlin) for their financial support.

\section{Factorability: Basic Definitions}
\label{Factorability: Basic Definitions}
In this section, we are going to define factorability structures. We will collect some basic facts and notation. The idea of factorability is to provide a special sort of structure on groups (and later on monoids) which allows to get, starting with the bar complex, a much smaller complex for computing group homology (or corresponding analogues). The original definition of this structure is due to B.~Visy (\cite{Visy}) and C.-F.~Bödigheimer, and can be found in \cite{CFBVisy}, along with the main results on factorability. The definition was generalized to monoids by R.~Wang (\cite{Wang}) and the first author (\cite{AlexThesis}). 

Let $M$ be a monoid and $\mathcal{E}$ a generating set. By $\mathcal{E}^*$, we denote the free monoid generated by $\mathcal{E}$. Recall that since $\mathcal{E}$ is a generating system for $M$, there is a monoid homomorphism $\mathcal{E}^*\to M$ sending each word to the element it represents. Denote by $N_{\mathcal{E}}$ the word length with respect to $\mathcal{E}$, i.e., $N_{\mathcal{E}}(m)$ is the least length of a word in $\mathcal{E}^*$ which represents the given element $m\in M$. If there is no danger of confusion, we write $N=N_{\mathcal{E}}$ for short. We proceed with the definition of factorability. 

\begin{Definition}\label{FactDefMonoid}
Let $M$ be a monoid and $\mathcal{E}$ a generating set not containing $1$. We say, a map
 \begin{eqnarray*}
 \eta=(\overline{\eta}, \eta')\colon M\to M\times M
 \end{eqnarray*}
  is a \textbf{factorization map} if it satisfies the following three axioms:
  \begin{enumerate}
\item[(F1)] For all $m\in M$, we have $m=\overline{\eta}(m)\eta'(m)$.
\item[(F2)] For all $m \in M$, we have $N(m)=N(\overline{\eta}(m))+N(\eta'(m))$.
\item[(F3)] For all $m \in M\setminus \{1\}$, the element $\eta'(m)$ lies in $\mathcal{E}\setminus\{1\}$.
\end{enumerate}

   For $1\leq i\leq n-1$, we denote by $f_i$ the map $M^n \to M^{n}$ which assigns to $(x_n, \ldots, x_1) \in M^n$ the tuple $(x_n, \ldots, x_{i+2}, \eta(x_{i+1}x_i), x_{i-1}, \ldots, x_1)$. We call $\eta$ a \textbf{factorability structure} on $(M,\mathcal{E})$ if the three maps 
\begin{eqnarray*}
 f_1f_2f_1f_2, \quad f_2f_1f_2, \quad f_2f_1f_2f_1\colon M^3 \to M^3 
\end{eqnarray*}
are equal in the \textbf{graded sense}, i.e., for each tuple $(x_3, x_2, x_1)\in M^3$, the three maps agree or each of them lowers the sum of the norms of the entries. If $\eta$ is a factorability structure on $M$, we call the triple $(M, \mathcal{E}, \eta)$ a \textbf{factorable monoid}.

 We call the triple $(M, \mathcal{E}, \eta)$ \textbf{weakly factorable monoid} if $\eta$ is a factorization map, and, in addition, the following diagram commutes in the graded sense (same sense as above):
 
 \begin{equation*}\label{eqWF}
 \begin{gathered}
 \xymatrix{ 
M\times \mathcal{E}\ar[dd]^{\mu}\ar[r]^-{\eta\times \id} &M\times \widetilde{\mathcal{E}} \times \mathcal{E}\ar[d]^{\id \times \mu}&\\
&M\times M\ar[r]^-{\id\times \eta}&M\times M\times \widetilde{\mathcal{E}}\ar[d]^{\mu\times \id}\\ 
M\ar[rr]_{\eta}&  & M\times \widetilde{\mathcal{E}}
}
\end{gathered}
\tag{WF}
\end{equation*}
Here, we write $\widetilde{\mathcal{E}}=\mathcal{E}\cup\{1\}$. Moreover, $\mu$ denotes the multiplication in the monoid.

From now on, we write $\overline{x}=\overline{\eta}(x)$ and $x'=\eta'(x)$ for an element $x\in M$ and a factorization map $\eta$ whenever confusion is unlikely. 
\end{Definition}

The following remark is immediate.
\begin{remark}\label{EtaErzeuger}
 Let $M$ be a monoid with a chosen generating system $\E$, and let $\eta\colon M\to M \times M$ be a factorization map. Let $t$ be an element of $\E$. Then
 \begin{enumerate}
  \item $\eta(1)=(1,1)$,
  \item $\eta(t)=(1,t)$.
 \end{enumerate}
\end{remark}

The following notion turns out to be useful in our context.

\begin{Definition}
In a monoid $M$ with a chosen generating set $\E$, we call a pair $(a,b) \in M \times M$ \textbf{geodesic} if $N_{\E}(ab)=N_{\E}(a)+N_{\E}(b)$, where $N_{\E}$ denotes the word length with respect to $\E$.
\end{Definition}

\begin{remark} \label{EineImplikationFaktorabel}
\begin{enumerate}
 \item From a result of Visy (\cite{Visy}, Corollary 3.1.7), given a factorization map $\eta$ on a monoid $M$ with a chosen generating system $\mathcal{E}$ and a geodesic pair $(x,y)$, the pairs $(x', y)$ and $(\overline{x}, \overline{x'y})$ are automatically geodesic. Note that Visy only shows this for groups, but word-by-word the same argument works also for monoids. 
 \item Furthermore, we can also transfer Visy's proof that the diagram above commutes in graded sense for all of $M\times M$ (instead of only $M\times \mathcal{E}$, cf.\ \cite{Visy}, Proposition 3.1.8. Indeed, Visy even uses this stronger statement for the definition of factorability.)
\end{enumerate}
\end{remark}

We give a slight reformulation of the last condition \ref{eqWF} which is easier to handle:
\begin{Lemma}[\cite{Visy}] \label{FactAxiom}
Let $M$ be a monoid with a chosen generating system $\E$. For a factorization map $\eta\colon M\to M\times M$, the condition \ref{eqWF} for a pair $(x,t) \in M\times \E$ is equivalent to: If both $(\eta'(x), t)$ and $(\overline{\eta}(x), \overline{\eta}(\eta'(x)t))$ are geodesic pairs, then $(x,t)$ is a geodesic pair and the equalities $\eta'(xt)=\eta'(\eta'(x)t)$ and $\overline{\eta}(x)\overline{\eta}(\eta'(x)t)=\overline{\eta}(xt)$ hold. In groups or, more generally, right-cancellative monoids, the last equation holds automatically by cancellation. 

Note furthermore that if $\eta'(ab)=\eta'(\eta'(a)b)$ and $\overline{\eta}(a)\overline{\eta}(\eta'(a)b)=\overline{\eta}(ab)$ hold for some pair $(a,b)\in M\times \E$, the norm condition for this pair is automatically satisfied. 
\end{Lemma}

\begin{remark}

\begin{enumerate}
\item Note that the original definition of factorability for groups (\cite{CFBVisy}), directly transferred to monoids, is exactly the weakly factorability property. However, it turns out in general not to be the right definition for our purposes, in particular, in order to obtain a small chain complex computing the monoid homology. See also \cite{StefanMehner}.  
 \item To illustrate the situation, we will depict $f_i$ by the following diagram: 
 \end{enumerate}
 
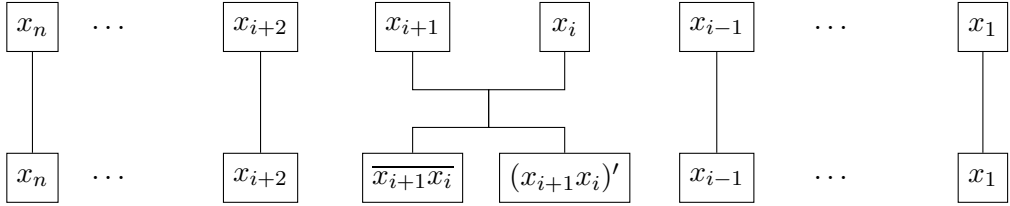
\begin{center}
 \begin{tikzpicture}
  \node[cloud](a1){$x_n$};
 \node[right of=a1](b1){$\ldots$};
 \node[cloud, right of=b1](c1){$x_{i+2}$};
  \node[cloud, right of=c1](d1){$x_{i+1}$};
 \node[cloud, right of=d1](e1){$x_{i}$};
 \node[cloud, right of=e1](f1){$x_{i-1}$};
 \node[right of=f1, xshift=0.5cm](g1){$\ldots$};
 \node[cloud, right of=g1](h1){$x_1$};
  \node[cloud, below of=a1](a2){$x_n$};
 \node[right of=a2](b2){$\ldots$};
 \node[cloud, right of=b2](c2){$x_{i+2}$};
  \node[cloud, right of=c2](d2){$\overline{x_{i+1}x_i}$};
 \node[cloud, right of=d2](e2){$(x_{i+1}x_{i})'$};
 \node[cloud, right of=e2](f2){$x_{i-1}$};
 \node[right of=f2, xshift=0.5cm](g2){$\ldots$};
 \node[cloud, right of=g2](h2){$x_1$};

 \draw (a1)--(a2);
  \draw (c1)--(c2);
\draw (d1.south) -- ++(0.0, -0.5) -| ++(1.0, -0.5) -- ++(0.0,-0.0)-- ++(-1.0, 0.)--  (d2.north);
    \draw (e1.south) -- ++(0.0,-0.5) -| ++(-1.0, -0.5) -- ++(0.0,-0.0)-- ++(1.0, 0.)--  (e2.north);
   \draw (f1)--(f2);
 \draw (h1)--(h2);
\end{tikzpicture}
\end{center}

 \begin{enumerate}
 \setcounter{enumi}{2}
  \item With this notation, we can depict the three maps in Definition \ref{FactDefMonoid} as follows:
  \end{enumerate}
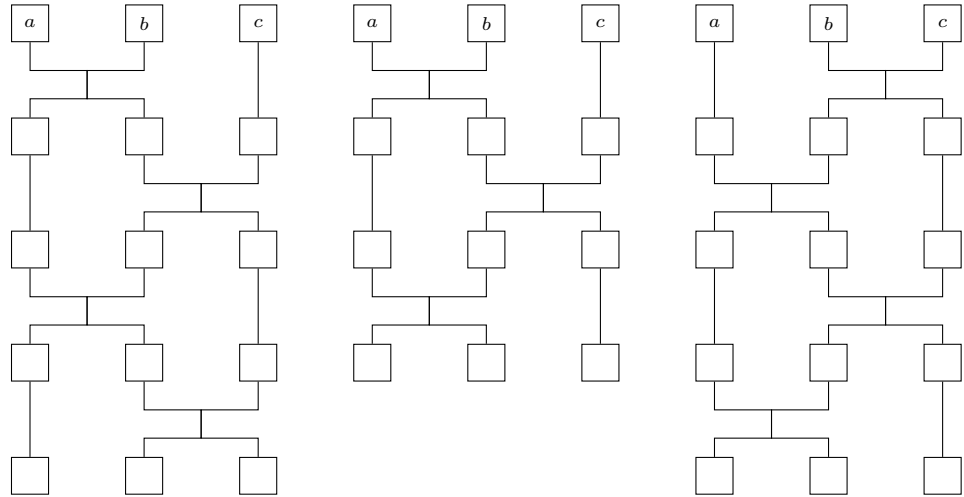
\vspace{0.5cm}
\begin{center}
 \begin{tikzpicture}[node distance = 1.5cm, auto, scale=0.75, minimum size=4.875mm]
 \tikzstyle{every node}=[font=\tiny]
 \node[rectangle, draw](a1){$a$};
 \node[rectangle, draw,right of=a1](b1){$b$};
 \node[rectangle, draw, right of=b1](c1){$c$};
 \node[rectangle, draw, below of=a1](a2){};
 \node[rectangle, draw, below of=b1](b2){};
 \node[rectangle, draw, below of=c1](c2){};
  \node[rectangle, draw, below of=a2](a3){};
 \node[rectangle, draw, below of=b2](b3){};
 \node[rectangle, draw, below of=c2](c3){};
  \node[rectangle, draw, below of=a3](a4){};
 \node[rectangle, draw,  below of=b3](b4){};
 \node[rectangle, draw,  below of=c3](c4){};
  \node[rectangle, draw,  below of=a4](a5){};
 \node[rectangle, draw,  below of=b4](b5){};
 \node[rectangle, draw,  below of=c4](c5){};
 
 \node[rectangle, draw,  left of=a1](f1){$c$};
\node[rectangle, draw, left of=f1](e1){$b$};
\node[rectangle, draw, left of=e1](d1){$a$};
 \node[rectangle, draw, below of=d1](d2){};
 \node[rectangle, draw, below of=e1](e2){};
 \node[rectangle, draw,  below of=f1](f2){};
  \node[rectangle, draw, below of=d2](d3){};
 \node[rectangle, draw,  below of=e2](e3){};
 \node[rectangle, draw,  below of=f2](f3){};
  \node[rectangle, draw, below of=d3](d4){};
 \node[rectangle, draw,  below of=e3](e4){};
 \node[rectangle, draw,  below of=f3](f4){};
 
  \node[rectangle, draw, left of=d1](i1){$c$};
\node[rectangle, draw, left of=i1](h1){$b$};
\node[rectangle, draw, left of=h1](g1){$a$};
 \node[rectangle, draw, below of=g1](g2){};
 \node[rectangle, draw, below of=h1](h2){};
 \node[rectangle, draw, below of=i1](i2){};
  \node[rectangle, draw, below of=g2](g3){};
 \node[rectangle, draw, below of=h2](h3){};
 \node[rectangle, draw, below of=i2](i3){};
  \node[rectangle, draw, below of=g3](g4){};
 \node[rectangle, draw, below of=h3](h4){};
 \node[rectangle, draw, below of=i3](i4){};
  \node[rectangle, draw, below of=g4](g5){};
 \node[rectangle, draw, below of=h4](h5){};
 \node[rectangle, draw, below of=i4](i5){};

\draw (a1) -- (a2); 
    \draw (b1.south) -- ++(0.0, -0.5) -| ++(1.0, -0.5) -- ++(0.0,-0.0)-- ++(-1.0, 0.)--  (b2.north);
    \draw (c1.south) -- ++(0.0,-0.5) -| ++(-1.0, -0.5) -- ++(0.0,-0.0)-- ++(1.0, 0.)--  (c2.north);
\draw (c2) -- (c3); 
    \draw (a2.south) -- ++(0.0, -0.5) -| ++(1.0, -0.5) -- ++(0.0,-0.0)-- ++(-1.0, 0.)--  (a3.north);
    \draw (b2.south) -- ++(0.0,-0.5) -| ++(-1.0, -0.5) -- ++(0.0,-0.0)-- ++(1.0, 0.)--  (b3.north);
\draw (a3) -- (a4); 
    \draw (b3.south) -- ++(0.0, -0.5) -| ++(1.0, -0.5) -- ++(0.0,-0.0)-- ++(-1.0, 0.)--  (b4.north);
    \draw (c3.south) -- ++(0.0,-0.5) -| ++(-1.0, -0.5) -- ++(0.0,-0.0)-- ++(1.0, 0.)--  (c4.north);
\draw (c4) -- (c5); 
    \draw (a4.south) -- ++(0.0, -0.5) -| ++(1.0, -0.5) -- ++(0.0,-0.0)-- ++(-1.0, 0.)--  (a5.north);
    \draw (b4.south) -- ++(0.0,-0.5) -| ++(-1.0, -0.5) -- ++(0.0,-0.0)-- ++(1.0, 0.)--  (b5.north);
    
    \draw (f1) -- (f2); 
    \draw (d1.south) -- ++(0.0, -0.5) -| ++(1.0, -0.5) -- ++(0.0,-0.0)-- ++(-1.0, 0.)--  (d2.north);
    \draw (e1.south) -- ++(0.0,-0.5) -| ++(-1.0, -0.5) -- ++(0.0,-0.0)-- ++(1.0, 0.)--  (e2.north);
\draw (d2) -- (d3); 
    \draw (e2.south) -- ++(0.0, -0.5) -| ++(1.0, -0.5) -- ++(0.0,-0.0)-- ++(-1.0, 0.)--  (e3.north);
    \draw (f2.south) -- ++(0.0,-0.5) -| ++(-1.0, -0.5) -- ++(0.0,-0.0)-- ++(1.0, 0.)--  (f3.north);
\draw (f3) -- (f4); 
    \draw (d3.south) -- ++(0.0, -0.5) -| ++(1.0, -0.5) -- ++(0.0,-0.0)-- ++(-1.0, 0.)--  (d4.north);
    \draw (e3.south) -- ++(0.0,-0.5) -| ++(-1.0, -0.5) -- ++(0.0,-0.0)-- ++(1.0, 0.)--  (e4.north);
    
       \draw (i1) -- (i2); 
    \draw (g1.south) -- ++(0.0, -0.5) -| ++(1.0, -0.5) -- ++(0.0,-0.0)-- ++(-1.0, 0.)--  (g2.north);
    \draw (h1.south) -- ++(0.0,-0.5) -| ++(-1.0, -0.5) -- ++(0.0,-0.0)-- ++(1.0, 0.)--  (h2.north);
\draw (g2) -- (g3); 
    \draw (h2.south) -- ++(0.0, -0.5) -| ++(1.0, -0.5) -- ++(0.0,-0.0)-- ++(-1.0, 0.)--  (h3.north);
    \draw (i2.south) -- ++(0.0,-0.5) -| ++(-1.0, -0.5) -- ++(0.0,-0.0)-- ++(1.0, 0.)--  (i3.north);
\draw (i3) -- (i4); 
    \draw (g3.south) -- ++(0.0, -0.5) -| ++(1.0, -0.5) -- ++(0.0,-0.0)-- ++(-1.0, 0.)--  (g4.north);
    \draw (h3.south) -- ++(0.0,-0.5) -| ++(-1.0, -0.5) -- ++(0.0,-0.0)-- ++(1.0, 0.)--  (h4.north);
  \draw (g4) -- (g5); 
    \draw (h4.south) -- ++(0.0, -0.5) -| ++(1.0, -0.5) -- ++(0.0,-0.0)-- ++(-1.0, 0.)--  (h5.north);
    \draw (i4.south) -- ++(0.0,-0.5) -| ++(-1.0, -0.5) -- ++(0.0,-0.0)-- ++(1.0, 0.)--  (i5.north);
 \end{tikzpicture}
\end{center}

\end{remark}

\begin{Notation} \label{NotationAlphaI}
If we have any map $\alpha\colon X^k \to X^l$ for some set $X$ and some natural numbers $k$ and $l$, we will define maps $\alpha_i\colon X^n \to X^{n-k+l}$ for $n\geq k$ and $1\leq i\leq n-k+1$ via
\begin{eqnarray*}
 \alpha_i(x_n,\ldots, x_1)=(x_n, \ldots, x_{i+k}, \alpha(x_{i+k-1},\ldots, x_i), x_{i-1}, \ldots, x_1).
\end{eqnarray*}

\end{Notation}

Since Definition \ref{FactDefMonoid} above is rather hard to check, we are going to use also an equivalent description. We will now define the recognition principle and then formulate this equivalent description.

\begin{Definition}[\cite{AlexThesis}, Recognition Principle]\label{RecognitionPrinciple}
Let $M$ be a monoid, let $\mathcal{E}$ be a generating system for $M$ and let $\eta\colon M\to M\times M$ be a factorization map. We say that $\eta$ satisfies the \textbf{recognition principle} if for all $m\in M$, $a\in \E$, the pair $(m,a)$ satisfies $\eta(ma)=(m,a)$ if and only if the pair $(m',a)$ satisfies $\eta(m'a)=(m',a)$, where $m'=\eta'(m)$. 
\end{Definition}

The criterion for factorability is now as follows:

\begin{Theorem}[\cite{AlexThesis}, Theorem 2.2.6]\label{CharacterizationRecognitionPrinciple}
Let $M$ be a monoid, let $\mathcal{E}$ be a generating system for $M$ and let $\eta\colon M\to M\times M$ be a factorization map. Then $\eta$ is a factorability structure on $M$ in the sense of Definition \ref{FactDefMonoid} if and only if it is a weak factorability structure and satisfies in addition Recognition Principle \ref{RecognitionPrinciple}. 
\end{Theorem}
We will subdivide the proof in several parts. Some parts are due to or inspired by M. Rodenhausen. We start by showing that factorability structure is a weak factorability structure:
\begin{Lemma}\label{FaktorabilitaetImpliziertSchwacheFakt}
 Let $M$ be a monoid with a chosen generating system $\mathcal{E}$ and let $\eta\colon M\to M\times M$ be a factorability structure. Then $\eta$ is also a weak factorability structure. 
\end{Lemma}

\begin{pf}
Consider a pair $(x,t)\in M\times \E$. We want to use Lemma \ref{FactAxiom}. So assume $(x',t)$ and $(\overline{x}, \overline{x't})$ are geodesic pairs. We consider the triple $(x,t,1)\in M^3$. We apply to it $f_2f_1f_2f_1$ first. We use Remark \ref{EtaErzeuger} (from now on without mentioning it explicitly).

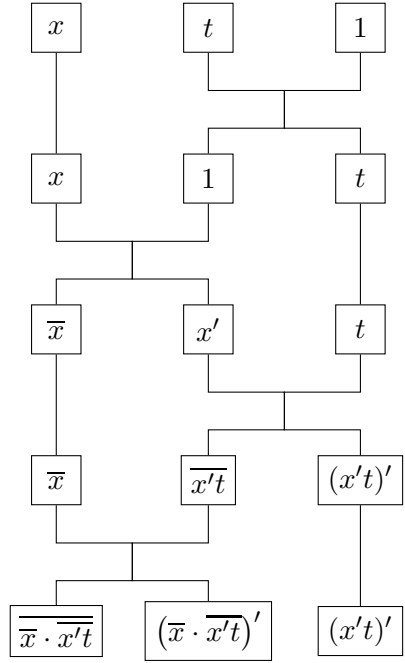
\begin{center}
     \begin{tikzpicture}[node distance = 2cm, auto]
 \node[cloud](a1){$x$};
 \node[cloud, right of=a1](b1){$t$};
 \node[cloud, right of=b1](c1){$1$};
 \node[cloud, below of=a1](a2){$x$};
 \node[cloud, below of=b1](b2){$1$};
 \node[cloud, below of=c1](c2){$t$};
  \node[cloud, below of=a2](a3){$\overline{x}$};
 \node[cloud, below of=b2](b3){$x'$};
 \node[cloud, below of=c2](c3){$t$};
  \node[cloud, below of=a3](a4){$\overline{x}$};
 \node[cloud, below of=b3](b4){$\overline{x't}$};
 \node[cloud, below of=c3](c4){$(x't)'$};
  \node[cloud, below of=a4](a5){$\overline{\overline{x}\cdot\overline{x't}}$};
 \node[cloud, below of=b4](b5){$\left(\overline{x}\cdot\overline{x't}\right)'$};
 \node[cloud, below of=c4](c5){$(x't)'$};


\draw (a1) -- (a2); 
    \draw (b1.south) -- ++(0.0, -0.5) -| ++(1.0, -0.5) -- ++(0.0,-0.0)-- ++(-1.0, 0.)--  (b2.north);
    \draw (c1.south) -- ++(0.0,-0.5) -| ++(-1.0, -0.5) -- ++(0.0,-0.0)-- ++(1.0, 0.)--  (c2.north);
\draw (c2) -- (c3); 
    \draw (a2.south) -- ++(0.0, -0.5) -| ++(1.0, -0.5) -- ++(0.0,-0.0)-- ++(-1.0, 0.)--  (a3.north);
    \draw (b2.south) -- ++(0.0,-0.5) -| ++(-1.0, -0.5) -- ++(0.0,-0.0)-- ++(1.0, 0.)--  (b3.north);
\draw (a3) -- (a4); 
    \draw (b3.south) -- ++(0.0, -0.5) -| ++(1.0, -0.5) -- ++(0.0,-0.0)-- ++(-1.0, 0.)--  (b4.north);
    \draw (c3.south) -- ++(0.0,-0.5) -| ++(-1.0, -0.5) -- ++(0.0,-0.0)-- ++(1.0, 0.)--  (c4.north);
\draw (c4) -- (c5); 
    \draw (a4.south) -- ++(0.0, -0.5) -| ++(1.0, -0.5) -- ++(0.0,-0.0)-- ++(-1.0, 0.)--  (a5.north);
    \draw (b4.south) -- ++(0.0,-0.5) -| ++(-1.0, -0.5) -- ++(0.0,-0.0)-- ++(1.0, 0.)--  (b5.north);
 \end{tikzpicture}
    \end{center}

By assumption, all of the multiplied pairs are geodesic, so the application of $f_2f_1f_2f_1$ preserves the sum of the norms when applied to $(x,t,1)$. By the definition of factorability, so does $f_2f_1f_2$, and the results are equal.
    
The application of $f_2f_1f_2$ to the triple $(x,t,1)$ yields:
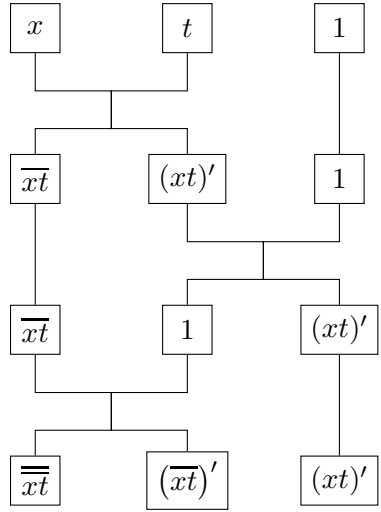
\begin{center}
   \begin{tikzpicture}[node distance = 2cm, auto]
 \node[cloud](a1){$x$};
 \node[cloud, right of=a1](b1){$t$};
 \node[cloud, right of=b1](c1){$1$};
 \node[cloud, below of=a1](a2){$\overline{xt}$};
 \node[cloud, below of=b1](b2){$(xt)'$};
 \node[cloud, below of=c1](c2){$1$};
 \node[cloud, below of=a2](a3){$\overline{xt}$};
 \node[cloud, below of=b2](b3){$1$};
 \node[cloud, below of=c2](c3){$(xt)'$};
 \node[cloud, below of=a3](a4){$\overline{\overline{xt}}$};
 \node[cloud, below of=b3](b4){$\left(\overline{xt}\right)'$};
 \node[cloud, below of=c3](c4){$(xt)'$};
\draw (c1) -- (c2); 
    \draw (a1.south) -- ++(0.0, -0.5) -| ++(1.0, -0.5) -- ++(0.0,-0.0)-- ++(-1.0, 0.)--  (a2.north);
    \draw (b1.south) -- ++(0.0,-0.5) -| ++(-1.0, -0.5) -- ++(0.0,-0.0)-- ++(1.0, 0.)--  (b2.north);
    \draw (a2) -- (a3);
    \draw (b2.south) -- ++(0.0, -0.5) -| ++(1.0, -0.5) -- ++(0.0,-0.0)-- ++(-1.0, 0.)--  (b3.north);
    \draw (c2.south) -- ++(0.0,-0.5) -| ++(-1.0, -0.5) -- ++(0.0,-0.0)-- ++(1.0, 0.)--  (c3.north);
    \draw (c3) -- (c4); 
    \draw (a3.south) -- ++(0.0, -0.5) -| ++(1.0, -0.5) -- ++(0.0,-0.0)-- ++(-1.0, 0.)--  (a4.north);
    \draw (b3.south) -- ++(0.0,-0.5) -| ++(-1.0, -0.5) -- ++(0.0,-0.0)-- ++(1.0, 0.)--  (b4.north);

 \end{tikzpicture}
    \end{center}
    
    Observe that an application of $f_i$ to any tuple either preserves or lowers the sum of the norms of the entries. Since we know that the application of $f_2f_1f_2$ does not lower the norm, we can conclude that $(x,t)$ is a geodesic pair. Moreover, since the resulting triples of applying $f_2f_1f_2$ and $f_2f_1f_2f_1$ are equal, they are still equal after multiplying the two left entries, and so we obtain:
    \begin{eqnarray*}
     (\overline{x}\cdot \overline{x't}, (x't)')=(\overline{xt}, (xt)').
    \end{eqnarray*}
This was exactly to be shown according to Lemma \ref{FactAxiom}. 
\end{pf}

Next, we are going to show that a factorability structure satisfies the recognition principle. 

\begin{Lemma}
 Let $M$ be a monoid with a chosen generating system $\E$, and let $\eta$ be a factorability structure on $(M,\E)$. Then $\eta$ satisfies Recognition Principle (as in Definition \ref{RecognitionPrinciple}).
\end{Lemma}

\begin{pf}
 Let $x$ be in $M$ and let $t$ be an element of $\E$. 
 
 First, assume that the pair $(x,t)$ is $\eta$-stable. In particular, it is geodesic. By Remark \ref{EineImplikationFaktorabel}, we are in the same situation as in the proof of the last lemma, so we obtain in particular:
 \begin{eqnarray*}
 f_2f_1f_2(x, t,1)=(\overline{\overline{xt}},\left(\overline{xt}\right)',(xt)').
 \end{eqnarray*}
 By assumption, this triple reduces to $(\overline{x}, x', t)$. Since $\eta$ is a factorability structure and $f_2f_1f_2$ is norm-preserving on $(x,t,1)$, we know that the application of $f_1$ to $f_2f_1f_2(x,t,1)$ does not change the triple, so we have 
 \begin{eqnarray*}
  \eta(x't)=(x',t),
 \end{eqnarray*}
 i.e., the pair $(x', t)$ is $\eta$-stable.
 
 Conversely, assume now that the pair $(x', t)$ is $\eta$-stable. Then, as in the last lemma, 
 \begin{eqnarray*}
  f_2f_1f_2f_1(x,t,1)=(\overline{\overline{x}\cdot\overline{x't}}, \left(\overline{x}\cdot\overline{x't}\right)', (x't)').
 \end{eqnarray*}
By assumption, this triple simplifies to $(\overline{x}, x', t)$. In particular, the application of $f_2f_1f_2f_1$ is norm-preserving, so we get an equality 
 \begin{eqnarray*}
 (\overline{x}, x', t)=f_2f_1f_2(x, t,1)=(\overline{\overline{xt}},\left(\overline{xt}\right)',(xt)').
 \end{eqnarray*}
So $(xt)'=t$ and 
\begin{eqnarray*}
 \overline{\overline{xt}}\cdot\left(\overline{xt}\right)'=\overline{xt}=\overline{x}\cdot x'=x,
\end{eqnarray*}
thus the pair $(x,t)$ is stable. This finishes the proof of the lemma.
\end{pf}

For the last part of Theorem \ref{CharacterizationRecognitionPrinciple}, we need to show the other implication. This is done in the next lemma.

\begin{Lemma} \label{WeakFactRecognImpliesFact}
 Let $M$ be a monoid, $\E$ a chosen generating system and $\eta$ a weak factorability structure on $M$ satisfying recognition principle. Then $\eta$ is a factorability structure. 
\end{Lemma}
\begin{pf}
 Note that commutativity in graded sense behaves well under precomposition with arbitrary maps and under postcomposition with maps which do not increase the norm. If we write $\alpha\equiv \beta$ for equality in graded sense, we can express the compatibility with composition as $\alpha\gamma\equiv \beta \gamma$ for all $\gamma$ and $\gamma \alpha\equiv \gamma \beta$ for $\gamma$ not norm-increasing. 
 
 We denote by $\mu$ the multiplication in the monoid and use Notation \ref{NotationAlphaI} to define $\mu_i$ on tuples of monoid elements. Then we can write the compositions in Diagram \ref{eqWF} as $\eta_1\mu_1$ and $\mu_2\eta_1\mu_1\eta_2$, respectively. Recall that by Remark \ref{EineImplikationFaktorabel}, these compositions are equal in graded sense on all of $M\times M$. Note that we can write $f_i$ as $\eta_i\mu_i$. So we consider 
 \begin{eqnarray*}
  f_2f_1f_2=\eta_2\mu_2\eta_1\mu_1\eta_2\mu_2.
 \end{eqnarray*}
By the preconsideration, we can derive from Diagram \ref{eqWF} that 
\begin{eqnarray*} 
f_2f_1f_2\equiv \eta_2\eta_1\mu_1\mu_2. 
\end{eqnarray*} 

On the other hand, 
\begin{eqnarray*}
 f_2f_1f_2f_1=\eta_2\mu_2\eta_1\mu_1\eta_2\mu_2\eta_1\mu_1\equiv\eta_2\eta_1\mu_1\mu_2\eta_1\mu_1.
\end{eqnarray*}
Observe that associativity of $\mu$ can be expressed by $\mu_1\mu_2=\mu_1\mu_1$. Furthermore, the condition (F2) implies that $\mu_i\eta_i=\id$. This in turn shows:
\begin{eqnarray*}
 f_2f_1f_2f_1\equiv\eta_2\eta_1\mu_1\mu_1=\eta_2\eta_1\mu_1\mu_2.
\end{eqnarray*}
Thus, we have shown $f_2f_1f_2\equiv f_2f_1f_2f_1$. 

Now we have to show that $f_2f_1f_2\equiv f_1f_2f_1f_2$. Observe that it is enough to show that whenever $f_2f_1f_2$ is norm-preserving for some triple, the application of $f_1$ to the resulting triple does not change it. Given a triple $(x_3,x_2,x_1)$ for which $f_2f_1f_2$ is norm-preserving, we set $x$ to be the product $x_3x_2x_1=\mu_1\mu_2(x_3,x_2,x_1)$. Then
\begin{eqnarray*}
 f_2f_1f_2(x_3, x_2,x_1)=\eta_2\eta_1(x)=(\overline{\overline{x}}, \left(\overline{x}\right)', x').
\end{eqnarray*}

Since $(\overline{x}, x')$ is by definition a stable pair, we may apply the Recognition Principle and conclude that $\left(\left(\overline{x}\right)', x'\right)$ is a stable pair, so the triple $f_2f_1f_2(x_3, x_2, x_1)$ remains unchanged under $f_1$. This completes the proof of the lemma.

\end{pf}

Observe that the last three lemmas combined give precisely Theorem \ref{CharacterizationRecognitionPrinciple}.

The following corollary will be important in our setting. The statement was first observed by M.Rodenhausen.

\begin{Cor}[\cite{AlexThesis}, Sections 2.1-2.2]\label{RechtskuerzbarAequivalenz}
 For right-cancellative monoids, the notion of factorability as in Definition \ref{FactDefMonoid} coincides with the notion of weak factorability. 
\end{Cor}

\begin{pf}
Due to Theorem \ref{CharacterizationRecognitionPrinciple}, we only have to show that a weak factorability structure on a right-cancellative monoid always satisfies Recognition Principle. So let $M$ be a right-cancellative monoid, $\E$ a generating set for $M$ and $\eta$ a weak factorability structure on $M$.

First, let $(x,t)\in M\times \E$ be a stable pair. In particular, it is geodesic, thus we know that 
\begin{eqnarray*}
 (x,t)=(\overline{x}\cdot \overline{x't}, (x't)'). 
\end{eqnarray*}
Thus $t=(x't)'$ and $x'\cdot t=\overline{x't}\cdot (x't)'$, so cancelling on the right yields $x'=\overline{x't}$, so that the pair $(x',t)$ is stable.

Now assume that the pair $(x',t)$ is stable for some pair $(x,t) \in M\times \E$. Then this pair is geodesic, and the pair $(\overline{x}, \overline{x't})=(\overline{x}, x')$ is geodesic as well. So by Lemma \ref{FactAxiom}, we know that $(x,t)$ is geodesic and that
\begin{eqnarray*}
 \eta(xt)=(\overline{x}\cdot\overline{x't}, (x't)')=(x,t).
\end{eqnarray*}
We have established the Recognition Principle, and thus the proof of the corollary.
\end{pf}

\section{Local Factorability and Normal Forms}
\label{Local factorability}

An important property of factorable monoids is the existence of well-behaved normal forms. Indeed, if we start with an element $x \in M$, we first write it as $x=\overline{x}x'$. We can continue with $\overline{x}$ and write it again as $\overline{x}=\overline{\overline{x}}\cdot (\overline{x})'$. Inductively, we can write $x$ as a product $x_k\ldots x_1$ of $k=N(x)$ generators. This normal form has the property of being everywhere stable. We start with the definition of this notion. 

\begin{Definition}[\cite{AlexThesis}, Rodenhausen]\label{DefinitionStable} 

 Let $(M,\mathcal{E}, \eta)$ be a factorable monoid. We say a tuple $(x_n, \ldots, x_1)\in M^n$ is \textbf{stable} at the $i$-th position if $\eta(x_{i+1}x_i)=(x_{i+1}, x_i)$. 
  We call a tuple \textbf{everywhere stable} if it is stable at each position. 
\end{Definition}

The tuple $(x_k, \ldots, x_1)$ associated with $x \in M$ by the procedure defined above has now the following property:

\begin{Lemma}[\cite{AlexThesis} Remark 2.1.27, Rodenhausen]\label{EtaNormalFormStabil}
 Let $(M,\mathcal{E}, \eta)$ be a factorable monoid. Then the tuple $(x_k, \ldots, x_1)$ associated with $x \in M$ by the procedure as above is everywhere stable. This tuple associated with $x$ will be called the \textbf{normal form} of $x$ (with respect to this factorability structure). 
\end{Lemma}

M.~Rodenhausen gave the following alternative description of factorability, which needs to be defined only on pairs of generators. A byproduct of this description is the existence of very special presentations for factorable monoids. Since the proofs of M.~Rodenhausen were unpublished so far, the second author wrote them down in her thesis (\cite{MyThesis}) with Rodenhausen's kind permission. 

\begin{Definition}(Rodenhausen)\label{Phifaktorabilitaet}
 Let $M$ be a monoid and $\mathcal{E}$ a generating system of this monoid. Denote by $\mathcal{E}^+$ the union of this generating system with $\{1\}$, and by $\mathcal{E}^*$ the free monoid generated by $\mathcal{E}$. In this section, we always assume $1\notin \E$. Then a \textbf{local factorability structure} is a map 
\begin{eqnarray*}
 \varphi\colon \mathcal{E}^+\times \E^+ \to \E^+\times \E^+
\end{eqnarray*}
with the following properties: 
\begin{enumerate}
 \item $M\cong \left\langle\mathcal{E}| (a,b)=\varphi(a,b)\right\rangle$.
 \item Idempotency: $\varphi^2=\varphi$.
 \item Value on norm $1$ elements: $\varphi(a,1)=(1,a)$.
 \item Stability for triples: $\varphi_2\varphi_1\varphi_2(a,b,c)$ is ($\varphi$-)totally stable (i.e., applying any $\varphi_i$ to this tuple leaves it unchanged) or contains a $1$, for all $a,b,c\in \mathcal{E}$.
 \item Normal form condition: $\NF(a,b,c)=\NF(\varphi_1(a,b,c))$ for all $a,b,c\in \E$. 
\end{enumerate}
(Recall we use the notational convention \ref{NotationAlphaI} to define $\varphi_i$). Here, the normal form of a tuple $(a_n, \ldots, a_1)$ is an element of $\E^*$ defined inductively as follows: The normal form of a string containing $1$ is the normal form of the same string with $1$ removed. For a string not containing $1$, define
\begin{eqnarray*}
 \NF(a_n,\ldots, a_1)=\begin{cases}
                       \varphi_{n-1}\ldots \varphi_1(\NF(a_n,\ldots, a_2), a_1), \mbox{ if it contains no } 1; \\
		      \NF(\varphi_{n-1}\ldots \varphi_1(\NF(a_n,\ldots, a_2), a_1)), \mbox{ otherwise. }
                      \end{cases}
\end{eqnarray*}
Define $\NF(a)=a$ for all $a\in \E$ and $\NF(())=()$. 
Furthermore, we call a string of the form $(1, 1,\ldots, 1, \NF(x))$ an \textbf{extended} normal form of $x$.
\end{Definition}

\begin{Theorem}[Rodenhausen]\label{PhiFaktorabilitaetEtaFaktorabilitaet}
 If a monoid $M$ with a generating system $\mathcal{E}$ is factorable, then $\varphi(a,b)=\eta(ab)$ defines a local factorability structure on this monoid. Conversely, one can construct out of a local factorability structure a factorability structure in the usual sense, and those share the same normal forms. These constructions are inverse to each other. 
\end{Theorem}

We will need some parts of the proof in order to construct a factorable monoid with non-complete associated rewriting system. The corresponding lemmas can be found in the appendix. 

\section{Rewriting System Basics}
\label{Rewriting system basics}

Our first aim is to investigate a connection between factorability structures and complete rewriting systems. The basic notions of rewriting systems are the topic of this section. The exposition in this section is based on D.~Cohen's survey article (\cite{Cohen}).

The basic idea of a rewriting system is easy: In a monoid presentation with generating set $S$, we specify a relation not just by a set, but by a pair of two words in the free monoid $S^*$ over $S$, i.e., an element in $S^*\times S^*$. Thus, we are going to determine a ``direction'' for each relation, and it must not be applied in the other direction. There are several properties which are desirable when considering such a rewriting system. We will be mostly interested in the notion of a complete rewriting system. In a complete rewriting system, any non-trivial chain of applications of rewriting rules stops after finite time producing a nice normal form. For our purposes, a result of K.~Brown (\cite{Brown}), which relates complete rewriting systems and noetherian matchings, is of particular interest. It will be made more precise in the next section.

	\begin{Definition}
		Let $S$ be a set (sometimes called alphabet) and denote by $S^*$ the free monoid over $S$. A set of \textbf{rewriting rules} $\mathcal R$ on $S$ is a set of tuples $(l,r) \in S^* \times S^*$. The string $l$ is called the \textbf{left side} and $r$ is called the \textbf{right side} of the rewriting rule.
		\begin{enumerate}
			\item	We introduce a relation on $S^*$ as follows: We say that $w$ \textbf{rewrites} to $z$, denoted by $w \to_{\mathcal R} z$, if there exist $u, v \in S^*$ and some rewriting rule $(l,r) \in \mathcal R$ such that $w = ulv$ and $z = urv$.
			
			\item	A word $w \in S^*$ is called \textbf{reducible} (with respect to $\mathcal R$) if there is some $z$ such that $w \to_{\mathcal R} z$. Otherwise, it is called \textbf{irreducible} (with respect to $\mathcal R$).
			
			\item	Denote by $\leftrightarrow_{\mathcal R}$ the reflexive, symmetric and transitive closure of $\to_{\mathcal R}$. Two words $w,z$ over $S$ are called \textbf{equivalent} if $w \leftrightarrow_{\mathcal R} z$. Set $M = S^* / \leftrightarrow_{\mathcal R}$. We then say that $(S, \mathcal R)$ is a \textbf{(string) rewriting system} for the monoid $M$.
		\end{enumerate}
	\end{Definition}

	  We now can define complete rewriting systems. 

	\begin{Definition}
		Let $(S, \mathcal R)$ be a rewriting system.
		\begin{enumerate}
			\item	$(S, \mathcal R)$ is called \textbf{minimal} if the right side $r$ of every rewriting rule $(l,r) \in \mathcal R$ is irreducible and if in addition the left side $l$ of every rewriting rule $(l,r) \in \mathcal R$ is irreducible with respect to $\mathcal R \sm \{(l,r)\}$. 
			
			\item	$(S, \mathcal R)$ is called \textbf{strongly minimal} if it is minimal and if in addition every element $s \in S$ is irreducible.
			
			\item	$(S, \mathcal R)$ is called \textbf{noetherian} if there is no infinite sequence
\begin{eqnarray*}
 w_1 \to_{\mathcal R} w_2 \to_{\mathcal R} w_3 \to_{\mathcal R} \ldots
\end{eqnarray*}
 of rewritings. This implies that every sequence of rewritings eventually arrives at an irreducible word.
			
			\item	$(S, \mathcal R)$ is called \textbf{convergent} if it is noetherian and if in every equivalence class of $\leftrightarrow_{\mathcal R}$ there is only one irreducible element.
		
		\item  A rewriting system is called \textbf{complete} if it is strongly minimal and convergent.\end{enumerate}
	\end{Definition}

As already mentioned, successive rewritings in a complete rewriting system $(S, \mathcal{R})$ induce a normal form, assigning to each $x\in M$ the unique irreducible word in $S^*$ in its equivalence class. 

We want to illustrate the notions with some small examples. 
\begin{example}
 \begin{enumerate}
  \item The set $S=\{a,b\}$ can be equipped with the rewriting system $aba\to bab$. The word $bab\in S^*$ is then irreducible, the word $abab$ is reducible. This rewriting system is strongly minimal. It is noetherian since each rewriting strictly decreases the number of $a$'s in the word. It is not convergent: We can rewrite $ababa$ to $babba$ and to $abbab$, which are both irreducible. The monoid defined by this rewriting system is isomorphic to the positive braid monoid $B_3^+$ on $3$ strands. 
 \item We can also equip the set $S=\{a,b\}$ with the rewriting system $ab\to ba$. It is not hard to see that this system is complete, and the set of normal forms is given by $b^{k}a^{l}$ with $k,l\geq 0$. The monoid given by this rewriting system is just the free abelian monoid on two generators.
 \item We can equip the set $S=\{a,b,c\}$ with the rewriting system $ab\to bc,\, bc\to ca$. This system is not minimal since there is a rewriting rule ($ab\to bc$), the right-hand side of which is reducible due to $bc \to ca$. The monoid defined by this rewriting system is a special case of Birman-Ko-Lee monoids (cf.\ \cite{BKL}). These are closely related to the braid groups.
 \item We consider the set $S=\{a,b,c\}$ with the rewriting system $ab\to bc,\, ab\to ca$. This system is again not minimal, since the left side $ab$ of the rewriting rule $ab\to bc$ is reducible even if we delete this rewriting rule. Note that this rewriting system defines the same monoid as the one in the last point.
 \item An easy example of a rewriting system which is not noetherian is given by a rewriting system $a\to a^2$ on the single generator $a$. 
 \end{enumerate}

\end{example}

The following theorem reveals an important connection between complete rewriting systems and homological properties of a monoid. 

\begin{Theorem}[Brown \cite{Brown}] \label{TheoremBrown}
	 Let $M$ be a monoid given by a complete rewriting system $(S, \mathcal{R})$. Then there exists a space homotopy equivalent to the bar complex of $M$ with cells in bijection with tuples $[x_n|\ldots |x_1]$ subject to the following conditions: If $w_i\in S^*$ is the irreducible representative of $x_i$, then we require
 \renewcommand{\labelenumi}{ (\alph{enumi})}
\begin{enumerate}
 \item $w_1\in S$,
 \item The word $w_{i+1}w_i$ is reducible for every $1\leq i\leq n-1$,
 \item  For every $1\leq i\leq n-1$, any proper (right) prefix of $w_{i+1}w_i$ is irreducible. 
\end{enumerate}
	\end{Theorem}

\section{Rewriting System of a Factorable Monoid}
\label{Rewriting System of a Factorable Monoid}
We want to deal with the question when a factorability structure on a monoid provides a complete rewriting system (as defined in Section \ref{Rewriting system basics}) for this monoid. We always obtain a rewriting system with exactly one irreducible element in each equivalence class, but we will show that it is not necessarily noetherian. We will yet exhibit several cases where the given rewriting system is noetherian.
First, we are going to make precise which rewriting system is going to be associated with a factorability structure on a monoid. The choice is quite self-evident.

\begin{Lemma} \label{FactorabilityStructureInducesRewriting}
 Let $(M, \mathcal{E}, \eta)$ be a factorable monoid. Then the rewriting rules 
\begin{eqnarray*}
 (x,y)\to \eta(xy) \mbox{ for } x,y\in \mathcal{E} \mbox{ if } (x,y) \mbox{ is unstable}
\end{eqnarray*}
 define a strongly minimal rewriting system on $M$ with exactly one irreducible element in each equivalence class. (Here, if $\overline{xy}=1$, we interpret the rewriting rule as $(x,y)\to xy$.)
\end{Lemma}

\begin{pf}
  First, by Proposition \ref{PhiFaktorabilitaetEtaFaktorabilitaet} of M.~Rodenhausen, we know that this rewriting system defines exactly the monoid $M$ we started with. Since the right-hand side of each rewriting rule is always a stable pair or a single element of $\mathcal{E}$ or $1$, we know that right-hand sides of our rewriting rules are irreducible. Furthermore, the elements of $\mathcal{E}$ are irreducible, and left-hand sides are irreducible if we remove the rule containing them. Thus this rewriting system is strongly minimal. 

  Moreover, each representative of an element of the monoid can be brought into its normal form with the normal form procedure described in Definition \ref{Phifaktorabilitaet}, which is obviously a chain of applications of the rewriting rules above. Furthermore, Lemma \ref{EtaNormalFormStabil} (in presence of Theorem \ref{PhiFaktorabilitaetEtaFaktorabilitaet}) states that the obtained normal form is totally stable, which translates exactly into irreducible in the language of rewriting systems. This implies there is exactly one irreducible element in each equivalence class of words in the free monoid on $\mathcal{E}$ (under the equivalence relation generated by above rules).
\end{pf}

Unfortunately, this rewriting system is not always noetherian, even if $\mathcal{E}$ is finite and the resulting monoid is right-cancellative. See appendix for an  example of a factorable monoid where the associated rewriting system is not noetherian.

\section{Main Lemma for Noetherianity}
\label{Finiteness of Qn'}

Our aim is to prove Theorem \ref{AnBnCn}, which reveals some of the structure of the monoid $Q_n'$, defined below. It is closely connected to the proof that the rewriting system defined by a factorability structure is noetherian in some cases. We will introduce some notation which is related to handling the applications of rewriting rules later. 

\begin{Definition}[\cite{AlexThesis}] \label{DefinitionMonoidQn}
Let $F_n$ be a free monoid on letters $1,2,\ldots, n$ (with the empty string as a neutral element). The elements of $F_n$ will be denoted either like $(1\,2\,3\,4)$ or like $(1,2,3,4)$ to increase the readability; sometimes, we also omit the brackets. Let $\sim_{P}$ be the congruence generated by 
\begin{eqnarray*}
ab\sim_P ba \mbox{ for } \abs{a-b}\geq 2\mbox{ and }\\
a^2\sim_P a \mbox{ for } 1\leq a \leq n.
\end{eqnarray*}
Recall that a congruence is a left and right invariant equivalence relation. 
Let $P_n$ be the quotient of $F_n$ by this congruence. Define now a congruence $\sim_{Q}$ on $P_n$ generated by 
\begin{eqnarray*}
(k \enspace k+1 \enspace k \enspace k+1)\sim_Q (k+1 \enspace k \enspace k+1) \mbox{ and }\\
(k+1 \enspace k \enspace k+1 \enspace k)\sim_Q (k+1 \enspace k \enspace k+1).
\end{eqnarray*}
Let $Q_n$ be the quotient of $P_n$ by this congruence. Last, define a congruence $\sim$ on $Q_n$ generated by the following relation: If for $I, J\in F_n$ the relation $kIJ\sim_Q IJ$ holds, $kI\sim_P Ik$ and $k$ does not occur in $I$, then we set $kJ\sim J$. Define the quotient monoid of this congruence to be $Q_n'$. 
\end{Definition}

The following evaluation lemma should motivate the definition of $Q_n$. Recall that in a monoid with chosen generating system $\mathcal{E}$, we denote by $N_{\mathcal{E}}$ the word-length with respect to $\mathcal{E}$. 
\begin{Lemma}[\cite{AlexThesis}, Evaluation Lemma]\label{EvaluationLemmaAlex}
  Let $(M, \mathcal{E}, \eta)$ be a factorable monoid. For any sequence $I=(i_s, \ldots, i_1) \in F_n$, we define $f_I\colon M^{n+1} \to M^{n+1}$ to be the composition $f_{i_s}\circ f_{i_{s-1}}\circ \ldots \circ f_{i_1}$. If $I \sim_Q J$, then the maps $f_I, f_J\colon M^{n+1} \to M^{n+1}$
are equal in the graded sense, i.e., for any tuple $(m_{n+1}, \ldots, m_1)$, we have either $f_I(m_{n+1}, \ldots, m_1)=f_J(m_{n+1}, \ldots, m_1)$ or both tuples have strictly smaller word-length: if we write
\begin{eqnarray*} 
f_I(m_{n+1}, \ldots, m_1) &=&(y_{n+1}, \ldots, y_1)  \mbox{ and}\\
 f_J(m_{n+1}, \ldots, m_1)&=&(z_{n+1}, \ldots, z_1),
 \end{eqnarray*}
  then the condition is
\begin{eqnarray*}
N_{\E}(m_{n+1})+\ldots + N_{\E}(m_1)>N_{\E}(y_{n+1})+\ldots + N_{\E}(y_1)
\end{eqnarray*}
and 
\begin{eqnarray*}
N_{\E}(m_{n+1})+\ldots + N_{\E}(m_1)>N_{\E}(z_{n+1})+\ldots + N_{\E}(z_1).
\end{eqnarray*}
\end{Lemma}

Thus, instead of proving graded identities for the $f_i$'s, we may often prove the corresponding identities in $Q_n$. We will also show later that a similar, but weaker evaluation lemma holds for $Q_n'$. 

\begin{remark}
Note that application of relations $\sim_P$ and $\sim_Q$ never changes the set of letters occuring in a sequence. This implies that the inclusion of generators induces monoid inclusions $P_n \subset P_{n+1}$ and $Q_n\subset Q_{n+1}$. Moreover, the relation $\sim$ also does not change the set of letters occuring in a sequence: The definition of $\sim$ requires that $kIJ\sim_Q IJ$ and $k$ does not occur in $I$, but since $\sim_Q$ preserves the set of letters, $k$ has to occur in $J$, thus the claim. This also implies $Q_n'$ is a submonoid of $Q_{n+1}'$. 
\end{remark}

For later use, we fix some notation for these monoids and collect some facts about them.

\begin{Notation}\label{NotationLeftMost}
We denote by $\sh_k\colon F_{n-k} \to F_n$ the shift homomorphism induced by $i \mapsto i+k$. 

For $I, J \in F_n$, we write $I\subset J$ if $I$ is a (possibly disconnected) subsequence of $J$. 

Denote by $I_{a}^{b}$ the sequence $( a \mbox{ } a+1 \ldots \mbox{ } b-1 \mbox{ } b)$. 

Denote by $D_k$ the sequence $I_k^{k} I_{k-1}^k \ldots I_2^k I_1^k$. 
\end{Notation}

The elements $D_n$ play a very special role in the monoids $Q_n$.

\begin{Theorem}[\cite{AlexThesis}, Section 2.3]\label{Dn}
The element represented by $D_n$ in $Q_n$ is an absorbing element, i.e., for any $I\in F_n$, we have $ID_n\sim_Q D_n\sim_Q D_nI$.
\end{Theorem}

\begin{remark}
The monoids $Q_n$ were considered independently by D.~Krammer (\cite{KrammerArtin}). In \cite{KrammerArtin}, Proposition 67 shows that $D_n$ is an absorbing element in $Q_n$. 
\end{remark}

In particular, the Evaluation Lemma implies that $f_{D_n}(m_{n+1}, \ldots, m_1)$ is everywhere stable if $f_{D_n}$ does not drop the norm. An immediate consequence of this is the following proposition. 

\begin{Prop}[\cite{AlexThesis}, Section 2.3]
 Let $(M, \mathcal{E}, \eta)$ be a factorable monoid. Let $m$ be an element of $M$ and let $m_n\ldots m_1$ be a minimal word in $\mathcal{E}$ representing $m$. Then $f_{D_{n-1}}(m_n, \ldots, m_1)$ is the normal form of $m$. 
\end{Prop}

The following definitions introduce particularly nice representatives of elements in $P_n$. These are going to make it easier to track down applications of some $f_I$ to tuples of monoid elements. 

\begin{Definition}[\cite{AlexThesis}]\label{DefinitionLeftMost}
 A sequence $(i_s, \ldots, i_1)\in F_n$ is called \textbf{left-most} if for every $s>t\geq 1$, the following holds: if $\abs{i_{t+1}-i_t}\geq 2$, then $i_{t+1}>i_t$. In other words, $i_t$ exceeds $i_{t+1}$ at most by $1$.
A sequence $(i_s, \ldots, i_1)$ is called \textbf{reduced} if it contains no subsequent equal entries.
\end{Definition}

\begin{example}
 The sequence $(4,2,1,2,3)\in F_4$ is left-most. In contrast, the sequence $(2,4,1,2,3)$ is not left-most as $\abs{2-4}\geq 2$, but $2<4$. 
\end{example}

The following criterion provides an equivalent description of left-most sequences. 

\begin{Lemma}[\cite{AlexThesis}, Section 2.3]\label{LeftMostCharakterisierungUntersequenzen}
 A sequence $I=(i_s, \ldots, i_1)\in F_n$ is left-most if and only if for every connected subsequence $J$ holds: If $a<b$ and $(a,b)\subset J$, then $I_a^b\subset J$. 
\end{Lemma}

The first author showed in his thesis that we can always find left-most, reduced representatives for elements in $P_n$: 

\begin{Prop}[\cite{AlexThesis}, Section 2.3] \label{AllLeftMost}
 Every sequence $I\in F_n$ is $\sim_P$-equivalent to a unique left-most, reduced one. 
\end{Prop}

\begin{remark}
 An analogous statement is proven by D.~Krammer(\cite{KrammerArtin}) via an appropriate rewriting system. 
\end{remark}

The following lemma shows that choosing such a representative preserves existence of certain subsequences:

\begin{Lemma} \label{LeftMostSubsequence1n}
 Let $J\in F_n$ be a sequence with $I_{1}^n\subset J$. Let furthermore $J'$ be a left-most, reduced sequence with $J'\sim_P J$. Then $I_{1}^{n} \subset J'$ holds. 
\end{Lemma}

\begin{pf}
This follows directly as one observes that whenever $(i, i+1)$ is a (possibly disconnected) subsequence of some word in $F_n$, those entries can never be interchanged using relations in $P_n$.
\end{pf}

This finishes for the present the list of general properties of $P_n$ and $Q_n$ which will be used later. 

Now we exhibit three statements which will be important for the main lemma mentioned above.

\begin{enumerate}[label=\textbf{\textit{{\Alph*(n)}}}]
\item Let $I$ be a sequence in $F_n$ with $I=1\wt{X}n$ and $\wt{X}=\sh_1(X)$ for some $X\in F_{n-2}$ (in other words, $1$ and $n$ do not occur in $\wt{X}$). Furthermore, we assume $I_1^n\subset I$. Then there are sequences $J, K\in F_{n-2}$ such that $I \sim \sh_2(J) I_1^n \sh_1(K)$. 
\end{enumerate}

This can be rephrased as follows: In $Q_n'$, if a representing sequence $I$ starts with $1$, ends with $n$, contains all other letters in increasing order in between (not necessarily as a connected subsequence), and no $1$'s and $n$'s occur in between, then $I$ is equivalent to another representative which contains $I_1^n$ as a connected subsequence, only entries greater than $2$ left to $I_1^n$ and entries between $2$ and $n-1$ on the right of it.

\begin{enumerate}[label=\textbf{\textit{{\Alph*(n)}}}, resume]
\item Let $Z$ be a sequence in $F_{n-1}$ and set $\wt{Z}=\sh_1{Z}$. Then we have $1\wt{Z}I_1^{n} \sim \wt{Z}I_1^{n}$. 
 \item Let $L$ be any sequence in $F_n$ and let $M$ be a sequence obtained from $L$ by deleting all $1$'s in $L$. Then $LI_1^{n} \sim MI_1^{n}$. 
\end{enumerate}

Intuitively, $C(n)$ tells us that the sequence $I_1^n$ ``swallows'' all the $1$'s left from it when considered as an element of $Q_n'$. This is a powerful statement for finding nice representatives of elements in $Q_n'$.

Now the main lemma for noetherianity is as follows:  
\begin{Theorem} \label{AnBnCn}
For all $n$, the statements $A(n)$, $B(n)$, $C(n)$ hold.
\end{Theorem}

We subdivide the proof in several lemmas. Basically, this is a somewhat involved induction argument. 

\begin{Lemma}
The statement $B(n)$ implies $C(n)$. 
\end{Lemma}

\begin{pf}
We can write $L$ from $C(n)$ as 
\begin{eqnarray*}
L=\sh_1(L_m)1\sh_1(L_{m-1})1\ldots \sh_1(L_1)1\sh_1(L_0)
\end{eqnarray*}
with $L_i\in F_{n-1}$. Now we can iteratively apply $B(n)$ with $Z=L_iL_{i-1}\ldots L_1L_0$ to delete the $1$'s.
\end{pf}

\begin{Lemma}
The statements $A(n)$ are true for $1\leq n\leq 3$. The statement $B(1)$ is also true.
\end{Lemma}

\begin{proof}
For $n=1$ and $n=2$, $X$ in $A(1)$ and $A(2)$ has to be the empty sequence, so $I=I_1^n$ and $A(1)$ and $A(2)$ hold. Similarly, for $n=1$, the string $Z$ in $B(1)$ is empty, so $B(1)$ holds. For $n=3$, the only non-empty case of $A(3)$ is $X=1$ (since we can apply idempotency relation in $P_1$). Then $I=(1\, 2\, 3)$, and the statement $A(3)$ obviously holds.
\end{proof}

\begin{Lemma}
Assume $A(k)$ and $B(k)$ hold for all $1\leq k\leq n-1$. Then the statement $B(n)$ holds.
\end{Lemma}

\begin{pf}

Let $Z$ be in $F_{n-1}$ and set $\wt{Z}=\sh_1{Z}$. We would like to show $1\wt{Z}I_1^{n} \sim \wt{Z}I_1^{n}$.

Without loss of generality, we may assume that $Z$ is left-most due to Proposition \ref{AllLeftMost}. 

 If $Z$ does not contain a $1$, all entries of $\wt{Z}$ are at least $3$ and $1\wt{Z} \sim_{P}\wt{Z}1$, so we conclude 
\begin{eqnarray*}
 1\wt{Z}I_1^{n} \sim_P \wt{Z}1I_1^{n} \sim_P \wt{Z}I_1^{n}
\end{eqnarray*}
as $I_1^{n}$ starts with a $1$. 
Now assume that $Z$ does contain a $1$. Then we can write $Z=Z'Z''$ with $Z'\in F_{n-1}$ not containing a $1$ and $Z''\in F_{n-1}$ starting with a $1$. If $\wt{Z'}=\sh_1(Z')$ and $\wt{Z''}=\sh_1(Z'')$, then $\wt{Z'}$ again contains only entries commuting with $1$, so it is enough to show that $1\wt{Z''}I_1^{n} \sim \wt{Z''}I_1^n$ since then we have
\begin{eqnarray*}
 1\wt{Z}I_1^{n}=1\wt{Z'}\wt{Z''}I_1^{n} \sim_P \wt{Z'}1\wt{Z''}I_1^{n} \sim \wt{Z'}\wt{Z''}I_1^{n}=\wt{Z}I_1^{n}.
\end{eqnarray*}
Let $r$ be the maximal entry occurring in $Z''$; by assumptions $r\leq n-1$. By the definition of $r$, we know that $Z'' \in F_{r}$. There must be a connected subsequence of $Z''$ starting with $1$, ending with $r$ and not containing $1$ or $r$ in between: Find the first occurrence of $r$ in $Z''$ starting from the left, then there is at least one $1$ left from it since $Z''$ starts with $1$; from all the $1$'s left to this $r$ take the one on the very right. So we may write $Z''$ as $Z''=U 1 \wt{V} r W$ with $U, W\in F_{r}$, $\wt{V}=\sh_1(V)$ with $V\in F_{r-2}$. Since $1\wt{V}r$ is a connected subsequence of $Z''$ and $Z''$ is a connected subsequence of the left-most sequence $Z$, we conclude that $1\wt{V}r$ is left-most. This implies $I_1^{r} \subset 1\wt{V}r$ by Lemma \ref{LeftMostCharakterisierungUntersequenzen}.
 Then we may apply $A(r)$ and get two sequences $J, K \in F_{r-2}$ such that $1\wt{V}r\sim \sh_2(J) I_1^{r} \sh_1(K)$. Therefore, $Z'' \sim U \sh_2(J) I_1^{r} \sh_1(K) W$. 

Observe that all entries of $\sh_2(J)$ are at least $3$. By $C(r)$, which holds since we assumed $B(r)$, we may also change $U$ to $\wt{U'}$ by deleting all $1$'s in $U$ so that $\wt{U'}=\sh_1(U')$ for some $U'\in F_{r-1}$. 

Now we put together what we have done so far: 
\begin{eqnarray*}
 1\sh_1(Z'')I_1^{n}\sim 1\sh_2(U')\sh_3(J) I_2^{r+1}\sh_2(K)\sh_1(W) I_1^ {r+1} I_{r+2}^{n}.
\end{eqnarray*}
Consider now the product of $\sh_2(D_{r-1})$ with the left-hand side. It is obvious that $\sh_2(D_{r-1})$ commutes with $1$ by $\sim_P$ and the $1$ does not occur in $\sh_2(D_{r-1})$. It is enough to show that 
\begin{eqnarray*}
\sh_2(D_{r-1})1\sh_2(U')\sh_3(J) I_2^{r+1}\sh_2(K)\sh_1(W) I_1^ {r+1} I_{r+2}^{n} \sim_Q \\
\sh_2(D_{r-1})\sh_2(U')\sh_3(J) I_2^{r+1}\sh_2(K)\sh_1(W) I_1^ {r+1} I_{r+2}^{n}
\end{eqnarray*}
since that by definition of $\sim$ in Definition \ref{DefinitionMonoidQn} implies 
\begin{eqnarray*}
1\sh_2(U')\sh_3(J) I_2^{r+1}\sh_2(K)\sh_1(W) I_1^ {r+1} I_{r+2}^{n} \sim \\
\sh_2(U')\sh_3(J) I_2^{r+1}\sh_2(K)\sh_1(W) I_1^ {r+1} I_{r+2}^{n}.
\end{eqnarray*}
Now we have 
\begin{eqnarray*}
\sh_2(D_{r-1})1\sh_2(U')\sh_3(J) I_2^{r+1}\sh_2(K)\sh_1(W) I_1^ {r+1} I_{r+2}^{n} \sim_P \\
1\sh_2(D_{r-1}U'\sh_1(J)) I_2^{r+1}\sh_2(K)\sh_1(W) I_1^ {r+1} I_{r+2}^{n} \sim_Q\\
1\sh_2(D_{r-1})I_2^{r+1}\sh_2(K)\sh_1(W)I_1^ {r+1} I_{r+2}^{n},
\end{eqnarray*}
where the first equivalence follows from the fact that $\sh_2(D_{r-1})$ has entries greater or equal to $3$, and the second one from the absorption property \ref{Dn} of $D_{r-1}$ in $Q_{r-1}$, where $U'$ and $\sh_1(J)$ define elements in. 

Observe now that we can write $\sh_2(D_{r-1})I_2^{r+1}$ as $\sh_1(D_r)$. Again by Theorem \ref{Dn}, we have 
\begin{eqnarray*}
\sh_1(D_r\sh_1(K) W) \sim_Q \sh_1(D_r)
\end{eqnarray*}
since $\sh_1(K), W\in F_r$ represent elements in $Q_r$. Together this yields
\begin{eqnarray*}
1\sh_2(D_{r-1})I_2^{r+1}\sh_2(K)\sh_1(W)I_1^ {r+1} I_{r+2}^{n} & \sim_Q  & 1\sh_1(D_r)I_1^ {r+1} I_{r+2}^{n}\\
 & = & 1D_{r+1} I_{r+2}^n \sim_Q D_{r+1} I_{r+2}^n,
\end{eqnarray*}
where we again used the absorption property in the last step. 

Using the same arguments, one observes that also 
\begin{eqnarray*}
\sh_2(D_{r-1})\sh_2(U')\sh_3(J) I_2^{r+1}\sh_2(K)\sh_1(W) I_1^ {r+1} I_{r+2}^{n}\sim_Q D_{r+1} I_{r+2}^n
\end{eqnarray*}
 implying the claim $B(n)$. 
 
 \end{pf}
 
\begin{Lemma}
Assume $A(k)$ and $B(k)$ hold for all $1\leq k\leq n-1$. Then the statement $A(n)$ holds.
\end{Lemma} 
 
 \begin{pf}
 Suppose we have $X\in F_{n-2}$ such that $I=1\wt{X}n$ contains $I_1^{n}$ as a (possibly) disconnected subsequence with $\wt{X}=\sh_1(X)$. Without loss of generality, we may assume $X$ to be left-most; due to Lemma \ref{LeftMostSubsequence1n}, the property $I_1^{n-2}\subset X$ is preserved. Write $X$ as $X=\wt{X_t} 1 \wt{X_{t-1}} 1 \ldots 1\wt{X_2} 1\wt{X_1}$ where $\wt{X_i}=\sh_1(X_i)$, with $X_i\in F_{n-3}$. The sequences $X_t$ and $X_1$ are possibly empty, whereas $X_i$ can be assumed non-empty for $2\leq i\leq t-1$. (Observe that $X$ has to contain a $1$ since $I_1^{n}\subset I=1\wt{X}n$, and that the $X_i$ are left-most again.) Set $\wt{\wt{X_i}}=\sh_2(X_i)$. Then 
 \begin{eqnarray*}
 I=1\wt{\wt{X_t}}2\wt{\wt{X_{t-1}}}2\ldots 2\wt{\wt{X_2}}2\wt{\wt{X_1}}n.
 \end{eqnarray*} 
  Since $I_1^{n}\subset I$, there must be a smallest $k$ such that $n-1\in \wt{\wt{X_k}}$, i.e., $n-3\in X_{k}$. As we already showed $A(n)$ for $n\leq 3$, we may assume $n\geq 4$. Thus, $n$ commutes with $2$, and 
  \begin{eqnarray*}
  I\sim_P 1\wt{\wt{X_t}}2\wt{\wt{X_{t-1}}}2\ldots 2\wt{\wt{X_{k+1}}}2\wt{\wt{X_k}}n 2\wt{\wt{X_{k-1}}}2\wt{\wt{X_{k-2}}}2\ldots 2\wt{\wt{X_1}}.
  \end{eqnarray*}
Note that $k<t$ since $I_1^n\subset 1\wt{X} n$, so there must be an entry $n-1$ after the first appearance of $2$. Now we have $2\wt{\wt{X_k}}n=\sh_1(1\wt{X_k} (n-1))$, where $X_k\in F_{(n-1)-2}$. Moreover, $2\wt{\wt{X_k}}$ is left-most and contains $(2,n-1)$ by assumptions, thus it contains $I_{2}^{n-1}$ by Lemma \ref{LeftMostCharakterisierungUntersequenzen} and therefore $I_{1}^{n-1}\subset 1\wt{X_k} (n-1)$. So we may use $A(n-1)$ and obtain sequences $J', K' \in F_{n-3}$ such that $1\wt{X_k} (n-1)\sim \sh_2(J') I_{1}^{n-1} \sh_1(K')$. 
Putting this into the formula above leads to
\begin{eqnarray*}
 I\sim 1\wt{\wt{X_t}}2\wt{\wt{X_{t-1}}}2\ldots 2\wt{\wt{X_{k+1}}}\sh_3(J') I_{2}^{n} \sh_2(K')2\wt{\wt{X_{k-1}}}2\wt{\wt{X_{k-2}}}2\ldots 2\wt{\wt{X_1}}.
\end{eqnarray*}
Use now $\sh_1(C(n-1))$ to see 
\begin{eqnarray*}
  I\sim 1\wt{\wt{X_t}}\wt{\wt{X_{t-1}}}\ldots \wt{\wt{X_{k+1}}}\sh_3(J') I_{2}^{n} \sh_2(K')2\wt{\wt{X_{k-1}}}2\wt{\wt{X_{k-2}}}2\ldots 2\wt{\wt{X_1}}.
\end{eqnarray*}
Since all entries in $\wt{\wt{X_t}}\wt{\wt{X_{t-1}}}\ldots \wt{\wt{X_{k+1}}}\sh_3(J')$ are at least $3$, this term commutes with $1$ and we obtain
\begin{eqnarray*}
 I\sim \sh_2(X_tX_{t-1}\ldots X_{k+1}\sh_1(J')) I_{1}^{n} \sh_1(\sh_1(K')1\wt{X_{k-1}}1\wt{X_{k-2}}1\ldots 1\wt{X_1}).
\end{eqnarray*}
Since $J:=X_tX_{t-1}\ldots X_{k+1}\sh_1(J')$ is in $F_{n-2}$ as well as 
\begin{eqnarray*}
 K:=\sh_1(K')1\wt{X_{k-1}}1\wt{X_{k-2}}1\ldots 1\wt{X_1},
\end{eqnarray*}
 this completes the proof of the lemma.
\end{pf}

\begin{proof}[Proof of Theorem \ref{AnBnCn}]
 follows now by induction assembling the lemmas above.
\end{proof}

This leads us to the following corollary.

\begin{Cor} \label{AnwendungenI1nNormalForm}
 In $Q_n'$, the equality $I_1^nI_1^n=I_2^nI_1^n$ follows from $C(n)$ with $L=I_{1}^{n}$. Applying the statement $C(n)$ iteratively, we obtain furthermore
\begin{eqnarray*}
 (I_1^n)^n \sim I_n^nI_{n-1}^n\ldots I_{2}^nI_1^n=D_n.
\end{eqnarray*}

\end{Cor}

\begin{remark}
 One can use the theorem above to show that $Q_n'$ is finite for all $n$; for details, we refer to \cite{MyThesis}. Contrary to this, the monoid $Q_n$ is infinite for $n\geq 3$, cf.\ \cite{AlexThesis} or \cite{KrammerArtin}. 
\end{remark}

\section{Factorable Monoids with Complete Rewriting Systems}
\label{Factorable Monoids with Complete Rewriting Systems}

Recall that for the monoid $Q_n$, the Evaluation Lemma \ref{EvaluationLemmaAlex} holds, i.e., for any tuples $I, J\in F_n$ with $I \sim_Q J$, we know that we have the equality $f_I\equiv f_J$ in the graded sense. Unfortunately, only weak analogs are true for the monoid $Q_n'$. However, they are sufficient to provide some examples of complete rewriting systems. So we will show a suitable Evaluation Lemma for $Q_n'$. First, we will prove a variant of the Evaluation Lemma \ref{EvaluationLemmaAlex} for the original monoid $Q_n$. 

\begin{Lemma}\label{EvaluationLemmaAlexUngraduiert}
  Let $(M, \mathcal{E}, \eta)$ be a factorable monoid satisfying the stronger conditions $(xs)'=(x's)'$ and $\overline{xs}=\overline{x}\cdot \overline{x's}$ for any $x\in M$ and $s\in \E$. Then 
  \begin{eqnarray*}
  f_1f_2f_1f_2=f_2f_1f_2=f_2f_1f_2f_1
  \end{eqnarray*}
   holds for all triples in this monoid. Recall that for a sequence $I=(i_s, \ldots, i_1) \in F_n$, we defined $f_I\colon M^{n+1} \to M^{n+1}$ to be the composition $f_{i_s}\circ f_{i_{s-1}}\circ \ldots \circ f_{i_1}$. If $I \sim_Q J$ and $(x_{n+1},x_n, \ldots, x_1)$ is a tuple of elements $x_i \in \mathcal{E}$, then 
\begin{eqnarray*}
 f_I(x_{n+1}, \ldots, x_1)=f_J(x_{n+1}, \ldots, x_1).
\end{eqnarray*}
Note that this equality holds not only in the graded sense.
\end{Lemma}

\begin{pf}
 Note that we can rewrite the assumption $(xs)'=(x's)'$ and $\overline{xs}=\overline{x}\cdot \overline{x's}$ for any $x\in M$ and $s\in \E$ as $\eta_1d_1=d_2\eta_1d_1\eta_2\colon M\times \E\to M\times M$, where we use Notation \ref{NotationAlphaI}. For the proof, we have first to show that $\eta_1d_1=d_2\eta_1d_1\eta_2$ also holds for all pairs in $M\times M$. This works exactly as the proof of the graded equality in Lemma \ref{WeakFactRecognImpliesFact} by deleting the word ``graded''. The second part of the proof works by applying the same method to the proof of Lemma 2.2.5 in \cite{AlexThesis}. 
\end{pf}

Now we are ready to prove the Evaluation Lemma for $Q_n'$.

\begin{Lemma} \label{EvaluationLemmaQnStrich}
  Let $(M, \mathcal{E}, \eta)$ be a factorable monoid satisfying the stronger conditions $(xs)'=(x's)'$ and $\overline{xs}=\overline{x}\cdot \overline{x's}$ for any $x\in M$ and $s\in \E$. If $I \sim J$ and $(x_{n+1},x_n, \ldots, x_1)$ is a tuple of elements $x_i \in \mathcal{E}$, then 
\begin{eqnarray*}
 f_I(x_{n+1}, \ldots, x_1)=f_J(x_{n+1}, \ldots, x_1).
\end{eqnarray*}
\end{Lemma}

\begin{pf}
 It is enough to show this for the defining relation of $\sim$. So let $U,V \in F_n$ be two words in letters $1,\ldots, n$ and $1\leq k\leq n$ such that $kU\sim_P Uk$, the letter $k$ does not occur in $U$ and such that $kUV\sim_Q UV$. We want to show that $f_kf_V$ and $f_V$ are equal evaluated on each $(x_{n+1}, x_n,\ldots, x_1)$.  We have only to show that under these conditions, $f_V(x_{n+1}, x_n,\ldots, x_1)$ is stable at the position $k$. Note that since $k$ does not occur in $U$, the only possibility for $kU\sim_P Uk$ to be true is that neither $k+1$ nor $k-1$ occur in $U$. So applying $f_U$ does not affect the letters in the $k$-th and the $k+1$-st place. Hence, at these places, $f_Uf_V(x_{n+1}, x_n,\ldots, x_1)$ and $f_V(x_{n+1}, x_n,\ldots, x_1)$ have equal entries. Thus, the same holds for $f_kf_Uf_V(x_{n+1}, x_n,\ldots, x_1)$ and  $f_kf_V(x_{n+1}, x_n,\ldots, x_1)$. But $kUV\sim_Q UV$, so we can apply the Evaluation Lemma \ref{EvaluationLemmaAlexUngraduiert}, which now holds in the proper and not 
only in the graded sense.
\end{pf}

This gives us a (quite restrictive) sufficient condition for the completeness of the associated rewriting system. 

\begin{Theorem}\label{CompleteRewritingSystemWeak}
 Let $(M, \mathcal{E}, \eta)$ be a factorable monoid satisfying the stronger conditions $(xs)'=(x's)'$ and $\overline{xs}=\overline{x}\cdot \overline{x's}$ for any $x\in M$ and $s\in \E$. Then the associated string rewriting system is complete. 
\end{Theorem}

\begin{pf}
We only have to prove the noetherianity of the rewriting system. Call a sequence of rewritings \textit{effective} with respect to a tuple $\underline{x}=(x_{n+1}, x_n, \ldots, x_1)$ in $\E^{n+1}$ if every application of a rewriting rule in this sequence changes the tuple. Recall furthermore that a finite sequence of rewritings corresponds uniquely to a tuple in $F_n$.

We will show the following statement: For any $n \in \mathbb{N}$, the length of tuples in $F_n$ which are effective for a fixed $\underline{x}$ is bounded by a number $c(n)$. This will yield the claim of the lemma.

Note that any effective sequence has to be reduced. In particular, we see that $c(1)=1$ is a desired upper bound. For $n=2$, we know that both sequences $(1 \, 2 \, 1\, 2)$ and $(2 \, 1 \, 2 \, 1)$ yield the normal form, so that the effective sequences are bounded by $c(2)=4$.  

Now assume the statement is proven for all natural numbers smaller than $n$, and let $\underline{x}=(x_{n+1}, x_n, \ldots, x_1)\in \E^{n+1}$ be any tuple of length $n+1$.  

Let $I \in F_n$ be effective with respect to $\underline{x}$. If $I$ does not contain the entry $n$, we are done by induction hypothesis. Otherwise, we can find the first occurence of $n$ in $I$ from the right, and so we have $I=J_1 \; n \; J_0$ for some $J_0 \in F_{n-1}$. Note that $J_0$ is effective with respect to $(x_n, \ldots, x_1)$, so by induction hypothesis, its length is bounded by $c(n-1)$. If $J_1$ does not contain a $1$, then it is in $\sh_1(F_{n-1})$, and the corresponding sequence in $F_{n-1}$ is effective with respect to the left $n$-subtuple of $f_nf_{J_0}(\underline{x})$. Thus, if $J_1$ does not contain a $1$, the length of $I$ is bounded by $2c(n-1)+1$. 

Now assume $J_1$ does contain a $1$. Then we can write $I$ either as 
\begin{eqnarray*}
K_2 \; 1 \; K_1 \; n \; K_0 \; n \; J_0
\end{eqnarray*}
or as
\begin{eqnarray*}
K_2 \; 1 \; K_1 \; n \; J_0
\end{eqnarray*}
with $K_1 \in \sh_1(F_{n-2})$ and $K_0 \in \sh_1(F_{n-1})$, depending on whether the right-most occurence of a $1$ in $J_1$ is separated from the already found $n$ by another occurence of $n$ or not.

We claim that in either case, the sequence $K_2$ cannot contain a $1$. Set $\underline{y}=f_{K_0}f_nf_{J_0}(\underline{x})$ in the first and $\underline{y}=f_{J_0}(\underline{x})$ in the second case. Then the sequence $K_2 \; 1 \; K_1 \; n$ is effective with respect to $\underline{y}$. Moreover, by statement $A(n)$ of Theorem \ref{AnBnCn}, we know that there are sequences $L, M \in F_{n-2}$ so that 
\begin{eqnarray*}
1 \; K_1 \; n \sim \sh_2(L) I_1^n \sh_1(M).
\end{eqnarray*}
By Evaluation Lemma \ref{EvaluationLemmaQnStrich}, this implies under the assumptions of the theorem that 
\begin{eqnarray*}
f_1f_{K_1}f_n(\underline{y})=f_{\sh_2(L)} f_{I_1^n} f_{\sh_1(M)}(\underline{y}).
\end{eqnarray*} 
Assume for contradiction that $K_2$ contains a $1$, so we can write $K_2= Z_2 \; 1 \; Z_1$ with $Z_2\in F_n$ and $Z_1 \in \sh_1(F_{n-1})$ (i.e., we consider the first occurence of $1$ from the right in $K_2$). Then by statement $C(n)$ of Theorem \ref{AnBnCn}, we can conclude 
\begin{eqnarray*}
1\; Z_1 \sh_2(L) I_1^n \sh_1(M) \sim  Z_1 \sh_2(L) I_1^n \sh_1(M). 
\end{eqnarray*} 
Thus, using Evaluation Lemma \ref{EvaluationLemmaQnStrich} again, we obtain
\begin{eqnarray*}
f_1f_{Z_1}f_{\sh_2(L)} f_{I_1^n} f_{\sh_1(M)}(\underline{y})=f_{Z_1}f_{\sh_2(L)} f_{I_1^n} f_{\sh_1(M)}(\underline{y}),
\end{eqnarray*}
contradicting the effectiveness of the sequence $I$. 

So we have proven that $K_2$ is contained in $\sh_1(F_{n-1})$. Again by arguing with appropriate subsequences of $I$ and tuples obtained from $\underline{x}$, we see that the sequences $K_2, K_1, K_0$ are effective with respect to some tuples, so that their lengths are bounded by $c(n-1), c(n-2), c(n-1)$, respectively.  

All in all, we have shown that the length of $I$ is bounded by 
\begin{eqnarray*}
c(n)=3c(n-1)+c(n-2)+3. 
\end{eqnarray*}
This completes the proof of the theorem. 
\end{pf}

\begin{remark}
Proposition 3.4.4 of \cite{AlexThesis}, stating a stronger version of our Lemma \ref{EvaluationLemmaQnStrich}, is unfortunately wrong. It cannot be proved without further assumptions, as the counterexample in the appendix shows. Thus, the proof of noetherianity for the rewriting system associated with a factorable monoid cannot be fixed in general. However, it can be fixed in special cases, e.g. as in Theorem \ref{CompleteRewritingSystemWeak} above. We will present more concrete examples later on.  
\end{remark}

\section{Discrete Morse Theory Basics}
\label{sec:discrete morse theory}

We want to show that even though the rewriting system associated with a factorability structure fails to be complete in general, an analogon of Brown's Theorem \ref{TheoremBrown} is true for any factorable monoid. To do so, we will need techniques from discrete Morse theory. We introduce an algebraic version here. Again, we follow the exposition in \cite{AlexThesis}.

	A \textbf{based chain complex} is a non-negatively graded chain complex $( {C}_*, \partial)$, where each $ {C}_n$ is a free $\Z$-module, together with a choice of basis $\Omega_n$ for each $ {C}_n$.
	In what follows, $( {C}_*, \Omega_*, \partial)$ will always be a based chain complex.

	We equip each $ {C}_n$ with the inner product $\langle \variable , \variable \rangle\colon  {C}_n \times  {C}_n \to \Z$ obtained by regarding $\Omega_n$ as an orthonormal basis for ${C}_n$.
	For elements $x \in  {C}_{n}$ and $y \in  {C}_{n-1}$, we say that $\langle \partial x, y \rangle$ is their \textbf{incidence number}.
	If $x$, $y$ have the ``wrong'' dimensions, i.e., if $x\in C_n$, but $y\notin C_{n-1}$, then we set their incidence number $\langle \partial x, y \rangle$ to be zero.

	\begin{Definition}
		A $\mathbb{Z}$-compatible \textbf{matching} on a based chain complex \linebreak $( {C}_*, \Omega_*, \partial)$ is an involution $\mu\colon \Omega_* \to \Omega_*$ satisfying the following property: For every $x \in \Omega_*$ which is not a fixed point of $\mu$, we have $\langle \partial x, \mu(x) \rangle = \pm 1$ or $\langle \partial \mu(x), x \rangle = \pm 1$. (This last condition is called $\mathbb{Z}$-compatibility.)
		
		The fixed points of a matching $\mu\colon \Omega_* \to \Omega_*$ are called \textbf{essential}. If $x \in \Omega_n$ is not a fixed point, then $\mu(x) \in \Omega_{n-1} \cup \Omega_{n+1}$. We say that $x$ is \textbf{collapsible} if $\mu(x) \in \Omega_{n-1}$, and it is called \textbf{redundant} if $\mu(x) \in \Omega_{n+1}$.
	\end{Definition}

	\begin{remark}\label{Zcompatibility}
	 Note that it is enough to check $\Mf{\partial\mu(x), x}=\pm 1$ for redundant cells in order to check that an involution $\mu\colon \Omega_* \to \Omega_*$ is $\mathbb{Z}$-compatible if we know that all non-fixed points are either collapsible or redundant. Indeed, let $x \in \Omega_n$ be a non-fixed point of an involution $\mu$ as above. We have to show that $\Mf{\mu(x), \partial x}=\pm 1$ for the case that $x$ is collapsible. In this case, the image $\mu(x)\in \Omega_{n-1}$ is redundant since $\mu(\mu(x))=x$ is in $\Omega_n$. So we know that for $y=\mu(x)$, we have $\Mf{\partial\mu(y), y}=\pm 1$. Inserting $y=\mu(x)$, we obtain $\Mf{\mu(x), \partial x}=\pm 1$. 
	\end{remark}

	Let $\mu$ be a matching on $( {C}_*, \Omega_*, \partial)$. For two redundant basis elements $x$, $z \in \Omega_*$ set $\mathbf{x \vdash z}$ to be the relation ``$z$ occurs in the boundary of the collapsible partner of $x$'', i.e., $\Mf{\partial \mu(x), z} \neq 0$.

	\begin{Definition}\label{AlgebraicMatchingNoetherian}
		A matching on a based chain complex is called \textbf{noetherian} if every infinite chain $x_1 \vdash x_2 \vdash x_3 \vdash \ldots$ eventually stabilizes.
	\end{Definition}
	
	Given a noetherian matching $\mu$ on $( {C}_*, \Omega_*, \partial)$, we define a linear map $\theta^\infty\colon  {C}_* \to  {C}_*$ as follows. Let $x \in \Omega_*$. If $x$ is essential, we set $\theta(x) = x$. If $x$ is collapsible, we set $\theta(x) = 0$, and if $x$ is redundant we set $\theta(x) = x - \varepsilon \cdot \partial\mu(x)$, where $\varepsilon = \Mf{\partial \mu(x),x}$.
	
	Note that if $x$ is redundant, then $\langle x, \theta(x) \rangle = 0$. It is now not hard to check that for every $x \in \Omega_*$ the sequence $\theta(x), \theta^2(x), \theta^3(x), \ldots$ stabilizes for noetherian matchings (cf.\ also \cite{AlexThesis}, Section 1.1), and we define $\theta^\infty(x) := \theta^N(x)$ for $N$ large enough. We linearly extend this map to obtain $\theta^\infty\colon  {C}_* \to  {C}_*$.
	
	We can now state the main theorem of discrete Morse theory.
	
	\begin{Theorem}[Brown, Cohen, Forman]\label{BCF}
		Let $( {C}_*, \Omega_*, \partial)$ be a based chain complex and let $\mu$ be a noetherian matching. Denote by 
		\begin{eqnarray*}
		 {C}_*^\theta = \im(\theta^\infty\colon  {C}_* \to  {C}_*)
\end{eqnarray*}		
 the $\theta$-invariant chains. Then $( {C}_*^\theta, \theta^\infty \circ \partial|_{\im(\theta^{\infty})})$ is a chain complex, and the map
		\begin{align*}
			\theta^\infty\colon ( {C}_*, \partial) \longrightarrow \left( {C}_*^\theta, \theta^\infty \circ \partial|_{\im(\theta^{\infty})}\right)
		\end{align*}
		is a chain homotopy equivalence. A basis of $ {C}_*^{\theta}$ is given by the essential cells. 
	\end{Theorem}
	
	For a proof see e.g. \cite{Forman}.

 \begin{remark} \label{RemarkBrownTheorem}
  Discrete Morse theory is the key tool for Brown's proof of Theorem \ref{TheoremBrown}: On a monoid $M$ given by a complete rewriting system $(S, \mathcal{R})$, he constructs a noetherian matching on the bar complex with essential cells as described in Theorem \ref{TheoremBrown}. 
 \end{remark}

The differential of the Morse complex can be described more explicitly. Given an essential cell $x \in \Omega_n$, we want to determine the coefficient of an essential cell $y\in \Omega_{n-1}$ in the boundary of $x$ (in the Morse complex). For this, consider the set $Z(x,y)$ of all chains  
\begin{eqnarray*}
z_r \vdash z_{r-1}\vdash \ldots \vdash z_1
\end{eqnarray*} of redundant cells in $\Omega_{n-1}$ so that $\Mf{\partial x, z_r}\neq 0$ and also $\Mf{\partial\mu(z_1), y}\neq 0$. So we can formulate the following theorem: 

\begin{Theorem}[\cite{Kozlov}, Section 11.3]\label{MorseKomplexExplizit}
Let $( {C}_*, \Omega_*, \partial)$ be a based chain complex and let $\mu$ be a noetherian matching on it. Then $C_*$ is chain homotopy equivalent to a chain complex with bases given by $\mu$-essential cells and with differentials $\partial^{\mu}(x)$ for $x$ essential equal to
\begin{eqnarray*}
\sum_{y\in \Omega_n \mbox{\footnotesize{ess.}}}\sum_{\substack{(z_r,\ldots, z_1)\in \\  Z(x,y)}} (-1)^r\frac{\Mf{\partial x, z_r} \Mf{\partial \mu(z_r), z_{r-1}} \Mf{\partial \mu(z_{r-1}), z_{r-2}}\ldots \Mf{\partial \mu(z_1), y}}{\Mf{\partial \mu(z_r), z_{r}}\Mf{\partial \mu(z_{r-1}), z_{r-1}}\ldots \Mf{\partial \mu(z_1), z_{1}}}y.
\end{eqnarray*} 
\end{Theorem}

\begin{remark}
Both Morse complexes described above can be seen to coincide by an analogous argument as in \cite{Forman}, Section 8. 
\end{remark}

\section{Matching for Factorable Monoids}
\label{Matching for Factorable Monoids}                                      
	
We are now going to construct a matching on the reduced, inhomogeneous bar complex of a factorable monoid. The idea is to make a similar construction as in Theorem \ref{TheoremBrown} and Remark \ref{RemarkBrownTheorem}. Although the rewriting system defined by the factorability structure is in general not noetherian, the matching defined by the same method turns out to be noetherian nevertheless. 

We will need to analyze some elements of $Q_n$ (cf.\ Definition \ref{DefinitionMonoidQn}) in order to prove the noetherianity. We will do this first.
 
\begin{Definition}
An element $I$ of $F_n$ is called \textbf{small} if it is not equivalent in $Q_n$ to an element of the form
\begin{eqnarray*}
(J \; k+1 \; k \; k+1 \; K)
\end{eqnarray*} 
for any $1\leq k\leq n$ and any $J, K\in F_n$. This is the same as to say that $I$ is not equivalent to an element of this form in $P_n$.
\end{Definition}

The following finiteness result will be crucial for the noetherianity.

\begin{Lemma}\label{FinitenessSmallSequences}
In $F_n$, there are only finitely many reduced, right-most, small sequences. (Right-most is defined dually to left-most, cf.\ Definition \ref{DefinitionLeftMost})
\end{Lemma}

\begin{proof}
Observe that connected subsequences of small sequences are small again.

First, we are going to prove that a reduced, right-most, small sequence has a unique maximal entry. This is clear for $n=1$. We proceed by induction on $n$. Assume we have proven the claim for all natural numbers smaller or equal to $n$, and consider a sequence as above with maximal entry $n+1$ (if the maximal entry is smaller, the sequence is contained in $F_n \subset F_{n+1}$ and the conclusion holds by induction hypothesis). Recall that right-most is a condition depending on consecutive pairs of entries, so any connected subsequence of a right-most sequence is right-most again. Assume for contradiction that $n+1$ occurs more that once. Then we can choose a connected subsequence of the form 
\begin{eqnarray*}
(n+1 \; I \; n+1),
\end{eqnarray*}
where $I$ does not contain $n+1$. Note that both this subsequence and $I$ are reduced, right-most and small. Since the sequence $(n+1 \; I \; n+1)$ is right-most and $I$ does not contain the maximal entry $n+1$, we can conclude that $I=(n, I')$ for some sequence $I'$. Furthermore, $n$ has to be the maximal entry of $I$ and thus it does not occur in $I'$ by induction hypothesis.  But then all entries in $I'$ are smaller or equal to $n-1$, thus $I'$ and $n+1$ commute in $P_{n+1}$ and we have
\begin{eqnarray*}
(n+1 \; I \; n+1) \simeq_{Q_{n+1}} (n+1 \; n \; n+1 \; I'),
\end{eqnarray*}
contradicting the smallness. This completes the induction step. 

Thus, any reduced, right-most, small sequence has a unique maximal entry. Now we prove the claim of the lemma by induction. For $n=1$, the claim is obvious. Assume now that we have proven the claim for all natural numbers $\leq n$, so in particular, there are constants $c_i$ for $1\leq i\leq n$ bounding the length of reduced, right-most, small sequences in $F_i$. Now consider a reduced, right-most, small sequence in $F_{n+1}$. Either it does not contain the entry $n+1$, in which case it is contained in $F_n\subset F_{n+1}$ and has length at most $c_n$, or it contains exactly one entry $n+1$, in which case it is of the form $(I \; n+1 \; J)$ with $I$, $J$ reduced, right-most, small sequences which lie in $F_n$. So the total length of such a sequence as we started with in $F_{n+1}$ is at most $2c_n+1$. This shows that there are only finitely many of these and completes the proof of the lemma.
\end{proof}

We will need the following criterion for smallness. We write $J_a^b$ for the sequence $(b\; b-1\; \ldots \;  a)$ for natural numbers $b\geq a$ (cf.\ the ``dual'' notation \ref{NotationLeftMost} for the left-most case). Observe that we can write any right-most reduced sequence in a form $J_{\sms{a_l}}^{\sms{b_l}} J_{\sms{a_{l-1}}}^{\sms{b_{l-1}}}\ldots J_{\sms{a_1}}^{\sms{b_1}}$ with $a_r<b_{r-1}$.

\begin{Lemma}\label{CriterionSmallness}
Let $J_{\sms{a_l}}^{\sms{b_l}} J_{\sms{a_{l-1}}}^{\sms{b_{l-1}}}\ldots J_{\sms{a_1}}^{\sms{b_1}}$ with $a_r<b_{r-1}$ be a right-most, reduced, small sequence.
\begin{enumerate}
\item Let $j<b_l$. Then $j J_{\sms{a_l}}^{\sms{b_l}} J_{\sms{a_{l-1}}}^{\sms{b_{l-1}}}\ldots J_{\sms{a_1}}^{\sms{b_1}}$ is also right-most, reduced and small.
\item Let $j=b_l+1$ and let either $l=1$ or $b_l+1<b_{l-1}$. Then $j J_{\sms{a_l}}^{\sms{b_l}} J_{\sms{a_{l-1}}}^{\sms{b_{l-1}}}\ldots J_{\sms{a_1}}^{\sms{b_1}}$ is also right-most, reduced and small.
\item If $a_{l+1} \leq b_{l+1}<b_l$, then $J_{\sms{a_{l+1}}}^{\sms{b_{l+1}}}J_{\sms{a_l}}^{\sms{b_l}} J_{\sms{a_{l-1}}}^{\sms{b_{l-1}}}\ldots J_{\sms{a_1}}^{\sms{b_1}}$ is also right-most, reduced and small.
\end{enumerate}

\end{Lemma}

\begin{proof}
The new sequence is clearly reduced and right-most, so we only have to prove it is also small.

 If the right-most reduced sequence $J_{\sms{a_l}}^{\sms{b_l}} J_{\sms{a_{l-1}}}^{\sms{b_{l-1}}}\ldots J_{\sms{a_1}}^{\sms{b_1}}$ is in addition small, then any connected subsequence of it is also small, and a small, right-most, reduced sequence has a unique maximal entry, as we have seen in the proof of the last lemma. This is in particular true for each subsequence of the form $J^b_aJ^d_c$; if we had $b\geq d$, then $J^d_a$ were a connected subsequence of  $J^b_a$, and the subsequence $J^d_aJ^d_c$ of $J^b_aJ^d_c$ has the maximal value $d$ occurring twice.  

Thus, a right-most, reduced, small sequence can be written in form 
\begin{eqnarray*}
J_{\sms{a_l}}^{\sms{b_l}} J_{\sms{a_{l-1}}}^{\sms{b_{l-1}}}\ldots J_{\sms{a_1}}^{\sms{b_1}}
\end{eqnarray*}
  with $a_r<b_{r-1}$ and $b_r<b_{r-1}$ for all $2\leq r\leq l$. Now consider the sequence 
\begin{eqnarray*}
jJ_{a_l}^{b_l} J_{a_{l-1}}^{b_{l-1}}\ldots J_{a_1}^{b_1}.
\end{eqnarray*}
with $j<b_l$. Assume for contradiction that it is not small, i.e., there exists another representative of this element in $Q_n$ of the form $(I \, k+1 \, k \, k+1 \, J)$. Observe that such a representative of this element in $Q_{n}$ must be possible to obtain by applying relations in $P_{n}$ only, since for applying any additional relation of $Q_{n}$, we already have to have  a connected subsequence of the form $(k+1, k, k+1)$. So we want to show that applying relations in $P_n$ only, we will not produce a sequence with a connected subsequence of the form $(k+1, k, k+1)$.

The sequence we started with has the following property: Whenever there are two equal entries in different spots of this sequence, forming a connected subsequence of the form $m \, I \, m$, then either $I \, m \sim_P m \,   I$ (this actually does not occur in the sequence we started with, but we need to allow it in a moment) or $I$ contains the entry $m+1$. We want to show that this property is preserved under applying relations in $P_n$. This will complete the proof of the first part of the lemma. 

We have three types of applications of the relations in $P_n$: 
\begin{eqnarray*}
I\; a^2\; J &\to & I \; a \;  J\\
I \; a \;  J&\to & I\; a^2\; J\\
I \; a\;b \; J &\to & I \; b \; a \; J \mbox{ for } \abs{a-b}\geq 2.
\end{eqnarray*}
These applications of relations surely do not affect the property for pairs of equal entries which lie both in $I$ or both in $J$. Moreover, if one of the equal entries lies in $I$ and the other one in $J$, then the set of letters separating them does not change and so the property holds. Moreover, it is clear that the property holds for all appearances of $a$ in $I$ or $J$ after applying one of the idempotency relations; and it also holds for the new equal pairs possibly created by the relations. Last, we have to analyze the case where $a$ or $b$ appear in $I$ (the case for $J$ will follow symmetrically) and we apply the last relation. If $a$ appears in $I$, we just add a letter commuting with $a$ to the set separating both appearances of $a$, so it still has the property above. If $b$ appears in $I$, then we delete one letter commuting with $b$ from the set of letters separating both appearances of $b$, in particular, the deleted entry is not $b+1$; so the property holds. This completes the case 
distinction and also the proof of the first part of the lemma.

For the second part, observe that in that case, the $P$-invariant property above is satisfied again.    

The third part follows immediately from applying the first part to $j=a_{l+1}$ and then repeated application of the second part.
\end{proof}

An immediate consequence of the preceding lemma is the following.
\begin{Cor}\label{SmallnessCriterionJ}
Let $J_{\sms{a_l}}^{\sms{b_l}} J_{\sms{a_{l-1}}}^{\sms{b_{l-1}}}\ldots J_{\sms{a_1}}^{\sms{b_1}}$ with $a_r<b_{r-1}$ be a right-most, reduced sequence. Then it is small if and only if $b_{r+1}<b_r$ holds for all $1\leq r \leq l-1$. 
\end{Cor}

Now we can proceed to define the promised matching.

Let $(M, \mathcal{E}, \eta)$ be a factorable monoid. Let $(\mathbb{B}_*M, d)$ denote the reduced, inhomogeneous bar complex of $M$. Recall that it can be seen as a based chain complex with basis 
\begin{eqnarray*}
 \Omega_n=\left\{[m_n|\ldots|m_1]\,\right |\left.\, m_i\in M, m_i\neq 1\right\}
\end{eqnarray*}
in degree $n$, and the differential is given by 
\begin{eqnarray*}
 d([m_n|\ldots|m_1])=\sum_{i=0}^n (-1)^{i} d_i([m_n|\ldots|m_1]),
\end{eqnarray*}
where $d_i$ are given by 
\begin{eqnarray*}
  d_i([m_n|\ldots|m_1])=\begin{cases}
                         [m_n|\ldots |m_2], \mbox{ if } i=0,\\
[m_n| \ldots | m_{i+1}m_i | \ldots | m_1], \mbox{ if } 0<i<n,\\
[m_{n-1}| \ldots | m_1], \mbox{ if } i=n.
                        \end{cases}
\end{eqnarray*}
(Here, we define $d_i([m_n|\ldots|m_1])$ to be $0$ whenever it would contain $1$ as an entry otherwise.)

We want to specify a matching $\mu\colon \Omega_*\to \Omega_*$. Let $\NF\colon M \to \mathcal{E}^*$ denote again the normal form in the factorable monoid. Recall that by Lemma \ref{FactorabilityStructureInducesRewriting}, we have a rewriting system associated with the factorability structure $\eta$. In particular, the notion of reducible/irreducible words in $\mathcal{E}$ is well-defined. 

Let $[m_n|\ldots|m_1]$ be \textit{essential}, i.e., a fixed point of the matching $\mu$, if and only if the following conditions hold: 
\begin{enumerate}
 \item $m_1\in \mathcal{E}$,
 \item The word $(\NF(m_{i+1}), \NF(m_i))$ is reducible for every $1\leq i\leq n-1$,
 \item  For every $1\leq i\leq n-1$, any proper (right) prefix of\\
 $(\NF(m_{i+1}), \NF(m_i))$ is irreducible. 
\end{enumerate}

These conditions are analogous to those in Theorem \ref{TheoremBrown}. 

Observe that since all rewriting rules are of the form $(x,y)\to \eta(xy)$ for $x,y\in \mathcal{E}$, these conditions are equivalent to:
\begin{enumerate}
 \item $m_i\in \mathcal{E}$ for $1\leq i\leq n$,
 \item The pairs $(m_{i+1}, m_i)$ are unstable for all $1\leq i\leq n-1$. 
\end{enumerate}

Next, we want to describe the redundant and collapsible cells in $\Omega_*$. Define first the height of a cell as the length of the maximal ``right subcell'' which is essential, more precisely,
\begin{eqnarray*}
 \hgt([m_n|\ldots|m_1]):=\max\{k\in \mathbb{N}| [m_k|\ldots |m_1] \mbox{ is essential}\}.
\end{eqnarray*}
Note that the height of $[m_n|\ldots |m_1]$ is an integer between $0$ and $n$. Such a cell is of height $n$ if and only if it is essential. Consider now an $n$-cell $[m_n|\ldots |m_1]$ of height $h<n$. We call such a cell \textit{collapsible} if $(m_{h+1}, m_h)$ is stable. Otherwise, we call the cell $[m_n|\ldots |m_1]$
 \textit{redundant}. Observe that in this case, $m_{h+1}$ is not in $\mathcal{E}$, so it is of $\E$-length at least $2$. Note also that the cell is always redundant if it is of height $0$. Now define for a cell $\underline{m}:=[m_n|\ldots | m_1]\in \Omega_n$ the map $\mu$ as follows:
\begin{eqnarray*}
 \mu(\underline{m})=\begin{cases}
                          [m_n|\ldots |m_{h+1}m_h|m_{h-1}|\ldots |m_1], \mbox{ for } \underline{m} \mbox{ collapsible of height } h,\\
			  [m_n|\ldots |\overline{\eta}(m_{h+1})| \eta'(m_{h+1})|m_{h}|\ldots |m_1], \mbox{ for } \underline{m} \mbox{ redundant of height } h.
                         \end{cases}
\end{eqnarray*}

\begin{Prop}\label{MatchingofFactorableMonoid}
 Let $(M, \mathcal{E}, \eta)$ be a factorable monoid. Then the map $\mu\colon \Omega_*\to \Omega_*$ as above is a noetherian, $\mathbb{Z}$-compatible matching on the bar complex of $M$.  
\end{Prop}

\begin{pf}
 First, we observe that $\mu$ is well-defined: If $\underline{m}:=[m_n|\ldots | m_1]\in \Omega_n$ is collapsible of height $h$, then the pair $(m_{h+1}, m_h)$ is stable, so in particular, $m_{h+1}m_h\neq 1$. If $\underline{m}:=[m_n|\ldots | m_1]\in \Omega_n$ is redundant of height $h$, then $m_{h+1}$ is not in $\mathcal{E}$, so it has $\E$-length at least $2$. Thus, we know that $\eta'(m_{h+1})$ is in $\mathcal{E}$. Furthermore, the length of $\overline{\eta}(m_{h+1})$ is at least $1$, hence this element is also non-trivial. 

Next, we are going to check that $\mu$ is an involution. Let $\underline{m}:=[m_n|\ldots | m_1]\in \Omega_n$ be collapsible of height $h$. Then we know by assumption that the cell $[m_{h-1}|\ldots |m_1]$ is essential, so the cell $\mu(\underline{m})$ is of height at least $h-1$. Moreover, by Recognition Principle, the pair $(m_{h+1}m_h, m_{h-1})$ is unstable if and only if the pair $(m_h, m_{h-1})$ is unstable. The latter statement holds since the original cell was of height $h$. Thus, the cell $\mu(\underline{m})$ is redundant of height $h-1$. To compute $\mu^2(\underline{m})$, recall that we assumed
\begin{eqnarray*}
 \eta(m_{h+1}m_h)=(m_{h+1}, m_h).
\end{eqnarray*}
This implies $\mu^2(\underline{m})=\underline{m}$ for collapsible cells $\underline{m}$.

Consider now a redundant cell $\underline{m}:=[m_n|\ldots | m_1]\in \Omega_n$ of height $h$. Again, the cell $\mu(\underline{m})$ is at least of height $h$. Furthermore, we know that the pair $(m_{h+1},m_h)$ is unstable by assumption. By Theorem \ref{CharacterizationRecognitionPrinciple}, we may use Recognition Principle in $M$. Thus, we know that the pair $(\eta'(m_{h+1}), m_{h})$ is also unstable. Together with the observation that  $\eta'(m_{h+1})$ lies in $\mathcal{E}$ we can conclude that $\mu(\underline{m})$ is at least of height $h+1$. Since $\eta(m_{h+1})$ is a stable pair by definition, we obtain that $\mu(\underline{m})$ is a collapsible cell of height $h+1$, and $\mu^2(\underline{m})=\underline{m}$ follows immediately.

Now we are going to check that $\mu$ is a $\mathbb{Z}$-compatible matching. According to Remark \ref{Zcompatibility}, it is enough to check $\Mf{d\mu(\underline{m}),\underline{m}}=\pm 1$ for redundant cells $\underline{m}:=[m_n|\ldots | m_1]$. Assume $\underline{m}$ has height $h$. By definition, $d_{h+1}\mu(\underline{m})=\underline{m}$. We will show $d_i\mu(\underline{m})\neq \underline{m}$ for $i\neq h+1$, which will prove the $\Z$-compatibility. Furthermore, we will take a closer look at $d_i\mu(\underline{m})$, since it will be helpful for proving that the matching $\mu$ is noetherian. We have to consider several cases. Of course, only the cases with $d_i\mu(\underline{m})\neq 0$ are interesting for our consideration.

\renewcommand{\labelenumi}{Case \arabic{enumi}:}
\begin{enumerate}
 \item Consider $i=0<h+1$. If $h>0$, then 
 \begin{eqnarray*}
  d_i\mu(\underline{m})=[m_n|\ldots |\overline{\eta}(m_{h+1})| \eta'(m_{h+1})|m_{h}|\ldots |m_2].
 \end{eqnarray*}
 This tuple has $\overline{\eta}(m_{h+1})$ in $(h+1)$-st place from the right.  Thus, it cannot equal $\underline{m}$, since $\underline{m}$ has $m_{h+1}$ in this spot and 
\begin{eqnarray*}
 N_{\E}(\overline{\eta}(m_{h+1}))<N_{\E}(m_{h+1}).
\end{eqnarray*}
Note furthermore that $m_j$ for $2\leq j\leq h$ as well as $\eta'(m_{h+1})$ are in $\E$, and all of the pairs $(m_{k+1}, m_k)$ for $2\leq k\leq h-1$ are unstable by assumption. Moreover, the pair $(\eta'(m_{h+1}), m_h)$ is unstable since the pair $(m_{h+1}, m_h)$ is unstable (by Recognition Principle). So the cell  $d_i\mu(\underline{m})$ has at least height $h$. Since the pair $(\overline{\eta}(m_{h+1}), \eta'(m_{h+1}))$ is stable, this cell has height exactly $h$, and it is collapsible. 

If $\underline{m}$ is of height $0$, then $d_i\mu(\underline{m})=[m_n|\ldots |\overline{\eta}(m_{1})]$. Here, we apply the same argument as before to the right-most position of the tuple. Observe that this cell may or may not be redundant, and there are no constraints on its height. 

\item Assume now $0<i<n+1$, and consider the case $i<h$, in particular, $h>0$. Then we have
\begin{eqnarray*}
 d_i\mu(\underline{m})=[m_n|\ldots |\overline{\eta}(m_{h+1})| \eta'(m_{h+1})|m_{h}|\ldots |m_{i+1}m_i|\ldots |m_1].
\end{eqnarray*}
This tuple has $\overline{\eta}(m_{h+1})$ in the $(h+1)$-st spot, and as before, this entry is different from $m_{h+1}$. Note that the height of $d_i\mu(\underline{m})$ is at least $i-1<h$, and it may or may not be redundant. Moreover, if the application of $d_i$ does not lower the norm, then the tuple $ d_i\mu(\underline{m})$ has the height exactly $i-1$. 

\item The case $0<i=h<n+1$ is almost the same: 
\begin{eqnarray*}
 d_h\mu(\underline{m})=[m_n|\ldots |\overline{\eta}(m_{h+1})| \eta'(m_{h+1})m_{h}|m_{h-1}|\ldots |m_1].
\end{eqnarray*}
This tuple has again $\overline{\eta}(m_{h+1})$ in the $(h+1)$-st spot, and as before, this entry is different from $m_{h+1}$. The height of $d_i\mu(\underline{m})$ is at least $h-1<h$, and it may or may not be redundant. The height is exactly $h-1$ if the application of $d_h$ does not lower the norm of the tuple.

\item Next, we deal with the case $0<i=h+2<n+1$.  We have in this case
\begin{eqnarray*}
 d_{h+2}\mu(\underline{m})=[m_n| \ldots | m_{h+3}| m_{h+2}\overline{\eta}(m_{h+1})| \eta'(m_{h+1})|m_{h}|\ldots |m_1].
\end{eqnarray*}
By a similar consideration as above, $d_i\mu(\underline{m})$ has $\eta'(m_{h+1})$ as an entry in the $h+1$-st spot. Since $N(m_{h+1})>1$, the tuple $d_i\mu(\underline{m})$ is different from $\underline{m}$. The cell $d_i\mu(\underline{m})$ may or may not be redundant, and it is of height at least $h+1$. If $d_{h+2}$ is norm-preserving, this cell has height exactly $h+1$ if it is redundant.

\item Now consider the case $0<i<n+1$ and $i> h+2$. We have here
\begin{eqnarray*}
 d_i\mu(\underline{m})=[m_n|\ldots|m_{i+1}m_i| \ldots |\overline{\eta}(m_{h+1})| \eta'(m_{h+1})|m_{h}|\ldots |m_1].
\end{eqnarray*}
By a similar consideration as above, $d_i\mu(\underline{m})$ has $\eta'(m_{h+1})$ as an entry in the $(h+1)$-st spot. Since $N(m_{h+1})>1$, the tuple $d_i\mu(\underline{m})$ is different from $\underline{m}$. Furthermore, in this case, the resulting cell is collapsible of height $h+1$.

\item Last, consider $i=n+1$. Recall that $h<n$, so $h+1<n+1$, and we can conclude again $d_i\mu(\underline{m})\neq \underline{m}$ by looking an the $(h+1)$-st spot.  If $n=h+1$, then the resulting cell is essential, otherwise it is collapsible of height $h+1$. 

\end{enumerate}
So we have proven that $\mu$ is a $\mathbb{Z}$-compatible matching. We still have to show that $\mu$ is noetherian. 

Assume for contradiction that there are redundant cells $\underline{x_1} \vdash \underline{x_2}\vdash \ldots$. First, we consider the sum of the norms of the entries for this sequence of cells. Since applying $\eta_i$ preserves this quantity and $d_i$ either preserves it or drops it, we can assume that this quantity is constant for the considered sequence. Note that this implies in particular that there are no applications of $d_0$ or $d_{n+1}$ in the sequence. 

Next, observe that if $x_j$ is obtained as $d_{i_j}\mu(x_{j-1})$, then the height of $x_j$ is $i_j-1$. This follows from the case distinction above since we assumed that $x_j$ is redundant and that $d_{i_j}$ does not drop the norm. Hence, $\mu(x_j)$ is given by $\eta_{i_j}(x_j)$, so by $\eta_{i_j}(d_{i_j}\mu(x_{j-1})))=f_{i_j}(\mu(x_{j-1}))$. Thus, we obtain the corresponding infinite sequence of collapsible partners of $\underline{x}_i$, which can be described as 
\begin{eqnarray*}
\mu(x_1), f_{i_2}(\mu(x_1)), f_{i_3}(f_{i_2}(\mu(x_1))), \ldots
\end{eqnarray*}

The $m$-th element in this sequence starting from the second one is characterized by the sequence $(i_m, i_{m-1}, \ldots, i_2)$ in $F_{n}$. By assumption, any two consecutive cells in the sequence of collapsible cells as above are distinct, thus the sequences $(i_m, i_{m-1}, \ldots, i_2)$ are all reduced. Furthermore, the foregoing case distinction shows that $i_{t+1} \leq i_t+1$, so that the sequence $(i_m, i_{m-1}, \ldots, i_1)$ is right-most. 

 We will prove next that all sequences $(i_m,  i_{m-1}, \ldots, i_2)$ constructed as above are small. Assume for contradiction that $r$ is the minimal index so that $(i_r,  i_{r-1}, \ldots, i_2)$ is not small.  By Lemma \ref{CriterionSmallness} and the case distinction above, we can conclude that $i_r=i_{r-1}+1$ and that the sequence  $(i_{r-1}, i_{r-2} \ldots, i_2)$ is of the form $J_{a_l}^{b_l} J_{a_{l-1}}^{b_{l-1}}\ldots J_{a_1}^{b_1}$  with $a_r<b_{r-1}$ and $b_r<b_{r-1}$ for all $2\leq r\leq l$ and $l\geq 2$ and $b_{l-1}=i_{r-1}+1=i_r=b_l+1$. So since $a_l \leq b_l-1=(b_l+1)-2$, the (possibly empty) sequence $J^{b_l-1}_{a_l}$ commutes with $b_l+1=b_{l-1}$, so that 
\begin{eqnarray*}
(i_{r-1}, \ldots, i_2) \sim_P (i_{r-1}, i_{r-1}+1)\, J 
\end{eqnarray*}
for some sequence $J$. Recall that $P$-equivalent sequences evaluate to equal sequences of applications of $f_i$ (independently of norm behavior). 

So we may assume the sequence in question (resulting from the collapsible cell sequence) would contain a connected subsequence of the form $(k+1, k, k+1)$. Denote the spots $k+2$, $k+1$, $k$, $k-1$ of the collapsible element of height $h$ the sequence $(k+1, k, k+1)$ is applied to by $(a,b,c,d)$. Consider the application of $f_kf_{k+1}$ which were assumed to be norm-preserving:

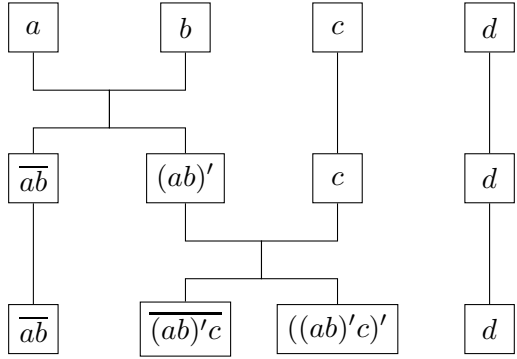
\begin{center}
     \begin{tikzpicture}[node distance = 2cm, auto]
 \node[cloud](a1){$a$};
 \node[cloud, right of=a1](b1){$b$};
 \node[cloud, right of=b1](c1){$c$};
  \node[cloud, right of=c1](d1){$d$};
  
 \node[cloud, below of=a1](a2){$\overline{ab}$};
 \node[cloud, below of=b1](b2){$(ab)'$};
  \node[cloud, below of=c1](c2){$c$};
   \node[cloud, below of=d1](d2){$d$};
   
  \node[cloud, below of=a2](a3){$\overline{ab}$};
 \node[cloud, below of=b2](b3){$\overline{(ab)'c}$};
 \node[cloud, below of=c2](c3){$((ab)'c)'$};
  \node[cloud, below of=d2](d3){$d$};


\draw (c1) -- (c2); 
\draw (d1) -- (d2); 
    \draw (a1.south) -- ++(0.0, -0.5) -| ++(1.0, -0.5) -- ++(0.0,-0.0)-- ++(-1.0, 0.)--  (a2.north);
    \draw (b1.south) -- ++(0.0,-0.5) -| ++(-1.0, -0.5) -- ++(0.0,-0.0)-- ++(1.0, 0.)--  (b2.north);
\draw (d2) -- (d3); 
\draw (a2) -- (a3); 
    \draw (b2.south) -- ++(0.0, -0.5) -| ++(1.0, -0.5) -- ++(0.0,-0.0)-- ++(-1.0, 0.)--  (b3.north);
    \draw (c2.south) -- ++(0.0,-0.5) -| ++(-1.0, -0.5) -- ++(0.0,-0.0)-- ++(1.0, 0.)--  (c3.north);
 \end{tikzpicture}
    \end{center}

(Note that for $k=1$, the last column does not exist. This is not going to be a problem.)

Since the application of first $d_{k+1}$ yields a redundant cell of height $k+1$, we know in particular that $c$ lies in $\E$. Furthermore, since factorability implies weak factorability and the applications of $f_i$ were assumed to be norm-preserving, we have
\begin{eqnarray*}
\overline{ab} \cdot\overline{(ab)'c}=\overline{abc},\\
((ab)'c)'=(abc)'.
\end{eqnarray*} 
So after $d_{k+1}f_kf_{k+1}$, the tuple is stable at position $k$, thus it is not redundant of height $k$. 

This is a contradiction, so all the sequences $(i_m,  i_{m-1}, \ldots, i_2)$ in $F_{n+1}$ resulting from the chain of collapsible cells as above are right-most, reduced and small, and thus have bounded length by Lemma \ref{FinitenessSmallSequences}. So we have proved that the matching $\mu$ is noetherian, and this completes the proof of the proposition. 
\end{pf}

\section{Chain Complex for Homology of Factorable Monoids}
\label{Chain Complex for Homology of Factorable Monoids}

In the previous section, we have shown that there exists a noetherian, $\Z$-compatible matching on the bar complex of a factorable monoid $(M, \E, \eta)$. In this section, we would like to describe the resulting complex. Recall that Theorem \ref{BCF} gives us a description of the desired complex.

The basis for each degree is given by the essential cells in this degree. Recall that these are given by tuples of the form $[m_n|m_{n-1}|\ldots|m_1]$ with $m_i\in \E$ for all $1\leq i\leq n$ and with $(m_{i+1}, m_i)$ unstable for all $1\leq i\leq n-1$. The aim of this section is to describe the differentials of the new complex explicitly. Recall that $d_i$ denotes the $i$-th part of the differential in the bar complex.  

First, we obtain the following proposition from Theorem \ref{MorseKomplexExplizit}.

\begin{Prop} \label{VisyDifferentialCoherent}
Let $(M, \E, \eta)$ be a factorable monoid, and let $\mu$ be the matching on $\mathbb{B}_*M$ associated with this factorable monoid as in Section \ref{Matching for Factorable Monoids}.  Then the Morse complex $(\mathbb{V}_*, \partial^{\mathbb{V}})$ of this matching is a free chain complex with basis
\begin{eqnarray*}
\left\lbrace [m_n|\ldots |m_1]\,  |\, m_i \in \E, m_i\neq 1, (m_{i+1}, m_i) \mbox{ unstable}\right\rbrace
\end{eqnarray*}
in degree $n$. 

The differential is given by 
\begin{eqnarray*}
\partial^{\mathbb{V}} ([m_n|\ldots | m_1])= \sum_{(j, i_r, \ldots, i_1)} (-1)^{j+r} d_jf_{i_r}\ldots f_{i_1}([m_n|\ldots | m_1]),
\end{eqnarray*}
where the sum runs over all reduced tuples $(j, i_r, \ldots, i_1)$ so that all \linebreak $d_{i_k}f_{i_{k-1}}\ldots  f_{i_1}([m_n|\ldots | m_1])$ are redundant and $d_jf_{i_r}\ldots f_{i_1}([m_n|\ldots | m_1])$ is essential. 
\end{Prop}

\begin{pf}
We apply Theorem \ref{MorseKomplexExplizit}. There, we have to build a sum over all chains $z_r \vdash \ldots \vdash z_1$ of redundant cells in $\Omega_{n-1}$ so that $\Mf{d( [m_n|\ldots | m_1]), z_r }\neq 0$. Thus, in our case $z_r$ is some $d_{i_1}([m_n|\ldots | m_1])$ which is required to be redundant.

We make a general analysis of cases where such a cell is redundant. Consider any cell $[x_s| \ldots | x_1]$ with all $x_i \in \E$. This cell has to be either essential or collapsible. Moreover, if $d_j([x_s| \ldots | x_1])$ is redundant, this implies that $0<j<n$, so we have 
\begin{eqnarray*}
d_j([x_s| \ldots | x_1])=[x_s| \ldots | x_{j+1} x_j| \ldots | x_1].
\end{eqnarray*}
For this to be redundant, it is necessary that $x_{j+1}x_j$ has exactly norm $2$. Thus, its collapsible partner is given by
\begin{eqnarray*}
\mu(d_j([x_s| \ldots | x_1]))=[x_s| \ldots | \overline{\eta}(x_{j+1} x_j) | \eta'(x_{j+1}x_j)| \ldots | x_1].
\end{eqnarray*}
In particular, all the entries of this cell are elements of $\E$. We conclude that in the chain $z_r \vdash \ldots \vdash z_1$, the sum of norms of entries remains constantly $n$. 

Thus, we are in the same situation as in the proof of Proposition \ref{MatchingofFactorableMonoid}, where the relation $z_k \vdash z_{k-1}$ was analyzed. There, we have seen that 
\begin{eqnarray*}
\mu(z_{k-1})=f_{i_{r-k+2}}(\mu(z_k))
\end{eqnarray*}
for some $1\leq i_{r-k+2} \leq i_{r-k+1}+1$. Observe for further use that we have seen that the sequence $(i_{r}, \ldots, i_1)$ is reduced, right-most and small. It is also unique for given $z_r \vdash \ldots \vdash z_1$ as $i_k=\hgt(z_{r-k+1})+1$. Moreover, by the same argument as there, we have
\begin{eqnarray*}
\mu(z_r)=\mu(d_{i_1}([m_n|\ldots | m_1]))=f_{i_1}([m_n|\ldots | m_1]).
\end{eqnarray*}
So the sequences $(i_{r}, \ldots, i_1)$  such that all $d_{i_k}f_{i_{k-1}}\ldots  f_{i_1}([m_n|\ldots | m_1])$ are redundant are in one-to-one correspondence with chains $z_r \vdash \ldots \vdash z_1$ of redundant cells in $\Omega_{n-1}$ so that $\Mf{d( [m_n|\ldots | m_1]), z_r }\neq 0$. 

Last, for the summand $d_jf_{i_r}\ldots f_{i_1}([m_n|\ldots | m_1])$ to occur in the sum defining the differential according to Theorem \ref{MorseKomplexExplizit}, it is in addition necessary that the cell 
\begin{eqnarray*}
d_jf_{i_r}\ldots f_{i_1}([m_n|\ldots | m_1])
\end{eqnarray*}
 is essential, as stated above. So the sum in the statement of the theorem contains exactly the same summands as the one in Theorem \ref{MorseKomplexExplizit}. We still have to check the coefficients in the sum. 

Recall that for the matching $\mu$ defined in the last section, we have 
\begin{eqnarray*}
 \Mf{d\mu(\underline{x}), \underline{x}}=(-1)^{h+1}
\end{eqnarray*}
 for any redundant cell $\underline{x}$ of height $h$. Moreover, from the case distinction of the proof of Proposition \ref{MatchingofFactorableMonoid}, we conclude that each 
\begin{eqnarray*}
 d_{i_k}f_{i_{k-1}}\ldots  f_{i_1}([m_n|\ldots | m_1]) 
\end{eqnarray*}
has height $i_k-1$. So the coefficient of $(j, i_r,\ldots, i_1)$-summand simplifies to 
\begin{eqnarray*}
 &&(-1)^{r}\frac{\Mf{d x, z_r} \Mf{d \mu(z_r), z_{r-1}} \Mf{d \mu(z_{r-1}), z_{r-2}}\ldots \Mf{d \mu(z_1), y}}{\Mf{d \mu(z_r), z_{r}}\Mf{d \mu(z_{r-1}), z_{r-1}}\ldots \Mf{d\mu(z_1), z_{1}}} \\
& =& (-1)^{r}\frac{\Mf{d x, z_r} \Mf{d \mu(z_r), z_{r-1}} \Mf{d \mu(z_{r-1}), z_{r-2}}\ldots \Mf{d \mu(z_1), y}}{(-1)^{i_r+\ldots +i_1}},
\end{eqnarray*}
where we write $x=[m_n|\ldots | m_1]$ and $y=d_jf_{i_r}\ldots f_{i_1}([m_n|\ldots | m_1])$.

Next, we note that $z_r=d_{i_1}x$ implies $\Mf{d x, z_r}=(-1)^{i_1}$. Similarly, 
\begin{eqnarray*}
\Mf{d \mu(z_k), z_{k-1}}= (-1)^{i_{r-k+2}}
\end{eqnarray*}
and
\begin{eqnarray*}
\Mf{d \mu(z_1), y}= (-1)^j.
\end{eqnarray*}
Thus, the desired coefficient is given by $(-1)^{j+r}$. This completes the proof of this proposition. 
\end{pf}

Note that in the proposition above, the indexing set of the sum depends on the cell we apply the differential to. Next, we want to exhibit, as far as possible, a common indexing set for all essential cells. This is going to be the set of all reduced, right-most, small sequences $(i_l, \ldots, i_1)$. 

Let us denote the set of all reduced, right-most, small sequences in $F_{n-1}$ by $\Lambda_{n-1}$. Recall that all tuples occurring in the sum above already lie in $\Lambda_{n-1}$. We want to say that a tuple $( i_l, \ldots, i_1)$ is $\underline{x}$-\textbf{coherent} if all $d_{i_k}f_{i_{k-1}}\ldots  f_{i_1}(\underline{x})$ are redundant, and denote by $F_{\underline{x}}$ the set of all $\underline{x}$-coherent sequences. It follows from the proof of the theorem above that if $\underline{x}$ has only entries in $\mathcal{E}$, then $F_{\underline{x}}\subset \Lambda_{n-1}$. The following lemma characterizing the complement of this inclusion will be crucial.

	\begin{Lemma}\label{not coherent stable}
		Consider $\underline{x}=[m_n|\ldots |m_1] \in \Omega_n$ with all $m_i$ lying in $\E$, and let $(i_s, \ldots, i_1)$ be right-most, reduced and small, but not $\underline{x}$-coherent. Then $f_{i_s}  \ldots  f_{i_1}(\underline{x}) = 0$ or there exists $1\leq t \leq s$ such that $f_{i_t} \ldots f_{i_1}(\underline{x})$ is stable at position $i_{t}-1$.
	\end{Lemma}

	\begin{proof}
		Let $t$ be the uniquely determined index such that $(i_{t-1}, \ldots, i_1)$ is $\underline{x}$-coherent and $(i_t, \ldots, i_1)$ is not. Set $\underline{y} := f_{i_{t-1}}  \ldots f_{i_1}(\underline{x})$. If $d_{i_t}$ is not norm-preserving for $\underline{y}$, then applying $\eta_{i_t}$ would produce a trivial entry and thus $f_{i_t}(\underline{y}) = 0$. In this case we are finished.
		
		 So assume that $d_{i_t}$ is norm-preserving for $\underline{y}$. We know by assumption that $d_{i_t}(\underline{y})$ is not redundant. We will now show that $f_{i_t}(\underline{y})$ is stable at position $i_t-1$.
		
		We use the case distinction of the proof of Proposition \ref{MatchingofFactorableMonoid}. Note that $1\leq i_t \leq i_{t-1}+1$ since $(i_s, \ldots, i_1)$ is right-most sequence in $F_{n-1}$. This excludes, since the height of $\underline{y}$ is $i_{t-1}-1$, the cases 1, 5 and 6. Thus, we can conclude that height of $d_{i_t}(\underline{y})$ is at least $i_t-1$.

		The situation is as follows: $d_{i_t}(\underline{y})$ has height  at least $i_t-1$ and its $i_t$-th entry has norm $2$. So $d_{i_t}(\underline{y})$ has height exactly $i_t-1$. Furthermore, $d_{i_t}(\underline{y})$ is not redundant. It follows that $d_{i_t}(\underline{y})$ is stable at position $i_t-1$. Let the entries at positions $i_t$ and $i_t-1$ of $d_{i_t}(\underline{y})$ be denoted by $a$ and $b$, respectively. Then $f_{i_t}(\underline{y})$ has the entries $(a',b)$ in the same places. Due to the Recognition Principle, this is again a stable pair. Thus, $f_{i_t} \ldots f_{i_1}(\underline{x})$ is stable at position $i_{t}-1$.
	\end{proof}
	
	The following lemmas give partial answers to what extend insertion and deletion of entries preserve smallness of finite sequences. Recall that any  connected subsequence of a small sequence is again small. In particular, a small sequence stays small when deleting its first or last entry. We investigate what happens if we delete an ``inner'' entry.
	
	\begin{Lemma}\label{LEM:wedge condition delete entry}
		Let $(i_s, \ldots, i_1)$ be a right-most, reduced and small sequence. For every $1\leq k\leq s$, the following holds: If $i_k < i_{k-1}$, then $(i_s, \ldots, \widehat{i_k}, \ldots, i_1)$ is right-most, reduced and small. 
	\end{Lemma}
	
	\begin{proof}
		Assume that $i_k < i_{k-1}$. We will use the third part of  Lemma \ref{CriterionSmallness} repeatedly. Recall that we can write any right-most reduced sequence in a form $J_{\sms{a_l}}^{\sms{b_l}} J_{\sms{a_{l-1}}}^{\sms{b_{l-1}}}\ldots J_{\sms{a_1}}^{\sms{b_1}}$  with $a_r<b_{r-1}$. Since it is small, we concluded furthermore $b_r<b_{r-1}$ for all $2\leq r\leq l$. Since $i_k<i_{k-1}$, there is $2\leq r\leq l$ with $i_k=a_r$, $i_{k-1}=b_{r-1}$. Thus, after deleting $i_k$, the sequence is either of the form 
		\begin{eqnarray*}
		J_{\sms{a_l}}^{\sms{b_l}} J_{\sms{a_{l-1}}}^{\sms{b_{l-1}}}\ldots J_{\sms{a_r+1}}^{\sms{b_r}}\ldots J_{\sms{a_1}}^{\sms{b_1}}
		\end{eqnarray*}
		or of the form
		\begin{eqnarray*}
		J_{\sms{a_l}}^{\sms{b_l}} J_{\sms{a_{l-1}}}^{\sms{b_{l-1}}}\ldots \widehat{J_{\sms{a_r}}^{\sms{b_r}}}\ldots J_{\sms{a_1}}^{\sms{b_1}},
		\end{eqnarray*}
		depending on whether $a_r=b_r$ holds. In both cases, repeated application of the third part of Lemma \ref{CriterionSmallness} yields the claim. 
		 
	\end{proof}
	
	\begin{Lemma}\label{LEM:deleting entries small}
		Let $(i_s, \ldots, i_1)$ be right-most, reduced and small. Let $j$ be such that $(i_s, \ldots, i_k, j, i_{k-1}, \ldots, i_1)$ is right-most and reduced but not small. Then one of the following holds
		\begin{itemize}
			\item	$j > i_{k-1}$ or
			\item	$i_k < j$ and there exists some $r' > k$ such that $i_{r'} = j$.
		\end{itemize}
	\end{Lemma}
	
	\begin{proof}
		Assume that $j \leq i_{k-1}$, i.e.,\ $j < i_{k-1}$ since the sequence is reduced. We are going to verify the second condition. Write again the sequence $(i_s, \ldots, i_1)$ in the form $J_{\sms{a_l}}^{\sms{b_l}} J_{\sms{a_{l-1}}}^{\sms{b_{l-1}}}\ldots J_{\sms{a_1}}^{\sms{b_1}}$ with $a_r<b_{r-1}$ and $b_r<b_{r-1}$ for all $2\leq r\leq l$. Since the sequence is still right-most and reduced after inserting $j$, it has again a similar presentation. Note that $j$ cannot be inserted inside some connected sequence $J_a^b$. Thus, there is some $1\leq r \leq l$ so that 
		\begin{eqnarray*}
		(i_s, \ldots, i_k)= J_{\sms{a_l}}^{\sms{b_l}}\ldots  J_{\sms{a_{r+1}}}^{\sms{b_{r+1}}}, \\
		(i_{k-1}, \ldots, i_1)=J_{\sms{a_r}}^{\sms{b_r}}\ldots  J_{\sms{a_{1}}}^{\sms{b_{1}}}.
		\end{eqnarray*}
		Note that if $j\leq a_{r+1}=i_k$, then by definition of right-most reduced sequences, we have $j=a_{r+1}-1<b_r$, so the decomposition of the new sequence is given by 
			\begin{eqnarray*}
 J_{\sms{a_l}}^{\sms{b_l}}\ldots  J_{\sms{a_{r+1}-1}}^{\sms{b_{r+1}}} J_{\sms{a_r}}^{\sms{b_r}}\ldots  J_{\sms{a_{1}}}^{\sms{b_{1}}}.
		\end{eqnarray*}
		Since  $b_s<b_{s-1}$ for all $2\leq s\leq l$, this sequence is small by Corollary \ref{SmallnessCriterionJ}, contradicting our assumptions on $j$. 
		
		Hence, we may conclude that $j>a_{r+1}=i_k$. Recall that by assumption, we have $j<i_{k-1}=b_r$. Thus, the decomposition of the new sequence is given by 
		\begin{eqnarray*}
 J_{\sms{a_l}}^{\sms{b_l}}\ldots  J_{\sms{a_{r+1}}}^{\sms{b_{r+1}}}  J_{\sms{j}}^{\sms{j}}  J_{\sms{a_r}}^{\sms{b_r}}\ldots  J_{\sms{a_{1}}}^{\sms{b_{1}}}.
		\end{eqnarray*} 
		Again by Corollary \ref{SmallnessCriterionJ}, there must exist an $r+1\leq t\leq l$ so that $j\leq b_t$, since the new sequence is not small. Since $b_{r+1}>b_t$, we have $a_{r+1}<j \leq b_{r+1}$, so $j$ appears in $ J_{\sms{a_{r+1}}}^{\sms{b_{r+1}}}$. This completes the proof of the lemma. 
	\end{proof}
	
	\begin{Prop}\label{PRP:wedge condition not coherent 0}
		Let $\underline{x} \in \Omega_n$ be essential. Then
		\begin{align*}
			\sum_{I \in \Lambda_{n-1} \sm F_{\underline{x}}} (-1)^{\#I} \cdot f_I(\underline{x}) = 0.
		\end{align*}
	\end{Prop}
	Note that the sum makes sense since $F_{\underline{x}} \subseteq \Lambda_{n-1}$ and $\Lambda_{n-1}$ is finite.
	
	\begin{proof}
		The proof is quite long and we therefore give a rough outline first. The idea is as follows. We define an action of $\Z/2$ on $ \Lambda_{n-1} \setminus F_{\underline{x}}$ such that ``the sum over each orbit is zero''. More precisely, we define for each essential $\underline{x}$ a map $\xi\colon \Lambda_{n-1} \setminus F_{\underline{x}} \to \Lambda_{n-1} \setminus F_{\underline{x}}$ with the following properties:
		\begin{enumerate}
			\item	$\xi^2 = \id$.
		
			\item	$f_{\xi(I)}(\underline{x}) = f_I(\underline{x})$.
		
			\item	If $I$ is a fixed point of $\xi$, then $f_I(\underline{x}) = 0$.
		
			\item\label{item:absolute value lengths 1}	If $I$ is not a fixed point of $\xi$, then the lengths of $I$ and $\xi(I)$ differ by one. Together with the second point, this gives $(-1)^{\#I}f_I(\underline{x}) + (-1)^{\#\xi(I)}f_{\xi(I)}(\underline{x}) = 0$.
		\end{enumerate}
		We now construct this map $\xi$. Consider $(i_s, \ldots, i_1) \in \Lambda_{n-1} \setminus F_{\underline{x}}$. If $f_{i_s} \ldots f_{i_1}(\underline{x}) = 0$, then $\xi$ will fix this sequence. 

	For other sequences $(i_s, \ldots, i_1) \in \Lambda_{n-1} \setminus F_{\underline{x}}$, we know by Lemma \ref{not coherent stable} that there is a $t$ so that the cell $f_{i_t} \ldots  f_{i_1}(\underline{x})$ is stable at position $i_t-1$. We will show that, whenever such a $t$ is given, one of the following three cases occurs:
\renewcommand{\labelenumi}{Case \arabic{enumi}:}	
	\begin{enumerate}
	\item  The sequence $(i_s, \ldots, i_{t+1})$ is $P$-equivalent to a sequence with $i_t-1$ on the very right; i.e., there is an entry $i_t-1$ in $(i_s, \ldots, i_{t+1})$ and on the right from it, there are only entries commuting with this one. In this case, we can define $\xi$ be deleting the entry $i_t-1$. 
	\item The sequence $(i_s, \ldots, i_{t+1})$ is not $P$-equivalent to a string with $i_t-1$ on the very right, but$(i_s, \ldots, i_{t+1}) . (i_t\!-\!1) . (i_t, \ldots, i_1)$ is small. Then we will be able to define $\xi$ via inserting $i_t-1$. 
	\item If none of the first two cases occurs, then there is a $t'>t$ so that the cell $f_{i_{t'}} \ldots  f_{i_1}(\underline{x})$ is stable at position $i_t'-1$.
	\end{enumerate}
	
	This makes clear how to define $\xi$ in general: We start with a minimal $t$ satisfying the condition above, and we are either done if we are in Case 1 or 2, or we continue with the minimal $t'$ as in Case 3. Note that this procedure has to terminate since the sequence we started with is finite. We now fill in the details of the case distinction above. 
	
		\renewcommand{\labelenumi}{Case \arabic{enumi}:}
		\begin{enumerate}
			\item	Assume $(i_s, \ldots, i_{t+1})$ is $P$-equivalent to a string with $i_t-1$ on the very right; i.e., there is an entry $i_t-1$ in $(i_s, \ldots, i_{t+1})$ and on the right from it, there are only entries commuting with this one. 
			
				The idea is to define $\xi(i_s, \ldots, i_1)$ to be the finite sequence that arises from ``deleting'' this $(i_t-1)$-entry. Let $k \geq t+1$ be minimal subject to $i_k = i_t-1$
				and set $\xi(i_s, \ldots, i_1) := (i_s, \ldots, \widehat{i_k}, \ldots, i_1)$. We now check that $\xi(i_s, \ldots, i_1)$ lies in $\Lambda_{n-1} \setminus F_{\underline{x}}$.
				
				By assumption, we have
				\begin{eqnarray*}
				(i_s, \ldots, i_{t+1}) \sim_P (i_s, \ldots, i_{k+1}) . (i_{k-1}, \ldots, i_{t+1}) . (i_k).
				\end{eqnarray*}
				
				We claim that $i_k < i_{k-1}$. If $k = t+1$, this is obvious. If $k > t+1$, we know by assumption that $|i_r - i_k| \geq 2$ for every $t\leq r \leq k$. Since $(i_s, \ldots, i_1)$ is right-most, we must have $i_k < i_{k-1}$.
				
				Lemma \ref{LEM:wedge condition delete entry} tells us that $(i_s, \ldots, \widehat{i_k}, \ldots, i_1)$ is indeed right-most, reduced and small. It is obviously not $\underline{x}$-coherent, because $(i_t, \ldots, i_1)$ is not. Note that by definition $\#\xi(I) = \#I - 1$. Observe that in this case, $f_{\xi(I)}(\underline{x}) = f_I(\underline{x})$ holds. Indeed, this follows since $f_{i_t}  \ldots  f_{i_1}(\underline{x})$ is stable at position $i_t-1$ via the Evaluation Lemma using the identity  $(i_s, \ldots, i_{t+1}) \sim_P (i_s, \ldots, i_{k+1}) . (i_{k-1}, \ldots, i_{t+1}) . (i_k)$.
				
			\item  Assume $(i_s, \ldots, i_{t+1})$ is not $P$-equivalent to a string with $i_t-1$ on the very right and assume that $(i_s, \ldots, i_{t+1}) . (i_t\!-\!1) . (i_t, \ldots, i_1)$ is small.
			
				The idea is to define $\xi(i_s, \ldots, i_1)$ by ``inserting'' the entry $i_t-1$ into $(i_s, \ldots, i_1)$. Note that $(i_s, \ldots, i_{t+1}).(i_t\!-\!1).(i_t, \ldots, i_1)$ need not be right-most. Recall however that by (the symmetric counterpart of) Proposition \ref{AllLeftMost}, there exists a unique right-most, reduced sequence $L$ which is $P$-equivalent to $(i_s, \ldots, i_{t+1}).(i_t\!-\!1).(i_t, \ldots, i_1)$.  So we define
				\begin{eqnarray*}
				\xi(i_s, \ldots, i_1) := L\; (i_t, \ldots, i_1).
				\end{eqnarray*}
It is easily checked that this sequence is right-most and reduced. It is not hard to see from the proof of Proposition \ref{AllLeftMost} that the length of $\xi(i_s, \ldots, i_1)$ is $s+1$. Furthermore,  $\xi(i_s, \ldots, i_1)$ is small since smallness depends only on $P$-equivalence class. Since $(i_t, \ldots, i_1)$ is not $\underline{x}$-coherent, $\xi(i_s, \ldots, i_1)$ is neither. Obviously, in this case $\xi$ has all the desired properties.
		\end{enumerate}
Before discussing the remaining case, let us remark that for all sequences $I$ that are covered by Cases 1 and 2, we have $\xi^2(I) = I$.
		
		\begin{enumerate}[resume]
			\item	Assume $(i_s, \ldots, i_{t+1})$ is not $P$-equivalent to a string with $i_t-1$ on the very right and assume that $(i_s, \ldots, i_{t+1}) . (i_t\!-\!1) . (i_t, \ldots, i_1)$ is not small.
			
				We are going to find an index $t' > t$ such that $f_{i_{t'}} \ldots f_{i_1}(\underline{x})$ is stable at position $i_{t'}-1$.  The existence of such an entry $t'$ will follow from Lemma \ref{LEM:deleting entries small}. To have it applicable, we need to find the right-most, reduced representative of $(i_s, \ldots, i_{t+1}) . (i_t\!-\!1) . (i_t, \ldots, i_1)$. Since $(i_s, \ldots, i_1)$ is right-most and reduced it follows that there exists some $k$ such that the right-most, reduced representative is of the form
				\begin{eqnarray*}
				 (i_s, \ldots, i_k) . (i_t\!-\!1) . (i_{k-1}, \ldots, i_1).
				\end{eqnarray*}
				
				Note that $k \geq t+1$ because $(i_t\!-\!1, i_t, \ldots, i_1)$ is right-most and reduced. Intuitively speaking, we obtain this right-most, reduced representative by successively pushing the entry $(i_t\!-\!1)$ to the left, and that we do as long as its left neighbor is $\geq i_t+1$. Observe also that the entry $(i_t\!-\!1)$ cannot be absorbed by this operations since otherwise,  $(i_s, \ldots, i_t)$ would be $P$-equivalent to a string with $i_t-1$ on the very right.
				
				$\quad\quad$\textbf{Claim 1.} $i_t-1 < i_{k-1}$.
				
				If $k-1 = t$ then the Claim 1 reads as $i_t-1 < i_t$, which is obvious. Otherwise, $i_t-1$ and $i_{k-1}$ commute, thus,\ $|(i_t-1) - i_{k-1}| \geq 2$. Since our sequence is right-most, this implies $i_t-1 < i_{k-1}$. Claim 1 is proven.
				
				We are in the following situation: The sequence $(i_s, \ldots, i_1)$ is right-most, reduced and small. The sequence $(i_s, \ldots, i_k) . (i_t\!-\!1) . (i_{k-1}, \ldots, i_1)$ is right-most and reduced but not small. Furthermore, $i_t-1 < i_{k-1}$. Lemma \ref{LEM:deleting entries small} guarantees that $i_k < i_t-1$ and that there exists an entry $t' > k$ such that $i_{t'} = i_t-1$. Amongst all possible choices of $t'$, we choose the smallest one.
				
				$\quad\quad$\textbf{Claim 2.} There is a sequence $K$ so that 
				\begin{eqnarray*}
				(i_{t'}\!-\!1)\; K \simeq_Q  (i_{t'}, \ldots, i_k, i_t\!-\!1).
				\end{eqnarray*}
				
				The connected subsequence $(i_{t'}, \ldots, i_k)$ is small, right-most and reduced  and thus by Corollary \ref{SmallnessCriterionJ},  we can write it as $J_{\sms{a_m}}^{\sms{b_m}} J_{\sms{a_{m-1}}}^{\sms{b_{m-1}}}\ldots J_{\sms{a_1}}^{\sms{b_1}}$ with $a_r<b_{r-1}$  and $b_{r}<b_{r-1}$ for all $2\leq r \leq m$. Note that if $m\geq 2$, then the inequality 
				\begin{eqnarray*}
				a_1=i_k<i_t-1=b_m=i_{t'} < b_1
				\end{eqnarray*}
				holds, so that $i_t-1$ appears in the sequence $ J_{\sms{a_1}}^{\sms{b_1}}$, contradicting the minimality of $t'$. 
				
				Thus, we conclude that $m=1$, and since the sequence \linebreak $(i_{t'}, \ldots, i_k, i_t\!-\!1)$ is reduced, $a_m$ is strictly smaller than $b_m$. So we have $i_{t'-1}=i_{t'}-1=b_m-1$, and all entries between $i_{t'-1}$ and $i_t-1$ commute with $b_m=i_t-1$. Thus, the sequence  $(i_{t'}, \ldots, i_k, i_t\!-\!1)$ is $P$-equivalent to one of the form $(b_m,  b_m-1, b_m) I$, and thus $Q$-equivalent to 
				\begin{eqnarray*}
				(b_m-1, b_m,  b_m-1, b_m) I= (i_{t'}-1)(b_m,  b_m-1, b_m) I.
				\end{eqnarray*}
				This completes the proof of the claim. 
				
				$\quad\quad$\textbf{Claim 3.} $f_{i_{t'}} \ldots f_{i_1}(\underline{x})$ is stable at position $i_{t'}-1$.
				
				Recall that $f_{i_t} \ldots f_{i_1}(\underline{x})$ is stable at position $i_t-1$, yielding
				\begin{eqnarray*}
				f_{i_{t'}} \ldots f_{i_1}(\underline{x}) &= f_{i_{t'}} \ldots f_{i_{t+1}} f_{i_{t}-1} f_{i_{t}} \ldots f_{i_1}(\underline{x}).
				\end{eqnarray*}
					
				Since $(i_{t'}, \ldots, i_{t+1}, i_t-1, i_t, \ldots, i_1)$ is $P$-equivalent to  \linebreak $ (i_{t'}, \ldots, i_k, i_t\!-\!1, i_{k-1}, \ldots, i_1)$, we obtain
				\begin{eqnarray*}
				 f_{i_{t'}} \ldots f_{i_{t+1}} f_{i_t-1} f_{i_{t}} \ldots f_{i_1}(\underline{x})=f_{i_{t'}} \ldots f_{i_{k}} f_{i_t-1} f_{i_{k-1}} \ldots f_{i_1}(\underline{x}).
				\end{eqnarray*}
				 
				Observe that $f_{i_{t'}} \ldots f_{i_1}$ is norm-preserving for $\underline{x}$: Since all entries of $\underline{x}$ are in $\E$, being not norm-preserving would imply $f_{i_{t'}} \ldots f_{i_1}(\underline{x})=0$.
				
				Thus, the $Q$-equivalence of the Claim 2 gives an (honest) equality in the Evaluation Lemma, so that  $f_{i_{t'}} \ldots f_{i_1}(\underline{x})$ is stable at position $i_{t'}-1$. Claim 3 is proven.
		\end{enumerate}
		
	As explained earlier, the completeness of this case distinction shows that the existence of the map $\xi$ with the desired properties, and this in turn yields the vanishing of the sum in question. 
	\end{proof}
	
	The following example should shed some light on the algorithm described in the proof above.
	
	\begin{example}
		Let $(M, \E, \eta)$ be a factorable monoid. Consider an essential $4$-cell $\underline{x} = [x_4 | x_3 | x_2 | x_1] \in \mathbb{B}_4 M$ with the property that the finite sequence $(2,1)$ is $\underline{x}$-coherent and
		\begin{align*}
			d_3 f_2 f_1(\underline{x}) = [ x_4 \overline{x_3 \overline{x_2 x_1}} | (x_3 \overline{x_2 x_1})' | (x_2 x_1)' ]
		\end{align*}
		is not redundant. We write $\Xi$ for the set of all finite sequences in $F_3$ that are right-most, reduced, small and that are of the form $L.(3,2,1)$ for some $L \in F_3$. Explicitly:
		\begin{align*}
			\Xi = \{ (3,2,1), (1,3,2,1),  (2,1,3,2,1), (1,2,1,3,2,1), (2,3,2,1), (1,2,3,2,1) \}.
		\end{align*}
		Note that $\Xi \subseteq \Lambda_3 \setminus F_{\underline{x}}$.
		
		Let us assume that every element in $\Xi$ is norm-preserving for $\underline{x}$, because otherwise it would be a fixed point of $\xi$ and nothing interesting happens. We will now explicitly describe the map $\xi: \Xi \to \Xi$. First of all, $f_3 f_2 f_1(\underline{x})$ is clearly stable at position $3$, and the proof of Proposition \ref{PRP:wedge condition not coherent 0} tells us that it is also stable at position $2$.
		
		\begin{itemize}
			\item	$(i_s, \ldots, i_1) = (3,2,1) = (\,) . (3,2,1)$.
			
				We will discuss this case in detail. The data is the following: $t=3$, $i_t-1=2$ and $(i_s, \ldots, i_{t+1}) = (\,)$. Clearly, we are not in Case 1. Moreover, the sequence $(i_s, \ldots, i_{t+1}) . (i_t\!-\!1) . (i_t, \ldots, i_1) = (2).(3,2,1)$ is small. Therefore we are in Case 2, yielding $\xi(3,2,1) = (2).(3,2,1) = (2,3,2,1)$.
				
			\item	Let us, for convenience, briefly discuss $(i_s, \ldots, i_1) = (2,3,2,1) = (2).(3,2,1)$.
			
				Clearly, we are in Case 1, yielding $\xi(2,3,2,1) = (\widehat{2},3,2,1) = (3,2,1)$. This shows that we indeed have $\xi^2(3,2,1) = (3,2,1)$.
				
			\item	$(i_s, \ldots, i_1) = (1,3,2,1) = (1) . (3,2,1)$.
			
				Clearly, we are not in Case 1, and $(1) . (2) . (3,2,1)$ is small. Therefore we are in Case 2, yielding 
				\begin{eqnarray*} 
				 \xi(1,3,2,1) = (1).(2).(3,2,1) = (1,2,3,2,1).
				\end{eqnarray*}
  It is also easy to see that, conversely, $\xi(1,2,3,2,1)=(1,3,2,1)$. 
				
			\item	$(i_s, \ldots, i_1) = (2,1,3,2,1) = (2,1) . (3,2,1)$.
			
				Again, we are not in Case 1, but $(2,1) . (2) . (3,2,1)$ is not small. Therefore, we are in Case 3. We therefore restart the case distinction with the sequence $(2,1,3,2,1)$ and $t=5$, $i_t-1=1$. We are again not in Case 1 and $(1,2,1,3,2,1)$ is small. Hence, $\xi(2,1,3,2,1) = (1,2,1,3,2,1)$. 
				
				\item $(i_s, \ldots, i_1) = (1,2,1,3,2,1)=(1,2,1).(3,2,1)$. 
				
				As in the last point, we are not in Case 1, and also $(1,2,1) . (2) . (3,2,1)$ is not small. Hence, we are again in Case 3. We therefore restart the case distinction with the sequence $(1,2,1,3,2,1)$ and $t=5$, $i_t-1=1$. Since $(1,2,1,3,2,1)=(1).(2,1,3,2,1)$, we are now in Case 1, and thus $\xi(1,2,1,3,2,1)=(2,1,3,2,1)$, as expected. 
		\end{itemize}

	\end{example}
	
	Combining  Propositions \ref{VisyDifferentialCoherent} and \ref{PRP:wedge condition not coherent 0}, we obtain the following theorem. We have linear maps $i_* \colon\mathbb{V}_*\to \mathbb{B}_*M$ and $\pi_* \colon \mathbb{B}_*M\to \mathbb{V}_*$, given by inclusion and projection on the essential cells. (These are not chain maps.) 
	
	\begin{Theorem}\label{THM:differential Visy resolution}
		Let $(M, \E, \eta)$ be a factorable monoid and denote by $(\mathbb{V}_*, \partial^{\mathbb{V}})$ its Morse complex associated with the matching of Section \ref{Matching for Factorable Monoids}. The differential $\partial^\mathbb{V}: \mathbb{V}_* \to \mathbb{V}_{*-1}$ can be written as follows,
		\begin{eqnarray*}
		\partial^\mathbb{V}_n = \pi_{n-1} \circ d \circ \sum_{I \in \Lambda_{n-1}} (-1)^{\# I}  f_I \circ i_n.
		\end{eqnarray*}
	\end{Theorem}
	
	Here, $d$ denotes the differential in the bar complex.
	
	\begin{remark}
A variant of Theorem \ref{THM:differential Visy resolution} was obtained first by quite different method by Visy \cite{Visy} and is to appear in \cite{CFBVisy}. 
	
	\end{remark}

\section{Garside theory: Basics}
\label{Garside theory: Basics}

We would like to report on the connection between factorability and Garside theory. In this section, we gather Garside theory basics needed later. Garside theory deals among other things with greedy normal forms, for example on monoids, and with consequences arising from the existence of such normal forms. These normal forms will provide a wide class of examples of factorable monoids and groups. 

The notion of a Garside monoid arises from the following observation: The properties of the braid monoids used by F. Garside in \cite{Garside} to solve the word and the conjugacy problems in the braid monoids and braid groups are also present in a wider class of groups. There are several sets of axioms which reflect the most important of those properties (e.g. \cite{DehornoyF}, \cite{DehornoyParis}, \cite{DehornoyLafont}, \cite{DDGKM}). We will stick to the definition below, which seems to be most appropriate in our context. First however, we clarify which sort of properties is needed. These are certain conditions on the divisibility order in the monoid. 

\begin{Definition}\cite{DehornoyLafont}
Let $M$ be a monoid and let $x,y$ be elements in $M$. We say ``$x$ is a \textbf{left-divisor} of $y$'' or, equivalently, ``$y$ is a \textbf{right-multiple} of $x$'', and write $x\preceq y$ if there is a $z\in M$ such that $y=xz$. We write $x \prec y$ if $x\preceq y$ and $x\neq y$. We call $M$ \textbf{left-noetherian} if there are no infinite descending sequences of the form $\ldots \prec x_3\prec x_2\prec x_1$. Symmetrically, we write $x\succeq y$ for $y$ being a right-divisor of $x$. 
\end{Definition}

\begin{remark}
 A least common left-multiple $c$ of two elements $a,b$ of a monoid $M$ is a common left-multiple of these elements with the following property: whenever $d$ is a common left-multiple of $a$ and $b$, we have $d\succeq c$. This should not be confused with the notion of minimal common left-multiple of $a$ and $b$, meaning a common left-multiple of $a$ and $b$ which is not right-divisible by any other common left-multiple of $a$ and $b$.
\end{remark}

\begin{remark}
There are many variants of the Garside theory. Right-cancellative, left-noetherian monoids in which any two elements admitting a common left-multiple also admit a least common left-multiple were previously called ``left locally Gaussian'' in \cite{DehornoyLafont}. It seems not to be used often in more recent papers, yet it fits exactly in our context. Unfortunately, ``left locally Gaussian'' was also called ``right locally Garside'' in \cite{DehornoyAlternating}. This notion is closely related to the notion of a preGarside monoid in a recent paper by E.~Godelle and L.~Paris (\cite{GodelleParis}).  
It can also be put in the context of the book project by P.~Dehornoy, F.~Digne, E.~Godelle, D.~Krammer and J.~Michel (\cite{DDGKM}).
\end{remark}

\begin{Definition}
A monoid $M$ is called \textbf{atomic} if for any element $m\in M\setminus\{1\}$, the number
\begin{eqnarray*}
\norm{m}:=\sup\{n\,|\, \exists\, a_1,\ldots, a_n \in M\setminus\{1\} \mbox{ such that } m=a_1\ldots a_n\}
\end{eqnarray*}
is finite.
\end{Definition}

The notion of a Garside monoid collects the nice (divisibility) properties enabling many conclusions about the monoid.

\begin{Definition}\cite{DehornoyF}
A monoid $M$ is called \textbf{Garside monoid}  if it is atomic, cancellative and the following conditions hold: 
\begin{enumerate}
 \item For any two elements $x,y$ in $M$, their least common left- and right-multiples and their greatest common left- and right-divisors exist. (We denote them by $\llcm(x, y)$, $\rlcm(x, y)$, $\lgcd(x, y)$, $\rgcd(x, y)$, respectively.)
\item There is an element $\Delta \in M$, called a \textbf{Garside element}, such that the set of the left-divisors of $\Delta$ coincides with the set of right-divisors of $\Delta$, is finite and generates $M$.  
\end{enumerate}
\end{Definition}

The braid groups were the inspiration for the notion of Garside groups, so it is not surprising that these and similar groups provide examples for Garside monoids or weaker forms of those. 

\begin{Definition}\label{ExampleDefinitionArtinGroups}
Recall that an \textbf{Artin group} is a group given by a group presentation of the form
\begin{eqnarray*}
 G(S)=\Mf{S| \underbrace{sts\ldots}_{m_{s,t}} =\underbrace{tst\ldots }_{m_{s,t}}\mbox{ for all } s\neq t \in S}.
\end{eqnarray*}
where $m_{s,t}$ are natural numbers $\geq 2$ or infinity, with $m_{s,t}=m_{t,s}$ for all $s\neq t \in S$. Here, $m_{s,t}=\infty$ means that the pair $s,t$ does not satisfy any relation. We can associate with each Artin group a \textbf{Coxeter group} $W(S)$ by adding relations $s^2=1$ for all $s\in S$; this is consistent with setting the numbers $m_{s,s}$ to be $1$. The matrix $M_S=(m_{s,t})_{s,t\in S}$ will be called the \textbf{Coxeter matrix} defining $G(S)$ or $W(S)$, and the pair $(S, M_S)$ will be also called the \textbf{Coxeter system}. For each Coxeter system, we can define the corresponding \textbf{(positive) Artin monoid} by the monoid presentation
\begin{eqnarray*}
  M(S)= \Mf{S| \underbrace{sts\ldots}_{m_{s,t}} =\underbrace{tst\ldots }_{m_{s,t}}\mbox{ for all } s\neq t \in S}_{mon}.
\end{eqnarray*}
For later use, we will denote the alternating word $sts\ldots$ with ${m}$ factors by $\Mf{s,t}^{m}$. Note that there are a monoid homomorphism $\pi\colon M(S)\to W(S)$ and a group homomorphism $\pi\colon G(S)\to W(S)$ mapping each generator $s$ to its image in the quotient group. 
We will call $M(S)$ as well as $G(S)$ or sometimes, by abuse of notation, even $S$ \textbf{of finite type} if the associated Coxeter group $W(S)$ is finite.
\end{Definition}

\begin{remark} 
  Artin monoids of finite type turn out to determine the behavior of the corresponding Artin group quite completely, as first shown by E.~Brieskorn and K.~Saito (\cite{BrieskornSaito}). In the same article, they investigate more generally the structure of all Artin monoids. Amongst other things, they show that any Artin monoid is left- and right-cancellative, left- and right-noetherian, and any two elements admitting a common left- or  right-multiple also admit a least common left- or right-multiple, respectively. Moreover, it is a Garside monoid if and only if it is of finite type. 

\end{remark}

\begin{remark}
There are further examples of Garside monoids or weaker versions of such as e.g. torus knot groups; see \cite{DDGKM} for a detailed account.
\end{remark}

Now, we are going to describe the normal form mentioned above. This is a greedy normal form: Loosely speaking, one tries to split off a generator which is as large as possible. We are going to make this more precise soon. Furthermore, we are going to recall that being normal form can be checked locally. The following relation makes it easier to describe the local behavior of the normal forms. 

\begin{Definition}\cite{DehornoyLafont}\label{DreieckERelation}
Let $M$ be a monoid and let $\mathcal{E}$ be a subset of $M$. For $x,y\in M$, we say that $x\vartriangleleft_{\mathcal{E}} y$ if every right-divisor of $xy$ lying in $\mathcal{E}$ is a right-divisor of $y$. 
\end{Definition}

We are ready to formulate the main property of the future normal form. 

\begin{Definition} \cite{DehornoyLafont}
 Let $M$ be a monoid and $\mathcal{E}$ a subset of $M$. We say that a finite sequence $(x_p,\ldots, x_1)$ in $\mathcal{E}^p$ is $\mathcal{E}$-\textbf{normal} if for $1\leq i<p$, we have $x_{i+1} \vartriangleleft_{\mathcal{E}} x_{i}$. 
\end{Definition}

The following result is closely related to the existence of the normal form:

\begin{Lemma}(\cite{DehornoyLafont}, Lemma 1.7)
 Let $M$ be a right-cancellative, left-noe\-therian monoid so that any two elements admitting a common left-multiple also admit a least common left-multiple. Let $\E$ a generating set closed under least common left-multiples. Then every element $x\in M\setminus \{1\}$ admits a unique greatest right-divisor lying in $\E$. Here, ``greatest'' means again: Every other right-divisor of $x$ lying in $\E$ is a right-divisor of the greatest one.
\end{Lemma}

The following theorem of P.~Dehornoy and Y.~Lafont (\cite{DehornoyLafont}) ensures the existence of normal forms in right-cancellative, left-noe\-therian monoids in which any two elements admitting a common left-multiple also admit a least common left-multiple. Before stating it, we will need one more notion. It is closely connected to the notion of least common multiples: the notion of a left-complement and right-complement.

\begin{Definition} \label{DefinitionLeftComplement}
 Let $M$ be a left-cancellative monoid. For any two $x,y\in M$ whose least common right-multiple exists, we denote by $x\backslash y$ the unique element such that $\rlcm(x,y)=x\cdot (x\backslash y)$, called the \textbf{right-complement} of $y$ in $x$.

Analogously, we define $x/y$ to be the \textbf{left-complement} in a right-can\-cellative monoid, i.e., the unique element with $\llcm(x, y)=(x/y)\cdot y$. 

\end{Definition}

From now on, we assume $1\notin \E$. We say $\E$ is closed under left-complements if for any $x,y\in \mathcal{E}$, their left-complement $x/y$ also lies in $\E\cup\{1\}$ if it exists.

\begin{Theorem}(\cite{DehornoyLafont}, Proposition 1.9)\label{NFGauss}
 Let $M$ be a right-cancellative, left-noetherian monoid so that any two elements admitting a common left-multiple also admit a least common left-multiple. Let $\mathcal{E}$ be a generating subset of $M$ that is closed under least common left-multiples and left-complements. 
Then every nontrivial element $x$ of $M$ admits a unique minimal decomposition $x=x_p\ldots x_1$ such that $(x_p,\ldots, x_1)$ is an $\mathcal{E}$-normal sequence.
\end{Theorem}

\begin{remark}
The original formulations in \cite{DehornoyLafont} use right locally Gaussian monoids (i.e., left-cancellative, right-noetherian monoids so that any two elements admitting a common right-multiple also admit a least common right-multiple) and left $\E$-normal forms. We use everywhere the mirrored version: 
We use $\vartriangleleft_{\mathcal{E}}$ as in Definition \ref{DreieckERelation} symmetrically to $\vartriangleright_{\E}$ of \cite{DehornoyLafont} and our $\mathcal{E}$ is closed under least common left-multiples and left-complement instead of least common right-multiples and right-complements in \cite{DehornoyLafont}. In what follows, we use the term ``normal form'' for the right $\E$-normal form of Theorem \ref{NFGauss}.
\end{remark}

For a Garside monoid, it is possible to make some conclusions about its group of fractions. Recall that the group of fractions of a monoid $M$ consists of a group $G$ together with a monoid homomorphism $i\colon M\to G$ (often suppressed in the notation) and it is characterized by the following universal property: Whenever $G'$ is a group and $f\colon M\to G'$ is a monoid homomorphism, $f$ factors uniquely through $i$, i.e., there is a unique group homomorphism $f'\colon G\to G'$ such that $f'i=f$. As usual, such a group is unique up to canonical isomorphism if it exists. One possible construction of the group of fractions is given by taking any monoid presentation of $M$ and considering it as a group presentation. 

Note that in general $i$ needs not to be injective. Obviously, a necessary condition for $i$ to be injective is the cancellativity of $M$. However, it was shown already by Malcev in 1937 (\cite{Malcev}) that this is not sufficient. One simple sufficient condition is provided by the Ore criterion: If in a cancellative monoid any two elements admit a common left-multiple, then this monoid embeds into a group, which is equivalent to the injectivity of $i$. Moreover, if the Ore condition is satisfied, any element of the group of fractions can be written in the form $a^{-1}b$ with some $a,b\in M$ (cf.\ e.g. \cite{CliffordPreston}, Section 1.10). 
In particular, any Garside monoid satisfies the Ore condition and embeds into its group of fractions. Furthermore, for the example of Artin groups and Artin monoids above, the above description shows that an Artin group is exactly the group of fractions of the corresponding Artin monoid. It is clear from facts collected so far that an Artin monoid of finite type embeds into the corresponding Artin group. In the general case, this question was open for a long time and it was answered affirmatively by L.~Paris:

\begin{Theorem}(\cite{ParisArtin})
Every Artin monoid injects into the corresponding Artin group.
\end{Theorem}

The fact that Garside monoids satisfy the Ore condition allows to extend the normal form of a Garside monoid to its group of fractions. This is our next aim. A group which can be written as a group of fractions of a Garside monoid is called \textbf{Garside group}. We introduce some notation in the following remark. 

\begin{remark}(\cite{DehornoyF})\label{Stern}
Let $M$ be a Garside monoid with a Garside element $\Delta$. We denote the set of left-divisors of $\Delta$ except for $1$ by $\mathcal{D}$, and set $S=\mathcal{D}\cup \{1\}$. Recall that $S$ also coincides with the set of right-divisors of $\Delta$. For any $t \in S$, we define $t^*=t\bs \Delta$ and $^{*}t=\Delta/t$. Observe that $\Delta=tt^*= {}^{*}tt$ and $({}^{*}t)^*=t$.  We also denote ${}^{*}t$ by $\alpha(t)$.

Furthermore, we denote by $\varphi$ the extension as an endomorphism $M\to M$ of the map $t \mapsto t^{**}$ on $\D$. Note that for any $x \in M$ we have $\Delta \varphi(x)=x\Delta$. Denote furthermore $\delta=\varphi^{-1}$. Note that $\delta=\alpha^2$ on $\mathcal{D}$.
\end{remark}

\begin{Prop}[\cite{DehornoyF}]\label{NFGarside} Let $M$ be a Garside monoid with a Garside element $\Delta$ and let $G$ be its group of fractions. Then each element of $G$ has a unique decomposition $x_p\ldots x_1 y_1^{-1}\ldots y_q^{-1}$ with $y_i, x_i\in \D$ for all $i$, $\rgcd(x_1, y_1)=1$ and 
\begin{eqnarray*}
\rgcd({}^*x_k,x_{k+1})=1\\
\rgcd({}^*y_k, y_{k+1})=1
\end{eqnarray*}
for all $k$. 

In particular, if $q=0$, $x_p\ldots x_1\in M$ and $x_1=\rgcd(x, \Delta)$. The normal form of the monoid elements coincides with the $\mathcal{D}$-normal form described in Theorem \ref{NFGauss}.
\end{Prop}
The condition in this proposition reformulates the ``greediness'' of the normal form. 

For the rest of this section, we collect several facts of Garside theory needed in later proofs. 
To investigate the behavior of the normal form with respect to products, we begin with the following lemma.

\begin{Lemma}[\cite{DehornoyF}]\label{Prod}
Let $M$ be a Garside monoid with a Garside element $\Delta$ and let $x,y$ be elements in $M$. Then we have:
\begin{eqnarray*}
\rgcd(xy, \Delta)=\rgcd(\rgcd(x, \Delta)\, y, \Delta)
\end{eqnarray*}
\end{Lemma}

For our proofs, we will need some rules for the operations of least common left-multiple and left-complement. They are summed up in the following lemma. 

\begin{Lemma}[\cite{DehornoyF}]\label{Rechenregeln}
Let $S$ be again the set of left-divisors of the Garside element $\Delta$ in the Garside monoid $M$. Let $x,y,z$ be elements of $M$ and $s,t \in S$. Then we have:
\begin{enumerate}
 \item $(xy)\backslash z=y\backslash(x\bs z)$
\item $z\bs(xy)=(z\bs x)((x\bs z)\bs y)$
\item $\rgcd(st, \Delta )=(\delta(t)\bs \alpha(s))^*$
\item $ st= \left ( \alpha(s)\bs\delta(t) \right)\cdot\left ((\delta(t)\bs \alpha(s))^*\right )$
\end{enumerate}

\end{Lemma}

To provide a factorability structure on Garside groups, we need to investigate the properties of the normal form given by Proposition \ref{NFGarside}. First, we need to know how it behaves with respect to the norm. This question is answered by R.~Charney and J.~Meier (\cite{CM}):

\begin{Lemma}[\cite{CM}]\label{Norm}
Let $M$ be a Garside monoid and $G$ its group of fractions. The normal form (as in Proposition \ref{NFGarside}) for an element $g \in G$ is geodesic, i.e., the word-length norm of $g$ with respect to $\mathcal{E}=\D\cup \D^{-1}$ is exactly the length of the normal form of $g$. 
\end{Lemma}
In particular, this implies that the embedding $M \to G$ preserves word-lengths. Even in cases where a monoid embeds into its group of fractions, this generally need not to be true. 

We will need some further properties of the word length in a Garside group. 

\begin{Lemma}[\cite{CM}]\label{NormExplizit} Let $M$ be a Garside monoid, $G$ its group of fractions. If $a \in M$ has the word length $k$ with respect to $\mathcal{D}$ and $n$ is a positive integer, then $a\Delta^{-n}$ has word length $N_{\mathcal{E}}(a\Delta^{-n})\leq \max\{k,n\}$ and the equality holds if and only if $\Delta\npreceq a$. 
\end{Lemma}

\begin{remark}
 The normal forms as in Proposition \ref{NFGarside} were already used to describe the homology of Garside groups by R.~Charney, J.~Meier and K.~Whittlesey in \cite{CMW}. We slightly generalized their result, cf.\ \cite{DMTDobrinskaya}.
\end{remark}

\begin{remark}\label{LeftGaussianNotInvertible}
Let $M$ be a right-cancellative, left-noetherian monoid so that any two elements admitting a common left-multiple also admit a least common left-multiple. Then there are no non-trivial invertible elements in $M$: Assume $ab=1$. Multiplying with $a^n$ on the left yields $a^{n+1}b=a^n$. Hence, for each $n$, the element $a^{n+1}$ is a left-divisor of $a^n$. Thus, we obtain a chain of left-divisors $\ldots \preceq a^3\preceq a^2 \preceq a^1\preceq 1$. Since the monoid is left-noetherian, this chain has to stabilize. So for some $k$, we have $a^{k+1}=a^k$ and this implies by right cancellation $a=1$. This implies also $b=1$.
\end{remark}

We will need later the following easy lemma.

\begin{Lemma} \label{LCMprod}
 Let $M$ be a right-cancellative, left-noetherian monoid so that any two elements admitting a common left-multiple also admit a least common left-multiple, and let $x,y\in M$ have the least common left-multiple $ax=by$. Then for any $z\in M$, $xz$ and $yz$ admit a least common left-multiple which is equal to $axz$.
\end{Lemma}

\begin{pf}
 Since $axz=byz$ is a common left-multiple of $xz$ and $yz$, the elements $xz$ and $yz$ have to admit a least common left-multiple. We write it in the form $cxz=dyz$. By right cancellation, we have $cx=dy$, and since $ax=by$ is the least common left-multiple of $x$ and $y$, we have $cx=uax$ for some $u\in M$. By definition, there is then an element $t$ in $M$ such that $axz=tcxz$ holds.  Right cancellation implies $a=tc$ and $c=ua$, so $1=tu$. Since $M$ is left-noetherian, this implies $u=t=1$, hence $axz$ is the least common left-multiple of $xz$ and $yz$. 
\end{pf}

\section{Factorability Structure on a Class of Monoids}
\label{Factorability structure on locally Gaussian monoids}

The aim of this section is to exhibit a factorability structure on right-cancellative, left-noetherian monoids which have the property that any two elements admitting a common left-multiple also admit a least common left-multiple. This factorability structure will correspond to the normal form we discussed in Section \ref{Garside theory: Basics}. First, we give a reformulation of Theorem \ref{NFGauss}. Again, we assume $1\notin \E$. Recall we say that $\E$ is closed under left-complements if for all $x,y\in \E$, the left-complement $x/y$ lies in $\E\cup\{1\}$ whenever this complement exists.

\begin{Cor}\label{NormalFormGreatestLeftDivisor}
 Let $M$ be a a right-cancellative, left-noetherian monoid so that any two elements admitting a common left-multiple also admit a least common left-multiple. Let $\mathcal{E}$ be a generating subset of $M$ that is closed under least common left-multiples and left-complement. For any $x \in M\setminus\{1\}$, the beginning $x_1$ of the (unique) $\mathcal{E}$-normal form of $x$ is the greatest right-divisor of $x$ lying in $\mathcal{E}$ (greatest with respect to ``being right-divisor'').
\end{Cor}

\begin{pf}
 This follows immediately from the proof of Theorem \ref{NFGauss} in \cite{DehornoyLafont}.
\end{pf}

\begin{Lemma} \label{LetztesAxiom}
  Let $M$ be a right-cancellative, left-noetherian monoid so that any two elements admitting a common left-multiple also admit a least common left-multiple. Let $\mathcal{E}$ be a generating subset of $M$ that is closed under least common left-multiples and left-complements. Let $x\in M\setminus\{1\}$ be any non-trivial element and let $a \in \mathcal{E}$ be a generator. Let $\NF(x)=x_p\ldots x_1$ be the (right) $\mathcal{E}$-normal form as described in Theorem \ref{NFGauss}. Let $\NF(xa)=y_q\ldots y_2y_1$. Then either $x_1a=y_1$ or there is a generator $z\in \mathcal{E}$ such that 
\begin{eqnarray*}
\NF(x_1a)=zy_1.
\end{eqnarray*}

\end{Lemma}

\begin{pf}
 Assume $y_1\neq x_1a$. We know by the maximality of $y_1$ that $y_1$ is a left-multiple of $a$, so there is a $t \in M$ with $y_1=ta$. Thus $t$ is the left-complement $y_1/a$, and hence $t\in \mathcal{E}$. Furthermore, we have $xa=y_q\ldots y_2ta$. This implies by cancellation $x\succeq t$. By maximality of $x_1$, we have $x_1\succeq t$ and by assumption, $t\neq x_1$. So there is a $z\in M\setminus\{1\}$ such that $x_1=zt$. As above, $z\in \mathcal{E}$. We already have $x_1a=zta=zy_1$ and have to show that $y_1$ is the greatest right-divisor of $x_1a$ lying in $\mathcal{E}$. Now, for any $u\in \mathcal{E}$ with $x_1a\succeq u$, we also have $xa=(x_p\ldots x_2)x_1a\succeq u$ and thus, by definition, $y_1\succeq u$. This proves the statement. 
\end{pf}

This lemma allows us to compare the lengths of normal forms of $x$ and $xa$ for a generator $a$. This is similar to Lemma 4.6 of \cite{MichelNoteBraid}. 

\begin{Cor} \label{NFLaenge}
 Let $M$ be a a right-cancellative, left-noetherian monoid so that any two elements admitting a common left-multiple also admit a least common left-multiple. Let $\mathcal{E}$ be a generating subset of $M$ that is closed under least common left-multiple and left-complement. Let $x\in M\setminus\{1\}$ and $a \in \mathcal{E}$. Let $\NF(x)=x_p\ldots x_1$ be the (right) $\mathcal{E}$-normal form of $x$. Let $\NF(xa)=y_q\ldots y_2y_1$. Then either $q=p$ or $q=p+1$. 
\end{Cor}

\begin{pf}
 We proceed by induction on $p$. For $p=1$, we have $x=x_1$, and the claim follows directly from Lemma \ref{LetztesAxiom}. Now assume we have proved the statement for all lengths of normal forms $\leq p-1$. We wish to show the statement for $x$ with $\NF(x)=x_p\ldots x_1$ as in the statement of the theorem. Set $v=x_p\ldots x_2$. We know that this is the normal form of $v$ by the definition of $\E$-normal sequences. Again, by Lemma \ref{LetztesAxiom} we can conclude that either $x_1a=y_1$ or there exists a $z\in \mathcal{E}$ such that $\NF(x_1a)=zy_1$.

In the first case, $xa=vy_1=y_q\ldots y_2y_1$, thus $v=y_q\ldots y_2$ and by the uniqueness of normal forms, $q=p$ follows. 

In the second case,  $\NF(x_1a)=zy_1$, so we know by the induction hypothesis that $vz$ has a normal form of length either $p-1$ or $p$ on the one hand. On the other hand, we have 
\begin{eqnarray*}
vzy_1=x_p\ldots x_2zy_1=x_p\ldots x_2x_1a=xa=y_q\ldots y_2y_1
\end{eqnarray*}
and $\NF(vz)=y_q\ldots y_2$ by cancellation and by the uniqueness of the normal form. So $q-1=p-1$ or $q-1=p$, and the claim follows.
\end{pf}

Now we can apply the previous corollary to show that the $\E$-normal form is geodesic. The following lemma is an analogue of Lemma \ref{Norm} by R.~Charney and J.~Meier.

\begin{Lemma}\label{LLGaussschNFNorm}
  Let $M$ be a right-cancellative, left-noetherian monoid so that any two elements admitting a common left-multiple also admit a least common left-multiple. Let $\mathcal{E}$ be a generating subset of $M$ that is closed under least common left-multiples and left-complements. Then the length $p$ of the normal form $x=x_p\ldots x_1$ coincides with the word length with respect to $\mathcal{E}$, $N_{\mathcal{E}}(x)$, for any $x\in M\setminus \{1\}$. 
\end{Lemma}

\begin{pf}
 It is clear that $l(\NF(x))\geq N_{\mathcal{E}}(x)$. For the other inequality, we proceed by induction on $k=N_{\mathcal{E}}(x)$. For $k=1$, the claim is clearly true. Assume we have proven the statement for all word lengths $\leq k-1$. Let $a_k\ldots a_1$ be a minimal word in $\mathcal{E}$ representing $x$. Then $v=a_k\ldots a_2$ is also a minimal word, i.e., word of minimal length, and has thus word length $k-1$, so $\NF(v)=v_{k-1}\ldots v_1$ by the induction hypothesis. By Corollary \ref{NFLaenge}, we know that the normal form of $x=va_1$ is either of length $k-1$ or of length $k$, so $\leq k$. Hence we are done.
\end{pf}

We sum up the results of this section in the following theorem. 

\begin{Theorem}\label{DefinitionFaktorabilitaetGauss}
 Let $M$ be a right-cancellative, left-noetherian monoid so that any two elements admitting a common left-multiple also admit a least common left-multiple. Let $\mathcal{E}$ be a generating subset of $M$ that is closed under least common left-multiple and left-complement. For any $x\in M\setminus\{1\}$ with normal form $\NF(x)=x_p\ldots x_2x_1$, set 
\begin{eqnarray*}
 \eta(x)=(x_p\ldots x_2, x_1),
\end{eqnarray*}
and set furthermore $\eta(1)=(1,1)$.
Then $(M, N_{\mathcal{E}}, \eta)$ is a factorable monoid.
\end{Theorem}

\begin{pf}
 The proof will essentially assemble the results of the last lemmas. First, since $M$ is right-cancellative, we need to show only weak factorability due to Corollary \ref{RechtskuerzbarAequivalenz}. By definition, $\eta'(x)$ is an element of $\mathcal{E}$ for $x\neq 1$, and 
\begin{eqnarray*}
 \overline{\eta}(x)\cdot \eta'(x)=x
\end{eqnarray*}
holds. Since $x_p\ldots x_2$ is the normal form of $\overline{\eta}(x)$ for an $x$ with $\NF(x)=x_p\ldots x_2x_1$, we conclude by Lemma \ref{LLGaussschNFNorm} that $N_{\mathcal{E}}(x)=p$ and $N_{\mathcal{E}}(\overline{\eta}(x))=p-1$, so the norm condition is satisfied. 

For the last condition given by the Diagram \ref{eqWF}, we have to compare $\NF(x)=x_p\ldots x_2x_1$ with $\NF(xa)$ for some $a\in \mathcal{E}$. Write $\NF(xa)=y_q\ldots y_1$ as before. First, by Lemma \ref{FactAxiom}, we have to show that $N_{\mathcal{E}}(xa)=p+1$ if $(x_1, a)$ and $(\overline{x}, \overline{x'a})=(x_p\ldots x_2, z)$ are assumed to be geodesic, where $z\in \mathcal{E}\setminus \{1\}$ is characterized by $\NF(x_1a)=zy_1$. The case distinction in the proof of Lemma \ref{NFLaenge}, combined with Lemma \ref{LLGaussschNFNorm}, shows that in this case, the $\mathcal{E}$-norm of $xa$ is $p+1$. Furthermore, we know that in this case
\begin{eqnarray*}
 (xa)'=y_1=(x'a)'
\end{eqnarray*}
holds. This completes the proof. 
\end{pf}

\begin{Cor} \label{RewritingSystemLLGaussch}
Let $M$ be a right-cancellative, left-noetherian monoid in which any two elements admitting a common left-multiple also admit a least common left-multiple. Then for any factorability structure on $M$ obtained from Theorem \ref{DefinitionFaktorabilitaetGauss}, its associated rewriting system is complete. 
\end{Cor}

\begin{pf}
 Let $M$ be a monoid as above and $\E$ a generating set closed under least common left-multiples and left-complements. We have to show that $(xs)'=(x's)'$ and $\overline{xs}=\overline{x}\cdot \overline{x's}$ for any $x\in M$ and $s\in \E$. Recall that the factorability structure was defined via $\eta(x)=(x_p\ldots x_2, x_1)$ for $\NF(x)=x_p\ldots x_2x_1$ for $x\neq 1$ and $\eta(1)=(1,1)$, where $\NF$ denotes the $\E$-normal form. Recall furthermore that by Corollary \ref{NormalFormGreatestLeftDivisor}, $x_1$ can be characterized as the greatest right-divisor of $x$ lying in $\E$. 

So we have to compare the greatest right-divisors of $xs$ and $x_1s$ lying in $\E$. Call them $a$ and $b$, respectively. Observe that since $xs\succeq x_1s$, we can conclude that $a\succeq b$. Since $N_{\E}(x_1s)\leq 2$, we know that the $\E$-normal form of $x_1s$ has length at most $2$. So write $x_1s=tb$ with $t\in \E\cup \{1\}$. 
Furthermore, we can write $a=cb$ with $c\in \E\cup\{1\}$ since $\E$ is closed under left-complements. 

Now if $t=1$, then $a=cx_1s$ and there is a $d\in M$ such that 
\begin{eqnarray*}
 xs=x_p\ldots x_1s=da=dcx_1s,
\end{eqnarray*}
so that $cx_1$ right-divides $x$ by right-cancellativity. Now observe that since $s, a\in \E$, also $cx_1\in \E$ since $\E$ is closed under left-complements. Thus, $cx_1$ is a right-divisor of $x$ lying in $\E$ and by definition of $x_1$, it has to right-divide $x_1$. So we conclude $c=1$ and thus $a=b$ in the case $t=1$.

Assume now $t\neq 1$. Since $b$ is the greatest right-divisor of $x_1s$ lying in $\E$, it has to be right-divisible by the right-divisor $s\in \E$ of $x_1s$; so there is a $u\in \E$ such that $b=us$. By the same argument, there is a $v\in \E$ so that $a=vs$ holds. We can again find a $d\in M$ such that $xs=da$. Inserting $a=vs$ and cancelling on the right, we obtain $x=dv$. Hence, there is a $w\in E$ so that $x_1=wv$, by the definition of $x_1$. This implies that $x_1s=wvs=wa$. By the definition of $b$, this shows that $a$ is a right-divisor of $b$. Since we already had $a\succeq b$, the statement $a=b$ follows. This proves $(xs)'=(x's)'$.

Since $xs=\overline{xs}\cdot (xs)'$ on the one hand, and, on the other hand, 
\begin{eqnarray*}
 xs=\overline{x}x's=\overline{x}\cdot \overline{x's}\cdot (x's)',
\end{eqnarray*}
and since we have already shown $(xs)'=(x's)'$, right-cancellativity implies $\overline{x}\cdot \overline{x's}=\overline{xs}$. This completes the proof since now we are in the situation to use Theorem \ref{CompleteRewritingSystemWeak}.

\end{pf}

\begin{remark}
Note that the corollary above in particular provides a complete rewriting system on each Artin monoid. Besides the finite-type Artin monoids (\cite{HermillerMeier}), some other complete rewriting systems for subclasses of Artin monoids were known before (cf.\ \cite{BahlsSmith}).
\end{remark}

\section{Factorability Structure on Artin Monoids}
\label{Factorability Structure on Artin Monoids}

We are going to describe the factorability structures on Artin monoids in more detail. In particular, we are going to take a closer look at square-free elements which are an appropriate generating system for factorability. 
We will rely on the analysis of divisibility in Artin monoids due to E.~Brieskorn and K.~Saito (\cite{BrieskornSaito}). The following results will be needed. 

\begin{Prop}[\cite{BrieskornSaito}, Prop. 2.3]
 All Artin monoids are cancellative.
\end{Prop}

The following lemma gives a necessary condition for three generators to have a common multiple. (It turns out also to be sufficient.)
\begin{Lemma}[\cite{BrieskornSaito}, \S 2]\label{BrieskornSaitokgVEndlichkeit}
 An element $z$ of an Artin monoid $M(S)$ can only be divisible by three different letters $a,b,c\in S$ if the Coxeter matrix $M_{\{a,b,c\}}$ defines a finite Coxeter group. 
\end{Lemma}

The following basic observation is also important for us.

\begin{Lemma}[\cite{BrieskornSaito}, \S 3] \label{BrieskornSaitokgVBuchstaben}
 For two letters $a,b$, their least common multiple is given by $\Mf{a,b}^{m_{a,b}}=aba\ldots$ (a product of $m_{a,b}$ factors).
\end{Lemma}

Furthermore, we will need some results about the square-free elements. Recall that an element of an Artin monoid $M(S)$ is \textbf{square-free} if there is no word representing this element and containing a square of a generator $s^2$ ($s\in S$). We will denote the set of square-free elements in $M(S)$ by $\QF(S)$. A classical theorem often attributed to J.~Tits or H.~Matsumoto indicates the importance of square-free elements (cf.\ e.g. \cite{Humphreys}).

\begin{Theorem}[Tits, Matsumoto]
 Any two reduced words representing the same element in a Coxeter group $W(S)$ can be transformed into each other only using braid relations, i.e., without increasing the length of the word in between. 
\end{Theorem}

In particular, one obtains a set-theoretic section $\tau\colon W(S) \to \QF(S)\subset M(S)$ of the projection $\pi\colon M(S)\to W(S)$. One can use this to show that $\pi\colon \QF(S)\to W(S)$ is a bijection. 
Furthermore, $\tau$ is not a monoid homomorphism in general, but we have $\tau(uv)=\tau(u)\tau(v)$ whenever $l(uv)=l(u)+l(v)$ holds for the Coxeter length $l$ on $W(S)$.

The following proposition is proven by F.~Digne and J.~Michel (\cite{DigneMichel}), based on \cite{MichelNoteBraid}. 

\begin{Prop} (\cite{DigneMichel})\label{QFabgeschlossen}
 Let $M(S)$ be an Artin monoid. Then the set of all square-free elements $\QF(S)\subset M(S)$ is closed under least common left- and  right-multiples and left- and right-complement. In particular, there are $\QF(S)$-normal forms in $M(S)$ like described in Theorem \ref{NFGauss}. 
\end{Prop}

This implies the following corollary concerning factorability structures.
\begin{Cor}
 Let $M(S)$ be an Artin monoid. Then there is a factorability structure on $M(S)$ with respect to the generating system $QF(S)$ of all square-free elements. The prefix $\eta'(x)$ of an element $x \in M(S)$ is given by the maximal square-free right-divisor of $x$. 
\end{Cor}

In some cases, it is possible to show that $\QF(S)$ is the smallest subset of $M(S)$ closed under least common left-multiples and left-complements and containing $S$. For example, this is surely true for all Artin groups of finite type (cf.\ e.g. Section 6.6 of \cite{KasselTuraev}). However, this does not hold in general. An example where such a subset is again finite and strictly smaller than $\QF(S)$ is given in second author's thesis (\cite{MyThesis}). A general finiteness result for these subsets can be found in \cite{EndlichkeitArtin}.

\section{Factorability Structure on Garside Groups}
\label{Factorability structure on Garside groups}

We have already shown that there is a factorability structure on right-cancellative, left-noetherian monoids in which any two elements admitting a common left-multiple also admit a least common left-multiple. However, it is hard to transfer information about monoids into information about the corresponding groups of fractions in the general case. In this section, we show how this works in the case of Garside monoids. The key ingredient is the normal form described by P.~Dehornoy (\cite{DehornoyF}) as in Proposition \ref{NFGarside}. Furthermore, we will use the results by R.~Charney and J.~Meier concerning the lengths of normal forms.

In this section, let $M$ be a Garside monoid with a Garside element $\Delta$ and let $\D$ be the set of its non-trivial left-divisors. Let $G$ be the group of fractions of $M$.

We define now $\eta$ on $G$ as follows: For any $z\in G$, let $x_p\ldots x_1 y_1^{-1}\ldots y_q^{-1}$ be the normal form of Proposition \ref{NFGarside}. We set
\begin{eqnarray*}
\eta'(z) &= &y_q^{-1}\\
\overline{\eta}(z) &= &x_p\ldots x_1 y_1^{-1}\ldots y_{q-1}^{-1}
\end{eqnarray*}
in the case that $q\neq 0$ and 
\begin{eqnarray*}
\eta'(z) &=& x_1\\
\overline{\eta}(z)&=&x_p\ldots x_2
\end{eqnarray*}
for $q=0$ (equivalently, for $z\in M$).

\begin{Prop}\label{GarsideGruppenFaktorabel}
The map $\eta$ turns the group $G$ with the chosen generating system $\mathcal{E}$ into a factorable group. Here, $\E$ denotes the generating set $\mathcal{E}=\D\cup \D^{-1}$.
\end{Prop}

Before proving the proposition, we need the following lemma varying slightly the normal forms already known in the literature. Since we need to connect these normal forms with the ones in Proposition \ref{NFGarside}, we will give a proof of this lemma. 

\begin{Lemma}\label{NFLemma}
For any $z\in G$, there is a unique decomposition  
\begin{eqnarray*}
 z=w_r\ldots w_1 \Delta^{-m}
\end{eqnarray*}
 with $m\geq 0$ where $w_r\ldots w_1 \in M$ is in its normal form like in Proposition \ref{NFGarside}, and $w_r\ldots w_1\nsucceq \Delta$ if $m>0$. If the normal form of $z$ described in Proposition \ref{NFGarside} is given by 
\begin{eqnarray*}
 x_p\ldots x_1 y_1^{-1}\ldots y_q^{-1}
\end{eqnarray*}
and $q>0$, then $m=q$ and $r\leq p+q$. Furthermore, if in this case $z$ is not a power of $\Delta$, the relation between $w_1$ and $y_q$ is given by $w_1=\varphi^{q-1}(y_q^*)$ or, equivalently, $y_q^{-1}=\Delta^{-1}\varphi^{-q}(w_1)$.
\end{Lemma}

\begin{pf}
The uniqueness is a direct consequence of the uniqueness of the normal form in Proposition \ref{NFGarside}. 

If $q=0$ (equivalently, $z\in M$), the existence of the decomposition is immediate. So we assume $q>0$ (equivalently, $z\notin M$). 

We start with the normal form of $z$ as in Proposition \ref{NFGarside} and apply the observations of Remark \ref{Stern}. In particular, we know for each divisor $s$ of $\Delta$ that $s^{-1}=s^*\Delta^{-1}$ holds. 
\begin{eqnarray*}
z&=&x_p\ldots x_1 y_1^{-1}\ldots y_q^{-1}\\
&=&x_p\ldots x_1 y_1^*\Delta^{-1}\ldots y_q^*\Delta^{-1}\\
&=&x_p\ldots x_1y_1^*\varphi(y_2^*)\ldots \varphi^{q-1}(y_q^*)\Delta^{-q}.
\end{eqnarray*}
Observe that all $\varphi^{i-1}(y_i^*)$ are divisors of $\Delta$ since $\varphi(\Delta)=\Delta$ and $y_i^*$ are divisors of $\Delta$. However, it can happen that $\varphi^{i-1}(y_i^*)=1$. Since $\varphi$ is a monoid automorphism, this is equivalent to $y_i^*=1$, and this in turn is equivalent to $y_i=\Delta$. Observe that if $p>0$, this cannot occur. Indeed, if $y_j=\Delta$, then we have
\begin{eqnarray*}
  x_p\ldots x_1 y_1^{-1}\ldots y_q^{-1}= x_p\ldots x_2(x_1\Delta^{-1}) \varphi^{-1}(y_1^{-1})\ldots \varphi^{-1}(y_{j-1}^{-1}) y_{j+1}^{-1}\ldots y_q^{-1}
\end{eqnarray*}
which would contradict the fact that this normal form is geodesic (cf.\ Lemma \ref{Norm}) since $x_1\Delta^{-1}\in \E\cup \{1\}$. Now let $p=0$. Then $z^{-1}=y_q\ldots y_1$ is a right greedy normal form normal form of Proposition \ref{NFGarside}. In it, all $\Delta$'s are on the very right. Indeed, if an element of $M$ can be written in the form $a\Delta b$, then it is equal to $a\varphi^{-1}(b)\Delta$. It follows from the definition of greediness that if the normal form of $z^{-1}$ contains a $\Delta$, then $\Delta=\rgcd(\Delta, z^{-1})=y_1$. Iterating this argument, we see that all $\Delta$'s in the normal form are on the very right. 

Denote now by $w_r,\ldots, w_1$ the letters $x_p, \ldots, x_1, y_1^*, \varphi(y_2^*), \ldots, \varphi^{q-1}(y_q^*)$ if $p>0$ and the letters $\varphi^{j-1}(y_j^*),\ldots, \varphi^{q-1}(y_q^*)$ for $p=0$ where $j$ is the smallest index such that $y_j\neq \Delta$. 

We check that $z\Delta^{q}$ has $w_r\ldots w_1$ as its normal form. Indeed, the equation 
\begin{eqnarray*}
\rgcd({}^*x_k,x_{k+1})=1
\end{eqnarray*}
 holds by definition. Next,  we obtain
\begin{eqnarray*}
\rgcd({}^*(y_1^*), x_1)=\rgcd(y_1, x_1)=1.
\end{eqnarray*}
At last, we note that Remark \ref{Stern} implies $x^*={}^*\varphi(x)$ and $\varphi(x^*)=\varphi(x)^*$ for all $\Delta\succeq x$. Recall furthermore from Remark \ref{Stern} that for all $t\in \mathcal{D}$, the identity $({}^*t)^*=t$ holds. Analogously, also the identity ${}^*(t^*)=t$ holds true. We use this to check
\begin{eqnarray*}
\rgcd({}^*\varphi^{i}(y_{i+1}^*), \varphi^{i-1}(y_i^*)) &=&\rgcd(\varphi^{i}(y_{i+1}), \varphi^i(^*y_i))\\
&=&\varphi^i(\rgcd({}^*y_i, y_{i+1}))=1.
\end{eqnarray*}
So, $x_p\ldots x_1y_1^*\varphi(y_2^*)\ldots \varphi^{q-1}(y_q^*)$ is a normal form. 

Assume now, for contradiction, $x_p\ldots x_1y_1^*\varphi(y_2^*)\ldots \varphi^{q-1}(y_q^*)\succeq \Delta$ holds. Then we know that $\varphi^{q-1}(y_q^*)=\Delta$ by the second part of Proposition \ref{NFGarside}, thus also $y_q^*=\Delta$ and $y_q=1$ follows. This contradicts the choice of $y_i$. 

The uniqueness of the new normal form now follows from the uniqueness of the normal form in Proposition \ref{NFGarside}.

We have defined $w_1$ to be $\varphi^{q-1}(y_q^*)$. We want to prove the equivalent formulation $y_q^{-1}=\Delta^{-1}\varphi^{-q}(w_1)$. The definition is equivalent to
\begin{eqnarray*}
\varphi^{1-q}(w_1)=y_q^*.
\end{eqnarray*}
We can rewrite $y_qy_q^*=\Delta$ to $y_q^*=y_q^{-1}\Delta$. Inserting this into the equation above and multiplying with $\Delta^{-1}$, we obtain 
\begin{eqnarray*}
y_q^{-1}=\varphi^{1-q}(w_1)\Delta^{-1}=\Delta^{-1}\varphi^{-q}(w_1).
\end{eqnarray*}
This completes the proof.
\end{pf}

The following easy fact will be used several times in the proof of Proposition \ref{GarsideGruppenFaktorabel}:
\begin{Lemma}\label{ProduktZweierTeilerStern}
 Let $M$ be a Garside monoid and let $a,b\in \mathcal{D}$ be two divisors of $\Delta$. Then $a^*\preceq b$ is equivalent to $ab\succeq \Delta$. Indeed, in this case one can even show that there is a $t\in\mathcal{D}\cup \{1\}$ so that $ab=t\Delta$. 
\end{Lemma}

\begin{pf}

The statement $a^*\preceq b$ is equivalent to the existence of a $d\in M$ with $a^*d=b$. Due to cancellativity, this is equivalent to the existence of $d\in M$ with $\Delta d=aa^*d=ab$, and this is in turn equivalent to the existence of $d\in M$ with $ab=\Delta d=\varphi^{-1}(d) \Delta$. This proves the first statement.

Now we assume these equivalent conditions hold. Since $b$ is a right-divisor of $\Delta$ and $d$ is a right-divisor of $b$, the element $d$ lies necessarily in $\mathcal{D}\cup \{1\}$. Since $\varphi^{-1}$ maps divisors of $\Delta$ to divisors of $\Delta$, this completes the proof. 

\end{pf}

We are now ready to prove the proposition above. 

\begin{pf}(of Proposition \ref{GarsideGruppenFaktorabel})
Since we are dealing with a group, which is in particular right-cancellative, we only have to show that $\eta$ is a weak factorability structure.

Note that for a non-trivial $z$ with normal form 
\begin{eqnarray*}
x_p\ldots x_1y_1^{-1}\ldots y_{q}^{-1}
\end{eqnarray*}
 the element $y_q^{-1}$ is in $\D^{-1}$ and $x_1 \in \D$, so (F3) of Definition \ref{FactDefMonoid} holds. Furthermore, (F1) holds by definition. By Lemma \ref{Norm} and since by definition $\overline{\eta}(z)$ has the normal form $x_p\ldots x_1 y_1^{-1}\ldots y_{q-1}^{-1}$ or $x_p\ldots x_2$, we observe that (F2) is fulfilled.

We now proceed to show that the last condition \ref{eqWF}. Note that if the two compositions in the diagram coincide, they are in particular equal in graded sense. 

We have to distinguish several cases. 
\begin{enumerate}
\item First, take $z\in M, s\in \D$. We use Proposition \ref{NFGarside} and Lemma \ref{Prod} to see:
\begin{eqnarray*}
\eta'(zs)=\rgcd(zs, \Delta)=\rgcd(\rgcd(z, \Delta)s, \Delta)=\eta'(\eta'(z)s).
\end{eqnarray*}
This implies that on such elements, both compositions from Definition \ref{FactDefMonoid} coincide.

\item Now assume $z \notin M$, $s\in \D$. Recall that by Lemma \ref{FactAxiom}, we have to show that if the pairs $(\eta'(z), s)$ and $(\overline{\eta}(z), \overline{\eta}(\eta'(z)s))$ are geodesic pairs, then $(z,s)$ is a geodesic pair and that in this case $\eta'(\eta'(z)s)=\eta'(zs)$ holds. So we assume the pairs $(\eta'(z), s)$ and $(\overline{\eta}(z), \overline{\eta}(\eta'(z)s))$ are geodesic.

 We use the normal form $z=w_{r}\ldots w_1\Delta^{-q}$, $q>0$, of Lemma \ref{NFLemma}. Observe that we can write
\begin{eqnarray*}
zs=w_{r}\ldots w_1\Delta^{-q}s=w_{r}\ldots w_1\varphi^{q}(s)\Delta^{-q}.
\end{eqnarray*}
This implies in particular that 
\begin{eqnarray*}
 N_{\E}(zs)\leq \max\{N_{\mathcal{E}}(w_{r}\ldots w_1\varphi^{q}(s)), q\}
\end{eqnarray*}
by Lemma \ref{NormExplizit}. (Recall that for elements in $M$, the word length with respect to $\E$ is the same as the word length with respect to $\D$.)

Now we have to analyze when the pairs $(\eta'(z), s)=(y^{-1}_q, s)$ and  $(\overline{\eta}(z), \overline{\eta}(\eta'(z)s))$ are geodesic. We show first that if $r=0$, the pair $(\eta'(z), s)$ is non-geodesic. If $r=0$, we have $z=\Delta^{-q}$ and $\eta'(z)=\Delta^{-1}$. This implies
\begin{eqnarray*}
 N_{\E}(\eta'(z)s)=N_{\E}(\Delta^{-1}s)=N_{\E}\left((s^*)^{-1}\right)\leq 1.
\end{eqnarray*}
So we don't need to consider the case $r=0$.

Assume from now on $r\geq 1$, i.e., $w_1\neq 1$. By Lemma \ref{NFLemma}, we have 
\begin{eqnarray*}
 y^{-1}_q=\Delta^{-1}\varphi^{-q}(w_1)=\varphi^{-q+1}(w_1)\Delta^{-1}.
\end{eqnarray*}
The product $\eta'(z)s$ can thus be written as
\begin{eqnarray*}
 \eta'(z)s = \varphi^{-q+1}(w_1)\Delta^{-1}s=\varphi^{-q+1}(w_1\varphi^{q}(s))\Delta^{-1}.
\end{eqnarray*}

Since we assumed the pair $(\eta'(z), s)$ to be geodesic, the element $\eta'(z)s$ has length $2$. So we use Lemma \ref{NormExplizit} to see 
\hspace{1cm}
\begin{eqnarray*}
 2=N_{\E}\left (\varphi^{-q+1}(w_1\varphi^{q}(s))\Delta^{-1}\right )\leq \max\{N_{\D}(\varphi^{-q+1}(w_1\varphi^{q}(s))), 1\}.
\end{eqnarray*}
Thus, we know that the $\D$-length of $\varphi^{-q+1}(w_1\varphi^{q}(s))$ is at least $2$. Since $\varphi$ is an automorphism of $M$ mapping $\D$ to itself, we conclude that $w_1\varphi^{q}(s)$ has at least $\D$-length $2$. But as a product of two elements of $\D$, it has $\D$-norm exactly $2$. In particular, we know that the above inequality is in this case an equality, so by Lemma \ref{NormExplizit}, the element $\varphi^{-q+1}(w_1\varphi^{q}(s))$ and thus also $w_1\varphi^{q}(s)$ is not divisible by $\Delta$.

Denote $\eta(w_1\varphi^{q}(s))$ by $(a,b)$. Since $w_1\varphi^{q}(s)$ has norm $2$, the pair $(a,b)$ is already the normal form of $w_1\varphi^{q}(s)$. In particular, $a$ and $b$ are in $\D$, and by the definition of the normal form in Proposition \ref{NFGarside}, we have $\rgcd({}^*b, a)=1$. Moreover, since $w_1\varphi^{q}(s)$ is not divisible by $\Delta$, the prefix $b$ is not equal to $\Delta$. We would like to express $\eta'(\eta'(z)s)$ in $a$ and $b$ using Lemma \ref{NFLemma}. Since $\varphi^{-q+1}(w_1\varphi^{q}(s))$ is not divisible by $\Delta$, the normal form of $\eta'(z)s$ in the sense of Lemma \ref{NFLemma} is given by $\NF(\varphi^{-q+1}(w_1\varphi^{q}(s)))\Delta^{-1}$. Observe that
\begin{eqnarray*}
  \eta'(z)s & = & \varphi^{-q+1}(w_1\varphi^{q}(s))\Delta^{-1} \\
  &=&\varphi^{-q+1}(ab)\Delta^{-1}=\varphi^{-q+1}(a)\varphi^{-q+1}(b) \Delta^{-1}.
\end{eqnarray*}
We check that $\varphi^{-q+1}(a)\varphi^{-q+1}(b)$ is the normal form of $\varphi^{-q+1}(w_1\varphi^{q}(s))$. We only have to show that $\varphi^{-q+1}(b)$ is the right-most letter of the normal form since the element $\varphi^{-q+1}(w_1\varphi^{q}(s))$ has norm $2$. Now we have: 
\begin{eqnarray*}
 \rgcd(\varphi^{-q+1}(a)\varphi^{-q+1}(b), \Delta)=\varphi^{-q+1}(\rgcd(ab, \Delta))=\varphi^{-q+1}(b).
\end{eqnarray*}
So we know by Lemma \ref{NFLemma} that 
\begin{eqnarray*}
 \eta'(\eta'(z)s)=\Delta^{-1}\varphi^{-1}(\varphi^{-q+1}(b))=\Delta^{-1}\varphi^{-q}(b)=\varphi^{-q+1}(b)\Delta^{-1}
\end{eqnarray*}
 and thus $\overline{\eta}(\eta'(z)s)=\varphi^{-q+1}(a)$.

We now look at the second pair. By definition, $\overline{\eta}(z)$ is given by $zy_q$. We want to compute its normal form as given in Lemma \ref{NFLemma}. We have
\begin{eqnarray*}
 zy_q=w_r\ldots w_1\Delta^{-q}y_q= w_r\ldots w_1\varphi^q(y_q)\Delta^{-q}.
\end{eqnarray*}

We look closer at the product $w_1\varphi(y_q)$ using Lemma \ref{NFLemma} once again:
\begin{eqnarray*}
 w_1\varphi^q(y_q)=\varphi^{q-1}(y^*_q)\varphi^q(y_q)=\varphi^{q-1}(y^*_q\varphi(y_q)).
\end{eqnarray*}
We can now simplify $y^*_q\varphi(y_q)$ by observing that 
\begin{eqnarray*}
 y_qy^*_q\varphi(y_q)=\Delta\varphi(y_q)=y_q\Delta,
\end{eqnarray*}
which immediately implies $y^{*}_q\varphi(y_q)=\Delta$. 
So we obtain in total 
\begin{eqnarray*}
 zy_q= w_r\ldots w_1\varphi^q(y_q)\Delta^{-q}=w_r\ldots w_2\Delta^{-q+1},
\end{eqnarray*}
which is now again a normal form in the sense of Lemma \ref{NFLemma}. Here, we use that $w_r\ldots w_2$ is not divisible by $\Delta$ since $w_r\ldots w_1$ is not divisible by $\Delta$. Hence, the $\E$-length of $w_r\ldots w_2\Delta^{-q+1}$ is by Lemma \ref{NormExplizit} given by $\max\{r-1, q-1\}$.

Since the pair 
\begin{eqnarray*}
(\overline{\eta}(z), \overline{\eta}(\eta'(z)s))= (w_r\ldots w_2\Delta^{-q+1}, \varphi^{-q+1}(a))
\end{eqnarray*}
 is geodesic, we know that the norm of 
 \begin{eqnarray*}
 w_r\ldots w_2\Delta^{-q+1}\cdot\varphi^{-q+1}(a)=w_r\ldots w_2a\Delta^{-q+1}
 \end{eqnarray*}
must be equal to $\max\{r-1, q-1\}+1=\max\{r, q\}$. Combining this with Lemma \ref{NormExplizit} applied to $w_r\ldots w_2a\Delta^{-q+1}$, we see that
\begin{eqnarray*}
\max\{r,q\} \leq \max\{N_{\D}(w_r\ldots w_2a), q-1\}\leq\max\{r, q-1\}.
\end{eqnarray*}
This, in turn, implies $r=\max\{r,q\} =\max\{r, q-1\}$. So we conclude that the norm of $w_r\ldots w_2a$ is $r$ and by the second part of Lemma \ref{NormExplizit}, $\Delta$ does not right-divide $w_r\ldots w_2a$. 

Recall that we need to calculate $\eta'(zs)$, so it is helpful to consider the normal form of $w_r\ldots w_1\varphi^q(s)=zs\Delta^q$. Using Lemma \ref{Prod} and Proposition \ref{NFGarside} once again, we see that the normal form of $w_r\ldots w_1\varphi^q(s)$ ends with 
\begin{eqnarray*}
\rgcd(w_r\ldots w_1\varphi^q(s), \Delta) &=&\rgcd(\rgcd(w_r\ldots w_1, \Delta)\varphi^q(s), \Delta)\\
&=&\rgcd(w_1\varphi^q(s),\Delta)=b.
\end{eqnarray*}
In particular, since we have seen $b\neq \Delta$, it follows that $\Delta$ does not right-divide $w_r\ldots w_1\varphi^q(s)$.
So the normal form of $w_r\ldots w_1\varphi^q(s)$ is given by
\begin{eqnarray*}
\NF(w_r\ldots w_1\varphi^q(s))=\NF(w_r\ldots w_2a)b.
\end{eqnarray*}
Since $\NF(w_r\ldots w_2a)$ has length $r$ (e.g. by Lemma \ref{LLGaussschNFNorm}), the normal form of $w_r\ldots w_1\varphi^q(s)$ is of length $r+1$ and the norm of 
\begin{eqnarray*}
zs=w_r\ldots w_1\varphi^q(s)\Delta^{-q}
\end{eqnarray*}
is $\max\{q,r+1\}=r+1$. So we have shown that $(z,s)$ is a geodesic pair. Using Lemma \ref{NFLemma}, we can compute $\eta'(zs)$: It is given by 
\begin{eqnarray*}
\Delta^{-1}\varphi^{-q}(b)=\varphi^{-q+1}(b)\Delta^{-1}.
\end{eqnarray*}
This coincides with $\eta'(\eta'(z)s)$. So we have proven the statement for this case.

\item Consider now the case $z \in M$, $s \in \D^{-1}$. There is a $t\in \D\cup\{1\}\setminus \{\Delta\}$ with $s=t\Delta^{-1}$, namely, if $s=u^{-1}$ with $u\in \D$, we have $t=u^*$. Let $z=x_p\ldots x_1$ be the normal form of $z$. Then $zs=x_p\ldots x_1 t \Delta^{-1}$. 

If $x_1^*\preceq t$, then there is by Lemma \ref{ProduktZweierTeilerStern} an element $c \in \D\cup \{1\}$ such that $x_1t=c\Delta$. Therefore, 
\begin{eqnarray*}
\eta'(z)\cdot s=x_1t\Delta^{-1}=c
\end{eqnarray*}
so the pair $(\eta'(z), s)$ is not geodesic. Thus, we can use Lemma \ref{FactAxiom} to conclude in this case.

If $x_1^*\npreceq t$, then Lemma \ref{ProduktZweierTeilerStern} implies that $x_1t\nsucceq \Delta$. We show first that this implies $x_p\ldots x_1 t$ is not right-divisible by $\Delta$. Suppose it were, then there is a $u\in M$ so that
\begin{eqnarray*}
x_p\ldots x_1t=u\Delta=u\cdot {}^*t t,
\end{eqnarray*}
and by cancellation, ${}^*t\in \D$ is a right-divisor of $x_p\ldots x_1$. By definition, $x_1$ is the greatest divisor of $x_p\ldots x_1$ lying in $\D$, so this would imply the existence of some $v\in M$ with $x_1=v\cdot {}^*t$, and $x_1t=v\cdot{}^*tt=v\Delta$, contradicting our assumptions. 
 
Hence, $\NF(x_p\ldots x_1 t)\Delta^{-1}$ is a normal form in the sense of Lemma \ref{NFLemma}. The normal form $\NF(x_p\ldots x_1 t)$ ends with 
\begin{eqnarray*}
 \rgcd(x_p\ldots x_1 t, \Delta)=\rgcd(\rgcd(x_p\ldots x_1, \Delta)t, \Delta)=\rgcd(x_1t,\Delta),
\end{eqnarray*}
which we will denote by $a$. Here, we use Lemma \ref{Prod} for the first equality. Thus by Lemma \ref{NFLemma} and by definition of $\eta$, we have $\eta'(zs)=\Delta^{-1}\varphi^{-1}(a)$. On the other hand,
\begin{eqnarray*}
 \eta'(\eta'(z)\cdot s)=\eta'(x_1t \Delta^{-1})=\Delta^{-1}\varphi^{-1}(\rgcd(x_1t, \Delta))=\Delta^{-1}\varphi^{-1}(a).
\end{eqnarray*}
This implies that on such elements $z$ and $s$, the maps from the Definition of factorability coincide, so the diagram \ref{eqWF} commutes.

\item Last, we come to the case $z\notin M$, $s \in \D^{-1}$. There is again a $t\in \D\cup\{1\}\setminus \{\Delta\}$ with $s=t\Delta^{-1}$. Let as before $z=w_{r}\ldots w_1\Delta^{-q}$, $q>0$, be the normal form of Lemma \ref{NFLemma}. Then $zs=w_{r}\ldots w_1\varphi^{q}(t)\Delta^{-q-1}$. Again, two cases are possible. 

If $w_1^*\preceq\varphi^{q}(t)$, we have $w_1\varphi^q(t)=d\Delta$ for some $d\in \D\cup\{1\}$ by Lemma \ref{ProduktZweierTeilerStern}. 
We obtain using Lemma \ref{NFLemma}:
\begin{eqnarray*}
\eta'(z)s&=&\Delta^{-1}\varphi^{-q}(w_1) t\Delta^{-1}\\
&=&\Delta^{-1}\varphi^{-q}(w_1\varphi^{q}(t))\Delta^{-1}=\Delta^{-1}\varphi^{-q}(d) \in \D^{-1}\cup\{1\}.
\end{eqnarray*}
So $(\eta'(z), s)$ is not geodesic. Thus, we can apply Lemma \ref{FactAxiom}.

If $w_1^*\npreceq\varphi^{q}(t)$, we conclude again that $\Delta$ is not a divisor of $w_{r}\ldots w_1\varphi^{q}(t)$. Therefore, $\NF(w_{r}\ldots w_1\varphi^{q}(t))\Delta^{-q-1}$ is the normal form of $zs$ in the sense of Lemma \ref{NFLemma}. We compute again as before:
\begin{eqnarray*}
 \rgcd(w_{r}\ldots w_1\varphi^{q}(t), \Delta)=\rgcd(w_1\varphi^{q}(t), \Delta)=:a.
\end{eqnarray*}
Thus by Lemma \ref{NFLemma}, we have $\eta'(zs)=\Delta^{-1}\varphi^{-q-1}(a)=\varphi^{-q}(a)\Delta^{-1}$.

On the other hand, we have to determine $\eta'(\eta'(z)\cdot s)$. First, we can write $\eta'(z)s$ as follows using Lemma \ref{NFLemma}:
\begin{eqnarray*}
 \eta'(z)\cdot s&=&\varphi^{-q+1}(w_1)\Delta^{-1}s\\
 &=&\varphi^{-q+1}(w_1)\Delta^{-1}t\Delta^{-1}=\varphi^{-q+1}(w_1)\varphi(t)\Delta^{-2}.
\end{eqnarray*}
We now want to compute its prefix. By Lemma \ref{ProduktZweierTeilerStern}, our assumption $w_1^*\npreceq\varphi^{q}(t)$ implies $w_1\varphi^{q}(t)\nsucceq \Delta$, and thus also
\begin{eqnarray*}
\varphi^{-q+1}(w_1)\varphi(t)=\varphi^{-q+1}(w_1\varphi^{q}(t))
\end{eqnarray*}
is not divisible by $\Delta$. So $\NF(\varphi^{-q+1}(w_1)\varphi(t))\Delta^{-2}$ is the normal form of $\eta'(z)s$ in sense of Lemma \ref{NFLemma}. We use this lemma as well as the Proposition \ref{NFGarside} once again to conclude:
\begin{eqnarray*}
 \eta'(\eta'(z) s)=\varphi^{-1} \left (\rgcd(\varphi^{-q+1}(w_1)\varphi(t), \Delta)\right)\Delta^{-1}\\
=\varphi^{-q}\left(\rgcd(w_1\varphi^{q}(t), \Delta)\right)\Delta^{-1}=\varphi^{-q}(a)\Delta^{-1}
\end{eqnarray*}
This implies that on such elements $z$ and $s$, the maps of the diagram \ref{eqWF} coincide. This completes the case distinction.

\end{enumerate}
\end{pf}

In the case of Garside groups, there are already complete rewriting systems describing them due to S.~Hermiller and J.~Meier (\cite{HermillerMeier}). 
\begin{remark}
Let $M$ be a Garside monoid with a Garside element $\Delta$. Let $G$ be the group of fractions of $M$, and let $\mathcal{D}$ be the set of left-divisors of $\Delta$ except for $1$. Then Lemma \ref{FactorabilityStructureInducesRewriting} allows us to associate a rewriting system with the factorability structure on $(G, \mathcal{D}\cup\mathcal{D}^{-1})$ described above. This rewriting system is exactly the second rewriting system $R_2$ for Garside groups in \cite{HermillerMeier}. There, this rewriting system is shown to be complete. By this argument, we know that the rewriting systems associated with factorability structures in Garside groups are complete. 
\end{remark}

\newpage
\begin{appendix}
\section*{Appendix: Factorable Monoid with Non-Complete Rewriting System}
\label{Factorable Monoid with Non-Complete Rewriting System}
We present now an example of a factorable monoid where the associated rewriting system is not noetherian. We prove the factorability by exhibiting a local factorability structure and using then the criterion \ref{PhiFaktorabilitaetEtaFaktorabilitaet}. We will also need the following two lemmas, which are used in the proof of Rodenhausen's Theorem. The first one reduces the fifth condition of local factorability to a few cases.

\begin{Lemma} \label{Bedingung5PhiFakt}
 Let $\mathcal{E}$ be a set. Let 
\begin{eqnarray*}
 \varphi\colon \mathcal{E}^+\times \E^+ \to \E^+\times \E^+
\end{eqnarray*}
be a map with $\varphi^2=\varphi$ and $\varphi(a,1)=(1,a)$ and satisfying stability for triples. (In other words, we start with any set $\E$ and a map $\varphi$ satisfying the second, third and fourth conditions of Definition \ref{Phifaktorabilitaet}.) Furthermore, we define the normal form function $\NF\colon \E^*\to \E^*$ as in Definition \ref{Phifaktorabilitaet}. Then the fifth condition of Definition \ref{Phifaktorabilitaet} is satisfied for a triple $(a,b,c)$ whenever $(a,b)$ or $(b,c)$ is a stable pair. 
\end{Lemma}

The second lemma says how to compute normal forms, even if $\varphi$ is not shown to be a local factorability map. 

\begin{Lemma} \label{NormalFormTriple}
Let $\mathcal{E}$ be a set. Let 
\begin{eqnarray*}
 \varphi\colon \mathcal{E}^+\times \E^+ \to \E^+\times \E^+
\end{eqnarray*}
be a map with $\varphi^2=\varphi$ and $\varphi(a,1)=(1,a)$ for which stability of triples holds. (In other words, we start with any set $\E$ and a map $\varphi$ satisfying the second, third and fourth conditions of Definition \ref{Phifaktorabilitaet}.) Furthermore, we define the normal form function $\NF\colon \E^*\to \E^*$ as in Definition \ref{Phifaktorabilitaet}. Then it follows that an application of $\varphi_1\varphi_2\varphi_1\varphi_2$ to a triple $(a,b,c)$ yields an extended normal form for this triple.
\end{Lemma}

Since the result is quite technical, we first show how the example was constructed. We want to find a cycle in the system of rewritings when $(\varphi_3\varphi_2\varphi_1\varphi_2)^{N}$ is applied. (This can be checked to be the only candidate to produce a cycle when applied to a $4$-tuple.) One can construct a monoid with a cycle of length $N=1$, but naive attempts yield then a not right-cancellative monoid. So we will force the cycle to have length $2$. Thus, we have necessarily rewritings of the form depicted on the next page. \\

Furthermore, one can see that the pairs $(a_2, b_3)$ and $(a_1, b_6)$ should not be geodesic since otherwise, the rewritings have to stabilize. The further rewriting rules arose during the proof. This should justify the definition in the next proposition.\newpage

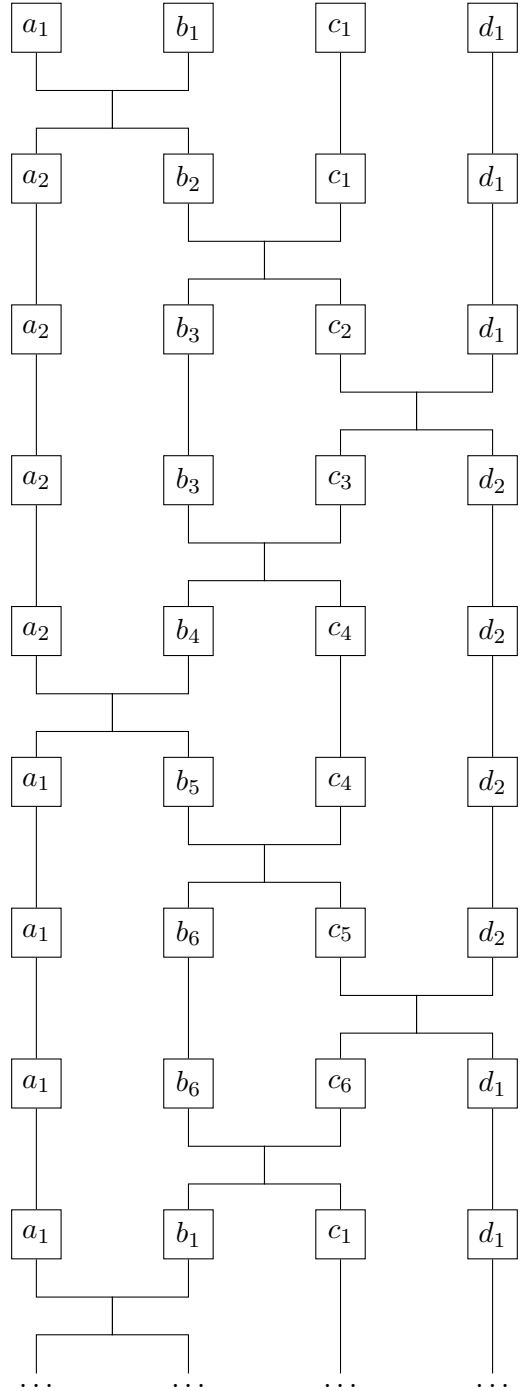
\begin{center}
     \begin{tikzpicture}[node distance = 2cm, auto]
 \node[cloud](a1){$a_1$};
 \node[cloud, right of=a1](b1){$b_1$};
 \node[cloud, right of=b1](c1){$c_{1}$};
  \node[cloud, right of=c1](d1){$d_{1}$};
  
 \node[cloud, below of=a1](a2){$a_{2}$};
 \node[cloud, below of=b1](b2){$b_{2}$};
  \node[cloud, below of=c1](c2){$c_{1}$};
   \node[cloud, below of=d1](d2){$d_{1}$};
   
  \node[cloud, below of=a2](a3){$a_{2}$};
 \node[cloud, below of=b2](b3){$b_{3}$};
 \node[cloud, below of=c2](c3){$c_{2}$};
  \node[cloud, below of=d2](d3){$d_{1}$};
 
  \node[cloud, below of=a3](a4){$a_{2}$};
 \node[cloud, below of=b3](b4){$b_3$};
 \node[cloud, below of=c3](c4){$c_3$};
  \node[cloud, below of=d3](d4){$d_2$};
  
  \node[cloud, below of=a4](a7){$a_2$};
 \node[cloud, below of=b4](b7){$b_4$};
 \node[cloud, below of=c4](c7){$c_4$};
 \node[cloud, below of=d4](d7){$d_2$};

    \node[cloud, below of=a7](a8){$a_{1}$};
 \node[cloud, below of=b7](b8){$b_{5}$};
 \node[cloud, below of=c7](c8){$c_{4}$};
 \node[cloud, below of=d7](d8){$d_{2}$};
 
     \node[cloud, below of=a8](a9){$a_{1}$};
 \node[cloud, below of=b8](b9){$b_{6}$};
 \node[cloud, below of=c8](c9){$c_{5}$};
 \node[cloud, below of=d8](d9){$d_{2}$};
 
      \node[cloud, below of=a9](a10){$a_{1}$};
 \node[cloud, below of=b9](b10){$b_{6}$};
 \node[cloud, below of=c9](c10){$c_{6}$};
 \node[cloud, below of=d9](d10){$d_{1}$};
 
       \node[cloud, below of=a10](a11){$a_{1}$};
 \node[cloud, below of=b10](b11){$b_{1}$};
 \node[cloud, below of=c10](c11){$c_{1}$};
 \node[cloud, below of=d10](d11){$d_{1}$};
 
    \node[below of=a11](a12){$\ldots$};
 \node[below of=b11](b12){$\ldots$};
 \node[below of=c11](c12){$\ldots$};
 \node[below of=d11](d12){$\ldots$};

\draw (c1) -- (c2); 
\draw (d1) -- (d2); 
    \draw (a1.south) -- ++(0.0, -0.5) -| ++(1.0, -0.5) -- ++(0.0,-0.0)-- ++(-1.0, 0.)--  (a2.north);
    \draw (b1.south) -- ++(0.0,-0.5) -| ++(-1.0, -0.5) -- ++(0.0,-0.0)-- ++(1.0, 0.)--  (b2.north);
\draw (d2) -- (d3); 
\draw (a2) -- (a3); 
    \draw (b2.south) -- ++(0.0, -0.5) -| ++(1.0, -0.5) -- ++(0.0,-0.0)-- ++(-1.0, 0.)--  (b3.north);
    \draw (c2.south) -- ++(0.0,-0.5) -| ++(-1.0, -0.5) -- ++(0.0,-0.0)-- ++(1.0, 0.)--  (c3.north);
\draw (a3) -- (a4); 
\draw (b3) -- (b4);
    \draw (c3.south) -- ++(0.0, -0.5) -| ++(1.0, -0.5) -- ++(0.0,-0.0)-- ++(-1.0, 0.)--  (c4.north);
    \draw (d3.south) -- ++(0.0,-0.5) -| ++(-1.0, -0.5) -- ++(0.0,-0.0)-- ++(1.0, 0.)--  (d4.north);
\draw (a4) -- (a7); 
\draw (d4) -- (d7); 
    \draw (b4.south) -- ++(0.0, -0.5) -| ++(1.0, -0.5) -- ++(0.0,-0.0)-- ++(-1.0, 0.)--  (b7.north);
    \draw (c4.south) -- ++(0.0,-0.5) -| ++(-1.0, -0.5) -- ++(0.0,-0.0)-- ++(1.0, 0.)--  (c7.north);
    
\draw (c7) -- (c8); 
\draw (d7) -- (d8); 
    \draw (a7.south) -- ++(0.0, -0.5) -| ++(1.0, -0.5) -- ++(0.0,-0.0)-- ++(-1.0, 0.)--  (a8.north);
    \draw (b7.south) -- ++(0.0,-0.5) -| ++(-1.0, -0.5) -- ++(0.0,-0.0)-- ++(1.0, 0.)--  (b8.north);
\draw (d8) -- (d9); 
\draw (a8) -- (a9); 
    \draw (b8.south) -- ++(0.0, -0.5) -| ++(1.0, -0.5) -- ++(0.0,-0.0)-- ++(-1.0, 0.)--  (b9.north);
    \draw (c8.south) -- ++(0.0,-0.5) -| ++(-1.0, -0.5) -- ++(0.0,-0.0)-- ++(1.0, 0.)--  (c9.north);
\draw (a9) -- (a10); 
\draw (b9) -- (b10);
    \draw (c9.south) -- ++(0.0, -0.5) -| ++(1.0, -0.5) -- ++(0.0,-0.0)-- ++(-1.0, 0.)--  (c10.north);
    \draw (d9.south) -- ++(0.0,-0.5) -| ++(-1.0, -0.5) -- ++(0.0,-0.0)-- ++(1.0, 0.)--  (d10.north);
\draw (a10) -- (a11); 
\draw (d10) -- (d11); 
    \draw (b10.south) -- ++(0.0, -0.5) -| ++(1.0, -0.5) -- ++(0.0,-0.0)-- ++(-1.0, 0.)--  (b11.north);
    \draw (c10.south) -- ++(0.0,-0.5) -| ++(-1.0, -0.5) -- ++(0.0,-0.0)-- ++(1.0, 0.)--  (c11.north);
    
\draw (c11) -- (c12); 
\draw (d11) -- (d12); 
    \draw (a11.south) -- ++(0.0, -0.5) -| ++(1.0, -0.5) -- ++(0.0,-0.0)-- ++(-1.0, 0.)--  (a12.north);
    \draw (b11.south) -- ++(0.0,-0.5) -| ++(-1.0, -0.5) -- ++(0.0,-0.0)-- ++(1.0, 0.)--  (b12.north);
 \end{tikzpicture}
    \end{center}
 \newpage
 
\begin{Prop}
 Let $\mathcal{E}$ be the following set:
\begin{eqnarray*}
 \{a_1, a_2, b_{1}, b_2, b_3, b_4, b_5, b_6, c_1, c_2, c_3, c_4, c_5, c_6, d_{1}, d_2, e_{2}, e_3, f_2, f_3, g_2, g_3, h_2, h_3, i, j, k\}
\end{eqnarray*}
Define a function $\varphi\colon \E^+\times \E^+ \to \E^+\times \E^+$ as follows: 
\begin{eqnarray*}
 \varphi(a_1, b_{1}) & = & (a_{2}, b_{2})\\
 \varphi(b_{2}, c_{1}) & = & (b_{3}, c_{2})\\
 \varphi(c_{2}, d_1) &=& (c_{3}, d_{2})\\
 \varphi(b_{3}, c_{3}) &=& (b_{4}, c_{4})\\
 \varphi(a_2, b_4) &=& (a_1, b_5)\\
 \varphi(b_5, c_4) &=& (b_6, c_5)\\
 \varphi(c_5, d_2) &=& (c_6, d_1)\\
 \varphi(b_6, c_6) &=& (b_1, c_1)\\
 \varphi(a_{2}, b_{3}) &=& (1, e_{2})\\
 \varphi(a_1, b_6) &=& (1, e_3)\\
 \varphi(e_{2}, c_{2}) &=& (f_{2}, g_{2})\\
 \varphi(e_{3}, c_{5}) &=& (f_{3}, g_{3})\\ 
 \varphi(e_{2}, c_{3}) &=& (f_{3}, g_{3})\\
 \varphi(e_3, c_{6}) &=& (f_{2}, g_{2})\\
 \varphi(g_2, d_1) &=& (h_2, i)\\
 \varphi(g_3, d_2) &=& (h_3, i)\\
 \varphi(f_2, h_2) &=& (j,k)\\
 \varphi(f_3, h_3) &=& (j,k)\\
 \varphi(s, 1) &= &(1,s) \mbox{ for all } s\in \E^+ 
\end{eqnarray*}
and $\varphi(s,t)=(s,t)$ if $(s,t)$ is not in the list above. 

This function is a local factorability structure in the sense of Definition \ref{Phifaktorabilitaet}. The associated rewriting system is not noetherian. Furthermore, the monoid $M$ defined by this local factorability structure is right-cancellative.
\end{Prop}

\begin{pf}
  The map $\varphi$ satisfies by definition $ \varphi(x,1) = (1,x)$ for all $x\in \E^+$ and also $\varphi^2=\varphi$.

Since the proof is quite technical and requires a lengthy case distinction with many similar steps, we will only present some cases. For a complete proof, we refer the reader to \cite{MyThesis}, Section 7.1.

Now we are going to check the fourth condition of Definition \ref{Phifaktorabilitaet}, the stability for triples condition. We will consider several cases. Note that the stability for triples condition is automatically satisfied if the triple we start with is totally stable. Moreover, we are done as soon as the triple contains a $1$; thus, we do not have to consider triples already containing $1$ and we are also immediately done with triples of the form $(a_2,b_3,t)$ or $(a_1,b_6,t)$ for all $t\in \E$. To make the steps more transparent, we will use graphical presentation. 
\renewcommand{\labelenumi}{Case \arabic{enumi}:}
\begin{enumerate}
 \item Here, we start with the triple $\mathbf{(a_1, b_{1}, t)}$ for some $t\in \E$. Observe that we are done after applying $\varphi_2$ unless $t=c_1$, so we assume this from the second step on. 
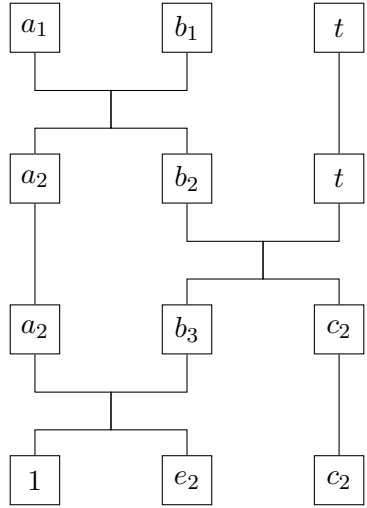
\begin{center}
     \begin{tikzpicture}[node distance = 2cm, auto]
 \node[cloud](a1){$a_1$};
 \node[cloud, right of=a1](b1){$b_{1}$};
 \node[cloud, right of=b1](c1){$t$};
 \node[cloud, below of=a1](a2){$a_2$};
 \node[cloud, below of=b1](b2){$b_2$};
 \node[cloud, below of=c1](c2){$t$};
 \node[cloud, below of=a2](a3){$a_2$};
 \node[cloud, below of=b2](b3){$b_{3}$};
 \node[cloud, below of=c2](c3){$c_{2}$};
 \node[cloud, below of=a3](a4){$1$};
 \node[cloud, below of=b3](b4){$e_2$};
 \node[cloud, below of=c3](c4){$c_{2}$};
%
\draw (c1) -- (c2); 
    \draw (a1.south) -- ++(0.0, -0.5) -| ++(1.0, -0.5) -- ++(0.0,-0.0)-- ++(-1.0, 0.)--  (a2.north);
    \draw (b1.south) -- ++(0.0,-0.5) -| ++(-1.0, -0.5) -- ++(0.0,-0.0)-- ++(1.0, 0.)--  (b2.north);
    \draw (a2) -- (a3);
    \draw (b2.south) -- ++(0.0, -0.5) -| ++(1.0, -0.5) -- ++(0.0,-0.0)-- ++(-1.0, 0.)--  (b3.north);
    \draw (c2.south) -- ++(0.0,-0.5) -| ++(-1.0, -0.5) -- ++(0.0,-0.0)-- ++(1.0, 0.)--  (c3.north);
    \draw (c3) -- (c4); 
    \draw (a3.south) -- ++(0.0, -0.5) -| ++(1.0, -0.5) -- ++(0.0,-0.0)-- ++(-1.0, 0.)--  (a4.north);
    \draw (b3.south) -- ++(0.0,-0.5) -| ++(-1.0, -0.5) -- ++(0.0,-0.0)-- ++(1.0, 0.)--  (b4.north);

 \end{tikzpicture}
    \end{center}
Thus, in this case, application of $\varphi_2\varphi_1\varphi_2$ yields a $1$. 

\item We start with the triple $\mathbf{(b_{2}, c_{1}, t)}$ for some $t\in \E$. Here, we are done after applying $\varphi_2$ unless $t=d_1$, so we assume this from the second step on. 
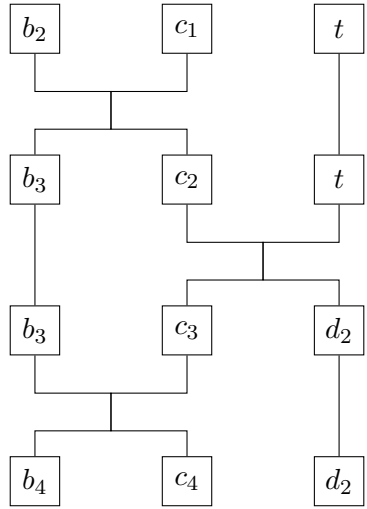
\begin{center}
     \begin{tikzpicture}[node distance = 2cm, auto]
 \node[cloud](a1){$b_{2}$};
 \node[cloud, right of=a1](b1){$c_{1}$};
 \node[cloud, right of=b1](c1){$t$};
 \node[cloud, below of=a1](a2){$b_{3}$};
 \node[cloud, below of=b1](b2){$c_{2}$};
 \node[cloud, below of=c1](c2){$t$};
 \node[cloud, below of=a2](a3){$b_{3}$};
 \node[cloud, below of=b2](b3){$c_{3}$};
 \node[cloud, below of=c2](c3){$d_{2}$};
 \node[cloud, below of=a3](a4){$b_{4}$};
 \node[cloud, below of=b3](b4){$c_{4}$};
 \node[cloud, below of=c3](c4){$d_{2}$};
%
\draw (c1) -- (c2); 
    \draw (a1.south) -- ++(0.0, -0.5) -| ++(1.0, -0.5) -- ++(0.0,-0.0)-- ++(-1.0, 0.)--  (a2.north);
    \draw (b1.south) -- ++(0.0,-0.5) -| ++(-1.0, -0.5) -- ++(0.0,-0.0)-- ++(1.0, 0.)--  (b2.north);
    \draw (a2) -- (a3);
    \draw (b2.south) -- ++(0.0, -0.5) -| ++(1.0, -0.5) -- ++(0.0,-0.0)-- ++(-1.0, 0.)--  (b3.north);
    \draw (c2.south) -- ++(0.0,-0.5) -| ++(-1.0, -0.5) -- ++(0.0,-0.0)-- ++(1.0, 0.)--  (c3.north);
    \draw (c3) -- (c4); 
    \draw (a3.south) -- ++(0.0, -0.5) -| ++(1.0, -0.5) -- ++(0.0,-0.0)-- ++(-1.0, 0.)--  (a4.north);
    \draw (b3.south) -- ++(0.0,-0.5) -| ++(-1.0, -0.5) -- ++(0.0,-0.0)-- ++(1.0, 0.)--  (b4.north);

 \end{tikzpicture}
    \end{center}

Since the pair $(c_{4}, d_{2})$ is stable, the resulting triple is totally stable.

\item We start with the triple $\mathbf{(e_{2}, c_{2}, t)}$ for some $t\in \E$. Here, we are done after applying $\varphi_2$ unless $t=d_{1}$, so we assume this from the second step on. 
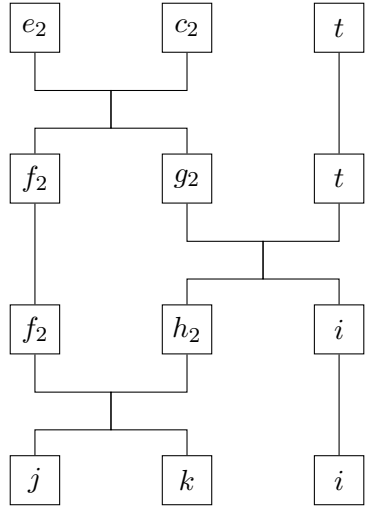
\begin{center}
     \begin{tikzpicture}[node distance = 2cm, auto]
 \node[cloud](a1){$e_{2}$};
 \node[cloud, right of=a1](b1){$c_{2}$};
 \node[cloud, right of=b1](c1){$t$};
 \node[cloud, below of=a1](a2){$f_{2}$};
 \node[cloud, below of=b1](b2){$g_{2}$};
 \node[cloud, below of=c1](c2){$t$};
 \node[cloud, below of=a2](a3){$f_{2}$};
 \node[cloud, below of=b2](b3){$h_{2}$};
 \node[cloud, below of=c2](c3){$i$};
 \node[cloud, below of=a3](a4){$j$};
 \node[cloud, below of=b3](b4){$k$};
 \node[cloud, below of=c3](c4){$i$};

\draw (c1) -- (c2); 
    \draw (a1.south) -- ++(0.0, -0.5) -| ++(1.0, -0.5) -- ++(0.0,-0.0)-- ++(-1.0, 0.)--  (a2.north);
    \draw (b1.south) -- ++(0.0,-0.5) -| ++(-1.0, -0.5) -- ++(0.0,-0.0)-- ++(1.0, 0.)--  (b2.north);
    \draw (a2) -- (a3);
    \draw (b2.south) -- ++(0.0, -0.5) -| ++(1.0, -0.5) -- ++(0.0,-0.0)-- ++(-1.0, 0.)--  (b3.north);
    \draw (c2.south) -- ++(0.0,-0.5) -| ++(-1.0, -0.5) -- ++(0.0,-0.0)-- ++(1.0, 0.)--  (c3.north);
    \draw (c3) -- (c4); 
    \draw (a3.south) -- ++(0.0, -0.5) -| ++(1.0, -0.5) -- ++(0.0,-0.0)-- ++(-1.0, 0.)--  (a4.north);
    \draw (b3.south) -- ++(0.0,-0.5) -| ++(-1.0, -0.5) -- ++(0.0,-0.0)-- ++(1.0, 0.)--  (b4.north);

 \end{tikzpicture}
    \end{center}

Since the pair $(k, i)$ is stable, the resulting triple is totally stable. 

\noindent After considering some cases where the first pair is unstable, we show some cases where the first pair is stable, so we only have to show that the application of $\varphi_2\varphi_1$ to such a triple yields either a $1$ or a totally stable triple. 

\end{enumerate}

 \renewcommand{\labelenumi}{Case \arabic{enumi}:}

 \begin{enumerate}
  \setcounter{enumi}{3}

\item We consider the triple $\mathbf{(t, a_{1}, b_{1})}$ for some $t\in \E$. Here, we are done after $\varphi_1$ since all pairs of the form $(t, a_{2})$ are stable.
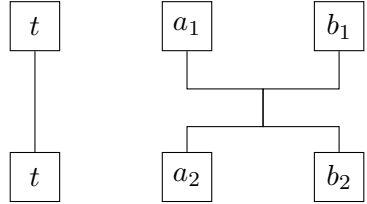
\begin{center}
     \begin{tikzpicture}[node distance = 2cm, auto]
 \node[cloud](a1){$t$};
 \node[cloud, right of=a1](b1){$a_{1}$};
 \node[cloud, right of=b1](c1){$b_{1}$};
 \node[cloud, below of=a1](a2){$t$};
 \node[cloud, below of=b1](b2){$a_{2}$};
 \node[cloud, below of=c1](c2){$b_{2}$};
 

\draw (a1) -- (a2); 
    \draw (b1.south) -- ++(0.0, -0.5) -| ++(1.0, -0.5) -- ++(0.0,-0.0)-- ++(-1.0, 0.)--  (b2.north);
    \draw (c1.south) -- ++(0.0,-0.5) -| ++(-1.0, -0.5) -- ++(0.0,-0.0)-- ++(1.0, 0.)--  (c2.north);

 \end{tikzpicture}
    \end{center}

\item We consider the triple $\mathbf{(t, c_{2}, d_{1})}$ for some $t\in \E$. Observe that by assumption $t\neq e_{2}$. Thus, we are done after applying $\varphi_1$ unless $t=b_{3}$, so we assume this in the second step. 
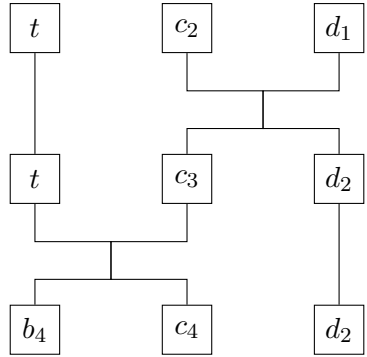
\begin{center}
     \begin{tikzpicture}[node distance = 2cm, auto]
 \node[cloud](a1){$t$};
 \node[cloud, right of=a1](b1){$c_{2}$};
 \node[cloud, right of=b1](c1){$d_{1}$};
 \node[cloud, below of=a1](a2){$t$};
 \node[cloud, below of=b1](b2){$c_{3}$};
 \node[cloud, below of=c1](c2){$d_{2}$};
 \node[cloud, below of=a2](a3){$b_{4}$};
 \node[cloud, below of=b2](b3){$c_{4}$};
 \node[cloud, below of=c2](c3){$d_{2}$};

\draw (a1) -- (a2); 
    \draw (b1.south) -- ++(0.0, -0.5) -| ++(1.0, -0.5) -- ++(0.0,-0.0)-- ++(-1.0, 0.)--  (b2.north);
    \draw (c1.south) -- ++(0.0,-0.5) -| ++(-1.0, -0.5) -- ++(0.0,-0.0)-- ++(1.0, 0.)--  (c2.north);
    \draw (c2) -- (c3);
    \draw (a2.south) -- ++(0.0, -0.5) -| ++(1.0, -0.5) -- ++(0.0,-0.0)-- ++(-1.0, 0.)--  (a3.north);
    \draw (b2.south) -- ++(0.0,-0.5) -| ++(-1.0, -0.5) -- ++(0.0,-0.0)-- ++(1.0, 0.)--  (b3.north);

 \end{tikzpicture}
    \end{center}
We are done in this case since the pair $(c_{4}, d_{2})$ is stable.

\end{enumerate}
After checking the complete list of triples which are not everywhere stable or contain $1$, the proof of the fourth condition would be complete. 

Now we are going to check the fifth condition for local factorability, the normal form condition. Recall that we have to show that the normal form of a triple remains unchanged under applying $\varphi_1$. We will use Lemma \ref{Bedingung5PhiFakt} which says we only have to check the totally unstable triples. For the given map $\varphi$, these are only the triples  $(a_{2}, b_{3}, c_{3})$, $(a_1, b_6, c_6)$, $(e_{2}, c_{2}, d_1)$ and $(e_3, c_5, d_2)$. We use Lemma \ref{NormalFormTriple} to compute normal forms for triples. 

In the first case, the application of $\varphi_1\varphi_2\varphi_1\varphi_2$ gives the following picture.

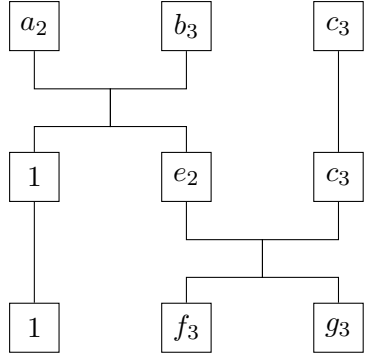
\begin{center}
     \begin{tikzpicture}[node distance = 2cm, auto]
 \node[cloud](a1){$a_{2}$};
 \node[cloud, right of=a1](b1){$b_{3}$};
 \node[cloud, right of=b1](c1){$c_{3}$};
 \node[cloud, below of=a1](a2){$1$};
 \node[cloud, below of=b1](b2){$e_{2}$};
 \node[cloud, below of=c1](c2){$c_{3}$};
  \node[cloud, below of=a2](a3){$1$};
 \node[cloud, below of=b2](b3){$f_{3}$};
 \node[cloud, below of=c2](c3){$g_{3}$};

\draw (c1) -- (c2); 
    \draw (a1.south) -- ++(0.0, -0.5) -| ++(1.0, -0.5) -- ++(0.0,-0.0)-- ++(-1.0, 0.)--  (a2.north);
    \draw (b1.south) -- ++(0.0,-0.5) -| ++(-1.0, -0.5) -- ++(0.0,-0.0)-- ++(1.0, 0.)--  (b2.north);
\draw (a2) -- (a3); 
    \draw (b2.south) -- ++(0.0, -0.5) -| ++(1.0, -0.5) -- ++(0.0,-0.0)-- ++(-1.0, 0.)--  (b3.north);
    \draw (c2.south) -- ++(0.0,-0.5) -| ++(-1.0, -0.5) -- ++(0.0,-0.0)-- ++(1.0, 0.)--  (c3.north);
 \end{tikzpicture}
    \end{center}

which is everywhere stable and thus already the (extended) normal form.

On the other hand, the application of $\varphi_1\varphi_2\varphi_1\varphi_2\varphi_1$ yields the following.

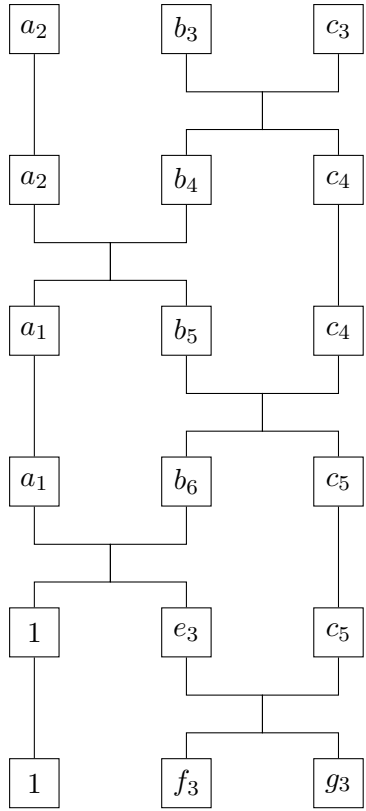
\begin{center}
     \begin{tikzpicture}[node distance = 2cm, auto]
 \node[cloud](a1){$a_{2}$};
 \node[cloud, right of=a1](b1){$b_{3}$};
 \node[cloud, right of=b1](c1){$c_{3}$};
 \node[cloud, below of=a1](a2){$a_{2}$};
 \node[cloud, below of=b1](b2){$b_{4}$};
 \node[cloud, below of=c1](c2){$c_{4}$};
  \node[cloud, below of=a2](a3){$a_{1}$};
 \node[cloud, below of=b2](b3){$b_{5}$};
 \node[cloud, below of=c2](c3){$c_{4}$};
  \node[cloud, below of=a3](a4){$a_{1}$};
 \node[cloud, below of=b3](b4){$b_{6}$};
 \node[cloud, below of=c3](c4){$c_{5}$};
  \node[cloud, below of=a4](a5){$1$};
 \node[cloud, below of=b4](b5){$e_{3}$};
 \node[cloud, below of=c4](c5){$c_{5}$};
  \node[cloud, below of=a5](a6){$1$};
 \node[cloud, below of=b5](b6){$f_{3}$};
 \node[cloud, below of=c5](c6){$g_{3}$};

\draw (a1) -- (a2); 
    \draw (b1.south) -- ++(0.0, -0.5) -| ++(1.0, -0.5) -- ++(0.0,-0.0)-- ++(-1.0, 0.)--  (b2.north);
    \draw (c1.south) -- ++(0.0,-0.5) -| ++(-1.0, -0.5) -- ++(0.0,-0.0)-- ++(1.0, 0.)--  (c2.north);
\draw (c2) -- (c3); 
    \draw (a2.south) -- ++(0.0, -0.5) -| ++(1.0, -0.5) -- ++(0.0,-0.0)-- ++(-1.0, 0.)--  (a3.north);
    \draw (b2.south) -- ++(0.0,-0.5) -| ++(-1.0, -0.5) -- ++(0.0,-0.0)-- ++(1.0, 0.)--  (b3.north);
\draw (a3) -- (a4); 
    \draw (b3.south) -- ++(0.0, -0.5) -| ++(1.0, -0.5) -- ++(0.0,-0.0)-- ++(-1.0, 0.)--  (b4.north);
    \draw (c3.south) -- ++(0.0,-0.5) -| ++(-1.0, -0.5) -- ++(0.0,-0.0)-- ++(1.0, 0.)--  (c4.north);
\draw (c4) -- (c5); 
    \draw (a4.south) -- ++(0.0, -0.5) -| ++(1.0, -0.5) -- ++(0.0,-0.0)-- ++(-1.0, 0.)--  (a5.north);
    \draw (b4.south) -- ++(0.0,-0.5) -| ++(-1.0, -0.5) -- ++(0.0,-0.0)-- ++(1.0, 0.)--  (b5.north);
\draw (a5) -- (a6); 
    \draw (b5.south) -- ++(0.0, -0.5) -| ++(1.0, -0.5) -- ++(0.0,-0.0)-- ++(-1.0, 0.)--  (b6.north);
    \draw (c5.south) -- ++(0.0,-0.5) -| ++(-1.0, -0.5) -- ++(0.0,-0.0)-- ++(1.0, 0.)--  (c6.north);
 \end{tikzpicture}
    \end{center}

So both normal forms coincide in this first case. 

The second case is very similar: The application of $\varphi_1\varphi_2\varphi_1\varphi_2$ gives the following picture.

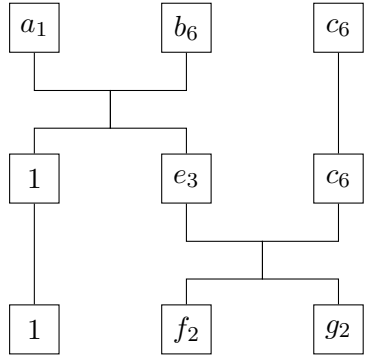
\begin{center}
     \begin{tikzpicture}[node distance = 2cm, auto]
 \node[cloud](a1){$a_{1}$};
 \node[cloud, right of=a1](b1){$b_{6}$};
 \node[cloud, right of=b1](c1){$c_{6}$};
 \node[cloud, below of=a1](a2){$1$};
 \node[cloud, below of=b1](b2){$e_{3}$};
 \node[cloud, below of=c1](c2){$c_{6}$};
  \node[cloud, below of=a2](a3){$1$};
 \node[cloud, below of=b2](b3){$f_{2}$};
 \node[cloud, below of=c2](c3){$g_{2}$};

\draw (c1) -- (c2); 
    \draw (a1.south) -- ++(0.0, -0.5) -| ++(1.0, -0.5) -- ++(0.0,-0.0)-- ++(-1.0, 0.)--  (a2.north);
    \draw (b1.south) -- ++(0.0,-0.5) -| ++(-1.0, -0.5) -- ++(0.0,-0.0)-- ++(1.0, 0.)--  (b2.north);
\draw (a2) -- (a3); 
    \draw (b2.south) -- ++(0.0, -0.5) -| ++(1.0, -0.5) -- ++(0.0,-0.0)-- ++(-1.0, 0.)--  (b3.north);
    \draw (c2.south) -- ++(0.0,-0.5) -| ++(-1.0, -0.5) -- ++(0.0,-0.0)-- ++(1.0, 0.)--  (c3.north);
 \end{tikzpicture}
    \end{center}

which is everywhere stable and thus already the (extended) normal form. 

On the other hand, the application of $\varphi_1\varphi_2\varphi_1\varphi_2\varphi_1$ yields the following.

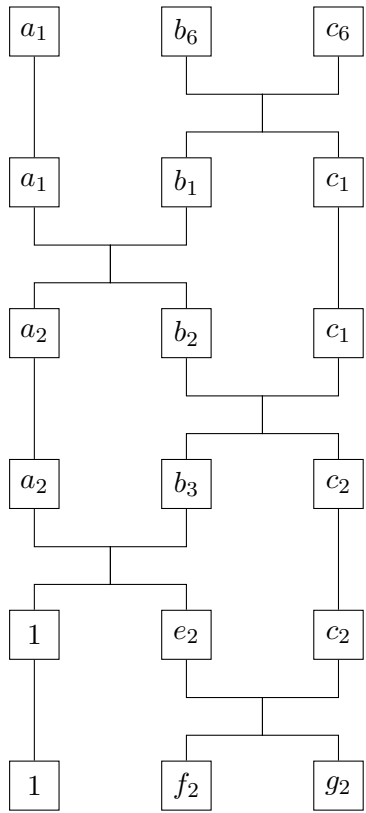
\begin{center}
     \begin{tikzpicture}[node distance = 2cm, auto]
 \node[cloud](a1){$a_{1}$};
 \node[cloud, right of=a1](b1){$b_{6}$};
 \node[cloud, right of=b1](c1){$c_{6}$};
 \node[cloud, below of=a1](a2){$a_{1}$};
 \node[cloud, below of=b1](b2){$b_{1}$};
 \node[cloud, below of=c1](c2){$c_{1}$};
  \node[cloud, below of=a2](a3){$a_{2}$};
 \node[cloud, below of=b2](b3){$b_{2}$};
 \node[cloud, below of=c2](c3){$c_{1}$};
  \node[cloud, below of=a3](a4){$a_{2}$};
 \node[cloud, below of=b3](b4){$b_{3}$};
 \node[cloud, below of=c3](c4){$c_{2}$};
  \node[cloud, below of=a4](a5){$1$};
 \node[cloud, below of=b4](b5){$e_{2}$};
 \node[cloud, below of=c4](c5){$c_{2}$};
  \node[cloud, below of=a5](a6){$1$};
 \node[cloud, below of=b5](b6){$f_{2}$};
 \node[cloud, below of=c5](c6){$g_{2}$};

\draw (a1) -- (a2); 
    \draw (b1.south) -- ++(0.0, -0.5) -| ++(1.0, -0.5) -- ++(0.0,-0.0)-- ++(-1.0, 0.)--  (b2.north);
    \draw (c1.south) -- ++(0.0,-0.5) -| ++(-1.0, -0.5) -- ++(0.0,-0.0)-- ++(1.0, 0.)--  (c2.north);
\draw (c2) -- (c3); 
    \draw (a2.south) -- ++(0.0, -0.5) -| ++(1.0, -0.5) -- ++(0.0,-0.0)-- ++(-1.0, 0.)--  (a3.north);
    \draw (b2.south) -- ++(0.0,-0.5) -| ++(-1.0, -0.5) -- ++(0.0,-0.0)-- ++(1.0, 0.)--  (b3.north);
\draw (a3) -- (a4); 
    \draw (b3.south) -- ++(0.0, -0.5) -| ++(1.0, -0.5) -- ++(0.0,-0.0)-- ++(-1.0, 0.)--  (b4.north);
    \draw (c3.south) -- ++(0.0,-0.5) -| ++(-1.0, -0.5) -- ++(0.0,-0.0)-- ++(1.0, 0.)--  (c4.north);
\draw (c4) -- (c5); 
    \draw (a4.south) -- ++(0.0, -0.5) -| ++(1.0, -0.5) -- ++(0.0,-0.0)-- ++(-1.0, 0.)--  (a5.north);
    \draw (b4.south) -- ++(0.0,-0.5) -| ++(-1.0, -0.5) -- ++(0.0,-0.0)-- ++(1.0, 0.)--  (b5.north);
\draw (a5) -- (a6); 
    \draw (b5.south) -- ++(0.0, -0.5) -| ++(1.0, -0.5) -- ++(0.0,-0.0)-- ++(-1.0, 0.)--  (b6.north);
    \draw (c5.south) -- ++(0.0,-0.5) -| ++(-1.0, -0.5) -- ++(0.0,-0.0)-- ++(1.0, 0.)--  (c6.north);
 \end{tikzpicture}
    \end{center}

So both normal forms coincide in this case.

We continue with the third case: The normal form of $(e_{2}, c_{2}, d_1)$ is obtained as follows.

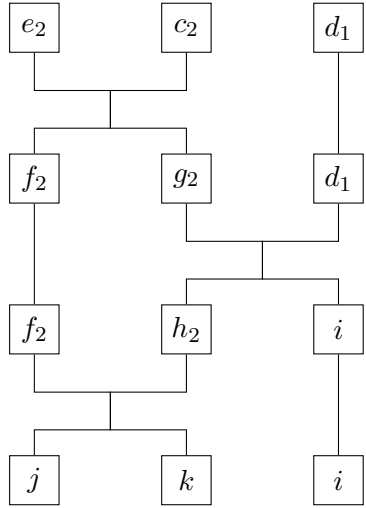
\begin{center}
     \begin{tikzpicture}[node distance = 2cm, auto]
 \node[cloud](a1){$e_{2}$};
 \node[cloud, right of=a1](b1){$c_{2}$};
 \node[cloud, right of=b1](c1){$d_{1}$};
 \node[cloud, below of=a1](a2){$f_2$};
 \node[cloud, below of=b1](b2){$g_{2}$};
 \node[cloud, below of=c1](c2){$d_{1}$};
  \node[cloud, below of=a2](a3){$f_2$};
 \node[cloud, below of=b2](b3){$h_{2}$};
 \node[cloud, below of=c2](c3){$i$};
  \node[cloud, below of=a3](a4){$j$};
 \node[cloud, below of=b3](b4){$k$};
 \node[cloud, below of=c3](c4){$i$};

\draw (c1) -- (c2); 
    \draw (a1.south) -- ++(0.0, -0.5) -| ++(1.0, -0.5) -- ++(0.0,-0.0)-- ++(-1.0, 0.)--  (a2.north);
    \draw (b1.south) -- ++(0.0,-0.5) -| ++(-1.0, -0.5) -- ++(0.0,-0.0)-- ++(1.0, 0.)--  (b2.north);
\draw (a2) -- (a3); 
    \draw (b2.south) -- ++(0.0, -0.5) -| ++(1.0, -0.5) -- ++(0.0,-0.0)-- ++(-1.0, 0.)--  (b3.north);
    \draw (c2.south) -- ++(0.0,-0.5) -| ++(-1.0, -0.5) -- ++(0.0,-0.0)-- ++(1.0, 0.)--  (c3.north);
\draw (c3) -- (c4); 
    \draw (a3.south) -- ++(0.0, -0.5) -| ++(1.0, -0.5) -- ++(0.0,-0.0)-- ++(-1.0, 0.)--  (a4.north);
    \draw (b3.south) -- ++(0.0,-0.5) -| ++(-1.0, -0.5) -- ++(0.0,-0.0)-- ++(1.0, 0.)--  (b4.north);
 \end{tikzpicture}
    \end{center}

and the tuple $(j,k,i)$ is already everywhere stable. On the other hand, the application of $\varphi_1\varphi_2\varphi_1\varphi_2\varphi_1$ yields the following.

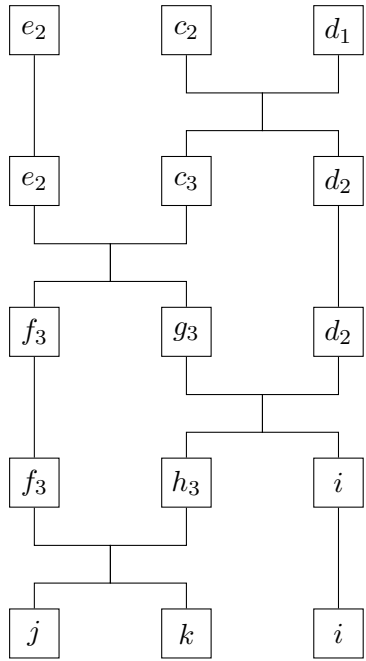
\begin{center}
     \begin{tikzpicture}[node distance = 2cm, auto]
 \node[cloud](a1){$e_{2}$};
 \node[cloud, right of=a1](b1){$c_{2}$};
 \node[cloud, right of=b1](c1){$d_{1}$};
 \node[cloud, below of=a1](a2){$e_{2}$};
 \node[cloud, below of=b1](b2){$c_{3}$};
 \node[cloud, below of=c1](c2){$d_{2}$};
  \node[cloud, below of=a2](a3){$f_{3}$};
 \node[cloud, below of=b2](b3){$g_{3}$};
 \node[cloud, below of=c2](c3){$d_{2}$};
  \node[cloud, below of=a3](a4){$f_{3}$};
 \node[cloud, below of=b3](b4){$h_{3}$};
 \node[cloud, below of=c3](c4){$i$};
  \node[cloud, below of=a4](a5){$j$};
 \node[cloud, below of=b4](b5){$k$};
 \node[cloud, below of=c4](c5){$i$};


\draw (a1) -- (a2); 
    \draw (b1.south) -- ++(0.0, -0.5) -| ++(1.0, -0.5) -- ++(0.0,-0.0)-- ++(-1.0, 0.)--  (b2.north);
    \draw (c1.south) -- ++(0.0,-0.5) -| ++(-1.0, -0.5) -- ++(0.0,-0.0)-- ++(1.0, 0.)--  (c2.north);
\draw (c2) -- (c3); 
    \draw (a2.south) -- ++(0.0, -0.5) -| ++(1.0, -0.5) -- ++(0.0,-0.0)-- ++(-1.0, 0.)--  (a3.north);
    \draw (b2.south) -- ++(0.0,-0.5) -| ++(-1.0, -0.5) -- ++(0.0,-0.0)-- ++(1.0, 0.)--  (b3.north);
\draw (a3) -- (a4); 
    \draw (b3.south) -- ++(0.0, -0.5) -| ++(1.0, -0.5) -- ++(0.0,-0.0)-- ++(-1.0, 0.)--  (b4.north);
    \draw (c3.south) -- ++(0.0,-0.5) -| ++(-1.0, -0.5) -- ++(0.0,-0.0)-- ++(1.0, 0.)--  (c4.north);
\draw (c4) -- (c5); 
    \draw (a4.south) -- ++(0.0, -0.5) -| ++(1.0, -0.5) -- ++(0.0,-0.0)-- ++(-1.0, 0.)--  (a5.north);
    \draw (b4.south) -- ++(0.0,-0.5) -| ++(-1.0, -0.5) -- ++(0.0,-0.0)-- ++(1.0, 0.)--  (b5.north);
 \end{tikzpicture}
    \end{center}
So we obtain in both cases the same normal form. The last case is again very similar: The normal form of $(e_{3}, c_{5}, d_2)$ is obtained as follows.

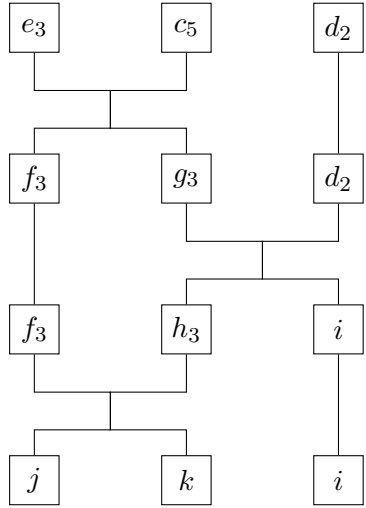
\begin{center}
     \begin{tikzpicture}[node distance = 2cm, auto]
 \node[cloud](a1){$e_{3}$};
 \node[cloud, right of=a1](b1){$c_{5}$};
 \node[cloud, right of=b1](c1){$d_{2}$};
 \node[cloud, below of=a1](a2){$f_3$};
 \node[cloud, below of=b1](b2){$g_{3}$};
 \node[cloud, below of=c1](c2){$d_{2}$};
  \node[cloud, below of=a2](a3){$f_3$};
 \node[cloud, below of=b2](b3){$h_{3}$};
 \node[cloud, below of=c2](c3){$i$};
  \node[cloud, below of=a3](a4){$j$};
 \node[cloud, below of=b3](b4){$k$};
 \node[cloud, below of=c3](c4){$i$};

\draw (c1) -- (c2); 
    \draw (a1.south) -- ++(0.0, -0.5) -| ++(1.0, -0.5) -- ++(0.0,-0.0)-- ++(-1.0, 0.)--  (a2.north);
    \draw (b1.south) -- ++(0.0,-0.5) -| ++(-1.0, -0.5) -- ++(0.0,-0.0)-- ++(1.0, 0.)--  (b2.north);
\draw (a2) -- (a3); 
    \draw (b2.south) -- ++(0.0, -0.5) -| ++(1.0, -0.5) -- ++(0.0,-0.0)-- ++(-1.0, 0.)--  (b3.north);
    \draw (c2.south) -- ++(0.0,-0.5) -| ++(-1.0, -0.5) -- ++(0.0,-0.0)-- ++(1.0, 0.)--  (c3.north);
\draw (c3) -- (c4); 
    \draw (a3.south) -- ++(0.0, -0.5) -| ++(1.0, -0.5) -- ++(0.0,-0.0)-- ++(-1.0, 0.)--  (a4.north);
    \draw (b3.south) -- ++(0.0,-0.5) -| ++(-1.0, -0.5) -- ++(0.0,-0.0)-- ++(1.0, 0.)--  (b4.north);
 \end{tikzpicture}
    \end{center}

and the tuple $(j,k,i)$ is already everywhere stable. On the other hand, the application of $\varphi_1\varphi_2\varphi_1\varphi_2\varphi_1$ yields the following.

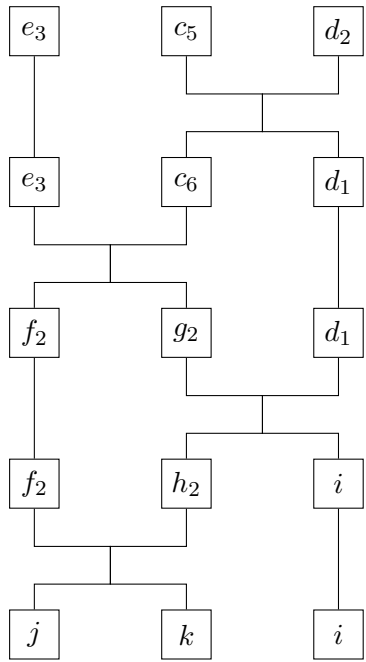
\begin{center}
     \begin{tikzpicture}[node distance = 2cm, auto]
 \node[cloud](a1){$e_{3}$};
 \node[cloud, right of=a1](b1){$c_{5}$};
 \node[cloud, right of=b1](c1){$d_{2}$};
 \node[cloud, below of=a1](a2){$e_{3}$};
 \node[cloud, below of=b1](b2){$c_{6}$};
 \node[cloud, below of=c1](c2){$d_{1}$};
  \node[cloud, below of=a2](a3){$f_{2}$};
 \node[cloud, below of=b2](b3){$g_{2}$};
 \node[cloud, below of=c2](c3){$d_{1}$};
  \node[cloud, below of=a3](a4){$f_{2}$};
 \node[cloud, below of=b3](b4){$h_{2}$};
 \node[cloud, below of=c3](c4){$i$};
  \node[cloud, below of=a4](a5){$j$};
 \node[cloud, below of=b4](b5){$k$};
 \node[cloud, below of=c4](c5){$i$};


\draw (a1) -- (a2); 
    \draw (b1.south) -- ++(0.0, -0.5) -| ++(1.0, -0.5) -- ++(0.0,-0.0)-- ++(-1.0, 0.)--  (b2.north);
    \draw (c1.south) -- ++(0.0,-0.5) -| ++(-1.0, -0.5) -- ++(0.0,-0.0)-- ++(1.0, 0.)--  (c2.north);
\draw (c2) -- (c3); 
    \draw (a2.south) -- ++(0.0, -0.5) -| ++(1.0, -0.5) -- ++(0.0,-0.0)-- ++(-1.0, 0.)--  (a3.north);
    \draw (b2.south) -- ++(0.0,-0.5) -| ++(-1.0, -0.5) -- ++(0.0,-0.0)-- ++(1.0, 0.)--  (b3.north);
\draw (a3) -- (a4); 
    \draw (b3.south) -- ++(0.0, -0.5) -| ++(1.0, -0.5) -- ++(0.0,-0.0)-- ++(-1.0, 0.)--  (b4.north);
    \draw (c3.south) -- ++(0.0,-0.5) -| ++(-1.0, -0.5) -- ++(0.0,-0.0)-- ++(1.0, 0.)--  (c4.north);
\draw (c4) -- (c5); 
    \draw (a4.south) -- ++(0.0, -0.5) -| ++(1.0, -0.5) -- ++(0.0,-0.0)-- ++(-1.0, 0.)--  (a5.north);
    \draw (b4.south) -- ++(0.0,-0.5) -| ++(-1.0, -0.5) -- ++(0.0,-0.0)-- ++(1.0, 0.)--  (b5.north);
 \end{tikzpicture}
    \end{center}

Thus, also in this case, the normal form condition is satisfied. This implies that the map $\varphi$ defined in this Lemma is indeed a local factorability structure. 

Next, observe that the associated rewriting system is not noetherian. Indeed, we have the chain of rewritings

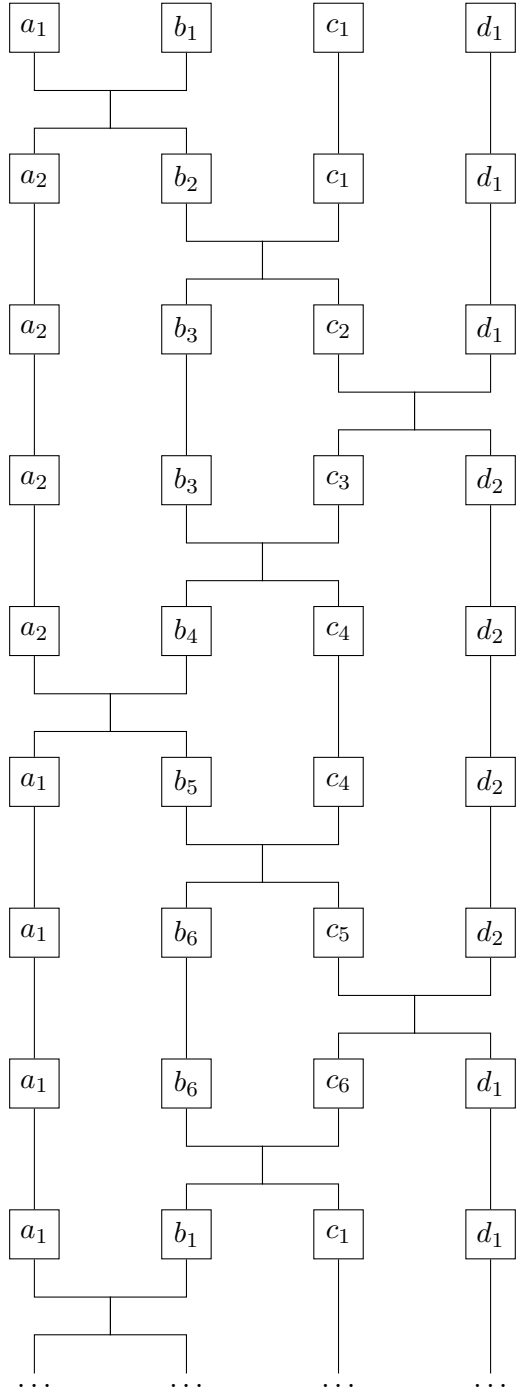
\begin{center}
     \begin{tikzpicture}[node distance = 2cm, auto]
 \node[cloud](a1){$a_1$};
 \node[cloud, right of=a1](b1){$b_1$};
 \node[cloud, right of=b1](c1){$c_{1}$};
  \node[cloud, right of=c1](d1){$d_{1}$};
  
 \node[cloud, below of=a1](a2){$a_{2}$};
 \node[cloud, below of=b1](b2){$b_{2}$};
  \node[cloud, below of=c1](c2){$c_{1}$};
   \node[cloud, below of=d1](d2){$d_{1}$};
   
  \node[cloud, below of=a2](a3){$a_{2}$};
 \node[cloud, below of=b2](b3){$b_{3}$};
 \node[cloud, below of=c2](c3){$c_{2}$};
  \node[cloud, below of=d2](d3){$d_{1}$};
 
  \node[cloud, below of=a3](a4){$a_{2}$};
 \node[cloud, below of=b3](b4){$b_3$};
 \node[cloud, below of=c3](c4){$c_3$};
  \node[cloud, below of=d3](d4){$d_2$};
  
  \node[cloud, below of=a4](a7){$a_2$};
 \node[cloud, below of=b4](b7){$b_4$};
 \node[cloud, below of=c4](c7){$c_4$};
 \node[cloud, below of=d4](d7){$d_2$};

    \node[cloud, below of=a7](a8){$a_{1}$};
 \node[cloud, below of=b7](b8){$b_{5}$};
 \node[cloud, below of=c7](c8){$c_{4}$};
 \node[cloud, below of=d7](d8){$d_{2}$};
 
     \node[cloud, below of=a8](a9){$a_{1}$};
 \node[cloud, below of=b8](b9){$b_{6}$};
 \node[cloud, below of=c8](c9){$c_{5}$};
 \node[cloud, below of=d8](d9){$d_{2}$};
 
      \node[cloud, below of=a9](a10){$a_{1}$};
 \node[cloud, below of=b9](b10){$b_{6}$};
 \node[cloud, below of=c9](c10){$c_{6}$};
 \node[cloud, below of=d9](d10){$d_{1}$};
 
       \node[cloud, below of=a10](a11){$a_{1}$};
 \node[cloud, below of=b10](b11){$b_{1}$};
 \node[cloud, below of=c10](c11){$c_{1}$};
 \node[cloud, below of=d10](d11){$d_{1}$};
 
    \node[below of=a11](a12){$\ldots$};
 \node[below of=b11](b12){$\ldots$};
 \node[below of=c11](c12){$\ldots$};
 \node[below of=d11](d12){$\ldots$};

\draw (c1) -- (c2); 
\draw (d1) -- (d2); 
    \draw (a1.south) -- ++(0.0, -0.5) -| ++(1.0, -0.5) -- ++(0.0,-0.0)-- ++(-1.0, 0.)--  (a2.north);
    \draw (b1.south) -- ++(0.0,-0.5) -| ++(-1.0, -0.5) -- ++(0.0,-0.0)-- ++(1.0, 0.)--  (b2.north);
\draw (d2) -- (d3); 
\draw (a2) -- (a3); 
    \draw (b2.south) -- ++(0.0, -0.5) -| ++(1.0, -0.5) -- ++(0.0,-0.0)-- ++(-1.0, 0.)--  (b3.north);
    \draw (c2.south) -- ++(0.0,-0.5) -| ++(-1.0, -0.5) -- ++(0.0,-0.0)-- ++(1.0, 0.)--  (c3.north);
\draw (a3) -- (a4); 
\draw (b3) -- (b4);
    \draw (c3.south) -- ++(0.0, -0.5) -| ++(1.0, -0.5) -- ++(0.0,-0.0)-- ++(-1.0, 0.)--  (c4.north);
    \draw (d3.south) -- ++(0.0,-0.5) -| ++(-1.0, -0.5) -- ++(0.0,-0.0)-- ++(1.0, 0.)--  (d4.north);
\draw (a4) -- (a7); 
\draw (d4) -- (d7); 
    \draw (b4.south) -- ++(0.0, -0.5) -| ++(1.0, -0.5) -- ++(0.0,-0.0)-- ++(-1.0, 0.)--  (b7.north);
    \draw (c4.south) -- ++(0.0,-0.5) -| ++(-1.0, -0.5) -- ++(0.0,-0.0)-- ++(1.0, 0.)--  (c7.north);
    
\draw (c7) -- (c8); 
\draw (d7) -- (d8); 
    \draw (a7.south) -- ++(0.0, -0.5) -| ++(1.0, -0.5) -- ++(0.0,-0.0)-- ++(-1.0, 0.)--  (a8.north);
    \draw (b7.south) -- ++(0.0,-0.5) -| ++(-1.0, -0.5) -- ++(0.0,-0.0)-- ++(1.0, 0.)--  (b8.north);
\draw (d8) -- (d9); 
\draw (a8) -- (a9); 
    \draw (b8.south) -- ++(0.0, -0.5) -| ++(1.0, -0.5) -- ++(0.0,-0.0)-- ++(-1.0, 0.)--  (b9.north);
    \draw (c8.south) -- ++(0.0,-0.5) -| ++(-1.0, -0.5) -- ++(0.0,-0.0)-- ++(1.0, 0.)--  (c9.north);
\draw (a9) -- (a10); 
\draw (b9) -- (b10);
    \draw (c9.south) -- ++(0.0, -0.5) -| ++(1.0, -0.5) -- ++(0.0,-0.0)-- ++(-1.0, 0.)--  (c10.north);
    \draw (d9.south) -- ++(0.0,-0.5) -| ++(-1.0, -0.5) -- ++(0.0,-0.0)-- ++(1.0, 0.)--  (d10.north);
\draw (a10) -- (a11); 
\draw (d10) -- (d11); 
    \draw (b10.south) -- ++(0.0, -0.5) -| ++(1.0, -0.5) -- ++(0.0,-0.0)-- ++(-1.0, 0.)--  (b11.north);
    \draw (c10.south) -- ++(0.0,-0.5) -| ++(-1.0, -0.5) -- ++(0.0,-0.0)-- ++(1.0, 0.)--  (c11.north);
    
\draw (c11) -- (c12); 
\draw (d11) -- (d12); 
    \draw (a11.south) -- ++(0.0, -0.5) -| ++(1.0, -0.5) -- ++(0.0,-0.0)-- ++(-1.0, 0.)--  (a12.north);
    \draw (b11.south) -- ++(0.0,-0.5) -| ++(-1.0, -0.5) -- ++(0.0,-0.0)-- ++(1.0, 0.)--  (b12.north);
 \end{tikzpicture}
    \end{center}

So we obtain a cycle in the rewritings. Observe that this cycle is non-trivial: For example, note that if $a_1=a_2$, there had to be a zigzag of rewritings from $a_1$ to $a_2$. But none of the non-trivial rewriting rules starts or ends with $a_1$, so this is impossible. 

Last, we are going to show that the monoid $M$ defined by the local factorability structure as above is right-cancellative. We are going to rely on the $\varphi$-normal forms in this monoid. Assume $M$ is not right-cancellative. Then there are some elements $x,y,z\in M$ such that $xz=yz$, but $x\neq y$. Obviously, $z\neq 1$. We want to consider an example with minimal $\E$-word length of $z$. Then $N_{\E}(z)=1$: Indeed, otherwise there is an $s\in \E$ and $w\in M\setminus\{1\}$ such that $z=ws$. Then $(xw)s=(yw)s$, so that either $xw=yw$ contradicting the minimality of $z$, or $xw\neq yw$, which is again contradicting the minimality. 

So we know that there are $x\neq y\in M$ and $z\in \E$ such that $xz=yz$. Let $(x_m,\ldots, x_1)$ be the normal form of $x$ and $(y_n, \ldots, y_1)$ the normal form of $y$. We may choose an example where $m+n$ is minimal. First, we want to demonstrate that $m+n\geq 3$. Indeed, $m+n$ has to be at least $1$ by definition. Then, if $m=1$ and $n=0$ (the other case can be treated symmetrically), we have $x_1z=z$, which in particular implies $\varphi(x_1,z)=(1,z)$. But there is no pair $(x_1,z)$ with $x_1\neq 1$ and image $(1,z)$ under $\varphi$, so this cannot happen. More generally, we can exclude the case $n=0$ (and so symmetrically $m=0$): We compute the normal form of $(x_m, \ldots, x_1, z)$ using the definition. We already know that this normal form has to be $z$. Note that in the step where the number of non-trivial letters in the string reduces to one, we have to get a $1$ out of a pair of elements of $\E$ by applying $\varphi$, so that $z$ must be $e_i$ with $i\in \{2,3\}$. But all pairs of the form $(x_1,
 e_i)$ 
are stable, so that $(x_m, \ldots, x_1, e_i)$ has to be already the normal form, contradicting the assumption $m+n>0$. Furthermore, we can exclude the case $m=n=1$, observing that in the definition of $\varphi$, there are no distinct pairs $(x_1, z)$, $(y_1,z)$ which are mapped to the same pair by $\varphi$.

Now we know $m+n\geq 3$. In order to compute the normal forms of tuples $(x_m, \ldots, x_1, z)$ and $(y_n,\ldots, y_1, z)$, we first have to apply $\varphi_m\varphi_{m-1}\ldots \varphi_1$ or $\varphi_n\varphi_{n-1}\ldots \varphi_1$, respectively.  Observe that if both $(x_1,z)$ and $(y_1, z)$ are stable pairs, then we would have two different normal forms for the same element $xz=yz$ of $M$, yielding a contradiction. So we may assume $\varphi(x_1,z)=(u_1, v_1)$ with $(x_1, z)\neq (u_1,v_1)$. Note that this in particular implies $z\neq v_1$ by the definition of $\varphi$. 

Consider now the case where applications of both $\varphi_m\varphi_{m-1}\ldots \varphi_1$ and $\varphi_n\varphi_{n-1}\ldots \varphi_1$ do not produce a $1$. 
Then the results have to be equal, in particular, $n=m$. Furthermore, this implies $\varphi(x_1, z)=(u_1,v_1)$ and $\varphi(y_1, z)=(t_1, v_1)$ with the same $v_1$, in particular, $(y_1,z)\neq (t_1,v_1)$. There are no two distinct pairs with same right letters in the list defining $\varphi$, i.e., if $\varphi(\alpha, \beta)=(\gamma, \delta)$ and $\varphi(\widetilde{\alpha}, \beta)=(\widetilde{\gamma}, \delta)$ and $\beta \neq \delta$, we know that $\widetilde{\alpha}=\alpha$ and $\widetilde{\gamma}=\gamma$. Hence, we may conclude that $x_1=y_1$ and $u_1=t_1$. Now since $x\neq y$, we know that $x_{m}\ldots x_2\neq y_n\ldots y_2$. The normal forms of $(x_m, \ldots, x_2, u_1)$ and $(y_n, \ldots, y_2, u_1)$ have to coincide since they are the same as the normal form of $xz=yz$ with the right-most letter $v_1$ deleted, and so the elements $x_m\ldots x_2u_1$ and $y_n\ldots y_2u_1$ have to coincide. But this contradicts the minimality assumption on $n+m$. 

So we have to consider the case where we obtain a $1$ while building the normal form. We may assume that application of $\varphi_p$ to 
\begin{eqnarray*}
 \varphi_{p-1}\ldots \varphi_1(x_n, x_{n-1}, \ldots, x_1,z)
\end{eqnarray*}
 produces the first $1$. Define $u_q$ and $v_q$ inductively via $u_0:=z$ and $\varphi(x_{q+1}, u_q)=(u_{q+1}, v_{q+1})$ for $0\leq q\leq p-1$. By assumption, $u_p=1$. (We illustrate the situation by the following picture.)
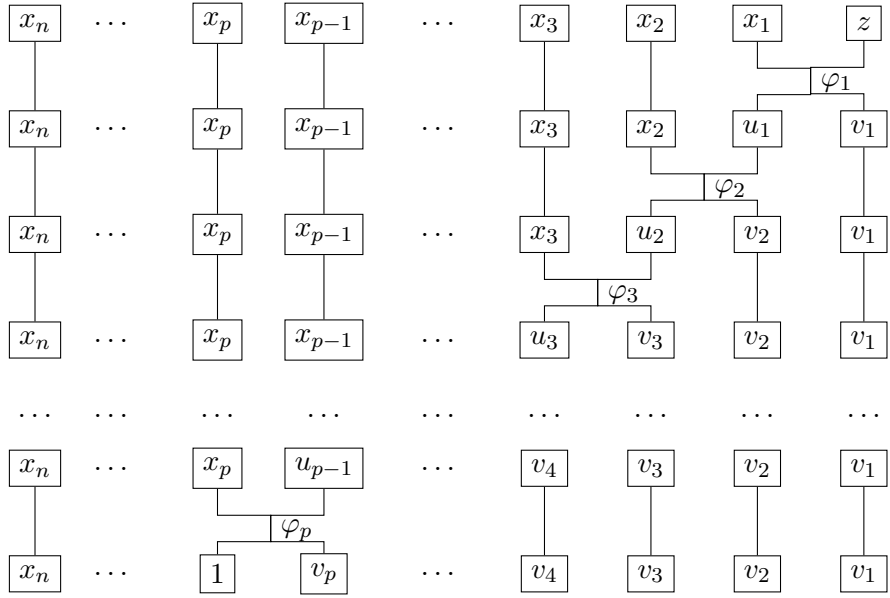
\begin{center}
\tikzstyle{cloud} = [draw, rectangle, node distance=1.4cm, minimum size=4.55mm]
 \begin{tikzpicture}[scale=0.7]
   \node[cloud](a1){$x_n$};
 \node[right of=a1](b1){$\ldots$};
 \node[cloud, right of=b1](c1){$x_{p}$};
  \node[cloud, right of=c1](d1){$x_{p-1}$};
 \node[right of=d1, xshift=0.5cm](g1){$\ldots$};
 \node[cloud, right of=g1](h1){$x_3$};
  \node[cloud, right of=h1](i1){$x_2$};
 \node[cloud, right of=i1](j1){$x_1$};
 \node[cloud, right of=j1](k1){$z$};
  \node[cloud, below of=a1](a2){$x_n$};
 \node[right of=a2](b2){$\ldots$};
 \node[cloud, right of=b2](c2){$x_{p}$};
  \node[cloud, right of=c2](d2){$x_{p-1}$};
 \node[right of=d2, xshift=0.5cm](g2){$\ldots$};
 \node[cloud, right of=g2](h2){$x_3$};
  \node[cloud, right of=h2](i2){$x_2$};
 \node[cloud, right of=i2](j2){$u_1$};
 \node[cloud, right of=j2](k2){$v_1$};

  \node[cloud, below of=a2](a3){$x_n$};
 \node[right of=a3](b3){$\ldots$};
 \node[cloud, right of=b3](c3){$x_{p}$};
  \node[cloud, right of=c3](d3){$x_{p-1}$};
 \node[right of=d3, xshift=0.5cm](g3){$\ldots$};
 \node[cloud, right of=g3](h3){$x_3$};
  \node[cloud, right of=h3](i3){$u_2$};
 \node[cloud, right of=i3](j3){$v_2$};
 \node[cloud, right of=j3](k3){$v_1$};

  \node[cloud, below of=a3](a4){$x_n$};
 \node[right of=a4](b4){$\ldots$};
 \node[cloud, right of=b4](c4){$x_{p}$};
  \node[cloud, right of=c4](d4){$x_{p-1}$};
 \node[right of=d4, xshift=0.5cm](g4){$\ldots$};
 \node[cloud, right of=g4](h4){$u_3$};
  \node[cloud, right of=h4](i4){$v_3$};
 \node[cloud, right of=i4](j4){$v_2$};
 \node[cloud, right of=j4](k4){$v_1$};

  \node[ below of=a4](a5){$\ldots$};
 \node[below of=b4](b5){$\ldots$};
 \node[ below of=c4](c5){$\ldots$};
  \node[ below of=d4](d5){$\ldots$};
 \node[below of=g4](g5){$\ldots$};
 \node[below of=h4](h5){$\ldots$};
  \node[below of=i4](i5){$\ldots$};
 \node[ below of=j4](j5){$\ldots$};
 \node[below of=k4](k5){$\ldots$};

  \node[cloud, below of=a5, yshift=0.7cm](a6){$x_n$};
 \node[right of=a6](b6){$\ldots$};
 \node[cloud, right of=b6](c6){$x_{p}$};
  \node[cloud, right of=c6](d6){$u_{p-1}$};
 \node[right of=d6, xshift=0.5cm](g6){$\ldots$};
 \node[cloud, right of=g6](h6){$v_4$};
  \node[cloud, right of=h6](i6){$v_3$};
 \node[cloud, right of=i6](j6){$v_2$};
 \node[cloud, right of=j6](k6){$v_1$};

  \node[cloud, below of=a6](a7){$x_n$};
 \node[right of=a7](b7){$\ldots$};
 \node[cloud, right of=b7](c7){$1$};
  \node[cloud, right of=c7](d7){$v_{p}$};
 \node[right of=d7, xshift=0.5cm](g7){$\ldots$};
 \node[cloud, right of=g7](h7){$v_4$};
  \node[cloud, right of=h7](i7){$v_3$};
 \node[cloud, right of=i7](j7){$v_2$};
 \node[cloud, right of=j7](k7){$v_1$};

 \draw (a1)--(a2);
  \draw (c1)--(c2);
  \draw (d1)--(d2);
  \draw (h1)--(h2);
  \draw (i1)--(i2);
\draw (j1.south) -- ++(0.0, -0.5) -| ++(1.0, -0.5) -- ++(0.0,-0.0)-- ++(-1.0, 0.)--  (j2.north);
    \draw (k1.south) -- ++(0.0,-0.5) -| ++(-1.0, -0.5) --node[anchor=west, yshift=0.15cm]{$\varphi_1$} ++(0.0,-0.0)-- ++(1.0, 0.)--  (k2.north);
 
\draw (a2)--(a3);
  \draw (c2)--(c3);
  \draw (d2)--(d3);
  \draw (h2)--(h3);
  \draw (k2)--(k3);
\draw (i2.south) -- ++(0.0, -0.5) -| ++(1.0, -0.5) -- ++(0.0,-0.0)-- ++(-1.0, 0.)--  (i3.north);
    \draw (j2.south) -- ++(0.0,-0.5) -| ++(-1.0, -0.5) --node[anchor=west, yshift=0.15cm]{$\varphi_2$} ++(0.0,-0.0)-- ++(1.0, 0.)--  (j3.north);

\draw (a3)--(a4);
  \draw (c3)--(c4);
  \draw (d3)--(d4);
  \draw (j3)--(j4);
  \draw (k3)--(k4);
\draw (h3.south) -- ++(0.0, -0.5) -| ++(1.0, -0.5) -- ++(0.0,-0.0)-- ++(-1.0, 0.)--  (h4.north);
    \draw (i3.south) -- ++(0.0,-0.5) -| ++(-1.0, -0.5) --node[anchor=west, yshift=0.15cm]{$\varphi_3$} ++(0.0,-0.0)-- ++(1.0, 0.)--  (i4.north);

\draw (a6)--(a7);
  \draw (j6)--(j7);
  \draw (k6)--(k7);
  \draw (h6)--(h7);
  \draw (i6)--(i7);
\draw (c6.south) -- ++(0.0, -0.5) -| ++(1.0, -0.5) -- ++(0.0,-0.0)-- ++(-1.0, 0.)--  (c7.north);
    \draw (d6.south) -- ++(0.0,-0.5) -| ++(-1.0, -0.5) --node[anchor=west, yshift=0.15cm]{$\varphi_p$} ++(0.0,-0.0)-- ++(1.0, 0.)--  (d7.north);
 \end{tikzpicture}

\end{center}
Observe that by definition of the normal form, we have 
\begin{eqnarray*}
 \NF(x_{p-1}, \ldots x_1,z)=\varphi_{p-1}\ldots\varphi_1(x_{p-1}, \ldots x_1,z)=(u_{p-1}, v_{p-1},\ldots, v_1).
\end{eqnarray*}
In particular, this tuple is totally stable. 

Now since $\varphi(x_p, u_{p-1})=(1,v_p)$ and we know that $x_p\neq 1$ and $u_{p-1}\neq 1$, we conclude that $v_p=e_r$ with $r\in\{2,3\}$. Thus, the pair $(x_p, u_{p-1})$ equals $(a_{4-r}, b_{3r-3})$. Furthermore, observe that there is no unstable pair which ends with $e_r$ on the right, so the tuple $(x_{m}, \ldots, x_{p+1}, e_r)$ is totally stable. 

If we assume $p=1$, then $z=b_{3r-3}$ and $x_1=a_{4-r}$ and $(x_m, \ldots, x_2, e_r)$ is the normal form of $xz=yz$. So in particular, this is the normal form of the tuple $(y_n,\ldots, y_1, b_{3r-3})$. In particular, since $b_{3r-3}\neq e_r$, the tuple $(y_1, b_{3r-3})$ has to be unstable, so that $y_1=a_{4-r}$ follows. This now implies that $y_n\ldots y_2e_r=x_m\ldots x_2 e_r$, while $y_n\ldots y_2\neq x_m\ldots x_2$, contradicting the minimality of $n+m$. 

So we have shown that $p\geq 2$. To simplify the notation, we consider first the case $r=2$. Here, we know that $x_{p}=a_2$ and $u_{p-1}=b_3$. Since the pair $(x_p, x_{p-1})$ was stable, we know that $x_{p-1}$ is neither $b_3$ nor $b_4$. In particular, the pair $(x_{p-1}, u_{p-2})$ cannot be stable since $u_{p-1}\neq x_{p-1}$. Thus, $(x_{p-1}, u_{p-2})$ has to be an unstable pair such that $\varphi(x_{p-1}, u_{p-2})=(b_3, v_{p-1})$. This implies $x_{p-1}=b_2$, $u_{p-2}=c_1$ and $v_{p-1}=c_2$.

If $p=2$, then $z=c_1$ and we have 
\begin{eqnarray*}
 \varphi_m\varphi_{m-1}\ldots \varphi_1(x_m, \ldots, x_1, z) &=&  \varphi_m\varphi_{m-1}\ldots \varphi_1(x_m, \ldots, x_3, a_2, b_2, c_1)\\
&=& \varphi_m\varphi_{m-1}\ldots \varphi_2(x_m, \ldots, x_3, a_2, b_3, c_2)\\
&=& \varphi_m\varphi_{m-1}\ldots \varphi_3(x_m, \ldots, x_3, 1, e_2, c_2)\\
&=& ( 1, x_m, \ldots, x_3, e_2, c_2).
\end{eqnarray*}
Computing the normal form of this tuple is done after applying $\varphi_1$ since then we obtain the tuple $(x_m, \ldots, x_3, f_2, g_2)$ and all pairs of the form $(x_3, f_2)$ are stable. So we know that $(y_n, \ldots, y_1, c_1)$ is not in the normal form. In particular, the pair $(y_1, c_1)$ is unstable and so $y_1=b_2$, so that $\varphi_n\ldots \varphi_1$ yields a $c_2$ in the right-most place when applied to $(y_n, \ldots, y_1, c_1)$. Since the normal form of this tuple ends with $g_2$ on the right, we know that applying $\varphi_n\ldots \varphi_1$ must have produced a $1$. If $y_2$ would not be equal to $a_2$, the triple $(y_2, \varphi(y_1, c_1))=(y_2, b_3, c_2)$ would be stable and thus the whole tuple $(y_n, \ldots, y_2, b_3, c_2)$, which contradicts our assumptions. So $y_2=a_2$, and thus we have the equality
\begin{eqnarray*}
 (y_n\ldots y_3e_2)c_2 &=& (y_n\ldots y_3a_2b_3)c_2 = (y_n\ldots y_3a_2)b_3c_2 =y_n\ldots y_3y_2b_2c_1\\
&=& y_n\ldots y_3y_2y_1z=x_m\ldots x_1z=(x_m\ldots x_3e_2)c_2.
\end{eqnarray*}
Now since $x_1=y_1$ and $x_2=y_2$ and $x\neq y$, we conclude that $x_m\ldots x_3\neq y_n\ldots y_3$. Then either $(y_n\ldots y_3)e_2=(x_m\ldots x_3)e_2$ and this produces a counterexample to right-cancellativity contradicting the minimality of $n+m$, or $(y_n\ldots y_3)e_2\neq(x_m\ldots x_3)e_2$, then we have a contradictory counterexample due to 
\begin{eqnarray*}
 (y_n\ldots y_3e_2)c_2 =(x_m\ldots x_3e_2)c_2.
\end{eqnarray*}

So we know that $p\geq 3$. Recall that $x_{p}=a_2$, $u_{p-1}=b_3$, $x_{p-1}=b_2$, $u_{p-2}=c_1$ and $v_{p-1}=c_2$. Now the pair $(x_{p-1}, x_{p-2})$ is stable, thus $x_{p-2}$ is not $c_1$. As before, this implies that the pair $(x_{p-2}, u_{p-3})$ is unstable and is mapped via $\varphi$ to $(c_1, v_{p-2})$. However, this is impossible. This completes the proof for the case $r=2$.

The proof in case $r=3$ is completely analogous. This can be also seen as follows: There is a map $\gamma\colon M\to M$, defined below by its values on generators, which is a monoid homomorphism and involution and which maps $e_2$ to $e_3$ and preserves $\varphi$. This allows to avoid the symmetrical argument. The map $\gamma$ is given as follows.

\begin{align*}
 a_1 &\mapsto  a_2\\
 a_2  &\mapsto  a_1\\
 b_1  &\mapsto  b_4\\
 b_2  &\mapsto  b_5\\
 b_3  &\mapsto  b_6\\
 b_4  &\mapsto  b_1\\
 b_5  &\mapsto  b_2\\
 b_6  &\mapsto b_3\\
 c_1  &\mapsto  c_4\\
 c_2  &\mapsto  c_5\\
 c_3  &\mapsto  c_6\\
 c_4  &\mapsto  c_1\\
 c_5  &\mapsto  c_2\\
 c_6  &\mapsto  c_3
 \end{align*}
 
 \begin{align*}
 d_1  &\mapsto d_2\\
 d_2  &\mapsto  d_1\\
 e_2  &\mapsto  e_3\\
 e_3  &\mapsto e_2\\
 f_2  &\mapsto  f_3\\
 f_3  &\mapsto  f_2\\
 g_2  &\mapsto  g_3\\
 g_3  &\mapsto  g_2\\
 h_2  &\mapsto  h_3\\
 h_3  &\mapsto  h_2\\
  i   &\mapsto  i\\
  j   &\mapsto  j\\
  k   &\mapsto  k
\end{align*}

This completes the proof. 
\end{pf}

\end{appendix}

\bibliographystyle{plain}
\bibliography{../UniKram/Garside}
\end{document}